\documentclass[english]{article}

\usepackage{amsthm}
\usepackage{sidecap}
\usepackage{geometry}
\usepackage{amsmath}
\usepackage{amsthm}
\usepackage{amssymb}
\usepackage{enumitem}
\usepackage{tabularx}
\usepackage[colorlinks=true, linkcolor=blue, citecolor=blue, urlcolor=blue]{hyperref}
\usepackage{bm}
\usepackage{tabu}

\usepackage{blkarray}

\DeclareMathOperator{\arccosh}{arccosh}

\usepackage[noadjust]{cite}

\usepackage{todonotes}
\usepackage{tabularx}
\usepackage{amsmath}
\usepackage{amssymb}
\usepackage{color}
\usepackage{enumitem}

\makeatletter

\newtheorem{theorem}{Theorem}
\newtheorem*{theorem*}{Theorem}
\newtheorem{lemma}{Lemma}
\newtheorem{example}{Example}
\newtheorem{proposition}{Proposition}

\newtheorem{remark}{Remark}
\newtheorem{definition}{Definition}

\newtheorem*{notation*}{Notation}
\newtheorem*{terminology*}{Terminology}

\DeclareMathOperator{\prox}{prox}
\DeclareMathOperator{\argmin}{argmin}

\newcommand{\bd}{{\operatorname{bdry}}}
\newcommand{\inter}{{\operatorname{int}}}

\newcommand{\Poincare}{Poincar\'e }

\newcommand{\grad}{\nabla}

\newcommand{\ph}{\partial^{\mathsf{h}} }

\DeclareMathOperator{\epi}{epi}

\newcommand{\calM}{M }

\newcommand{\reals}{\mathbb{R}}

\usepackage{lastpage}

\begin{document}
	
	\title{\Large Horospherically Convex Optimization on Hadamard Manifolds\\Part I: Analysis and Algorithms}

	\author{Christopher Criscitiello\thanks{Institute of Mathematics, EPFL, \texttt{christopher.criscitiello@epfl.ch}.  Equal contribution.}
		\and 
		Jungbin Kim\thanks{Department of Industrial Engineering, Seoul National University, \texttt{kjb2952@snu.ac.kr}.  Equal contribution.}
	}
	\date{May 22, 2025}
	
	\maketitle
	
	\begin{abstract}
		
		Geodesic convexity (g-convexity) is a natural generalization of convexity to Riemannian manifolds. However, g-convexity lacks many desirable properties satisfied by Euclidean convexity.  For instance, the natural notions of half-spaces and affine functions are themselves not g-convex.
		Moreover, recent studies have shown that the oracle complexity of geodesically convex optimization necessarily depends on the curvature of the manifold~\cite{hamilton2021nogo,criscitiello2022negative,criscitiello2023curvature}, a computational bottleneck for several problems, e.g., tensor scaling.
		Recently, Lewis \emph{et al.}~\cite{lewis2024horoballs} addressed this challenge by proving curvature-independent convergence of subgradient descent, assuming horospherical convexity of the objective's sublevel sets.
		
		{Using a similar idea, we introduce a generalization of convex functions to Hadamard manifolds, utilizing horoballs and Busemann functions as building blocks (as proxies for half-spaces and affine functions). 
			We refer to this new notion as \emph{horospherical convexity} (h-convexity).}
		We provide algorithms for both nonsmooth and smooth h-convex optimization, which have curvature-independent guarantees \emph{exactly} matching those from Euclidean space; this includes generalizations of subgradient descent and Nesterov's accelerated method.
		Motivated by applications, we extend these algorithms and their convergence rates to minimizing a \emph{sum} of horospherically convex functions, assuming access to a weighted-Fr\'echet-mean oracle.
	\end{abstract}

	\setcounter{tocdepth}{2} 
	{\small 
		\tableofcontents
	}
	
	\newpage
	
	\section{Introduction}
	\label{sec:intro}
	
	We consider the optimization problem:  
	\begin{equation}
		\label{eq:problem-intro}
		\min_{x\in M}\; f(x),
	\end{equation}
	where \( M \) is a Hadamard manifold, i.e., a complete, simply connected Riemannian manifold with non-positive sectional curvature.  This includes Euclidean spaces $M = \mathbb{R}^n$, hyperbolic spaces $M = \mathbb{H}^n$, and the set of positive definite matrices with affine-invariant metric (the Fisher-Rao metric for Gaussian covariance matrices)~\cite{bridson2013metric,skovgaard1984riemgeogaussians}.
	For a comprehensive introduction to optimization on manifolds, see \cite{absil2008optimization, boumal2020intromanifolds}.
	
	Since problem~\eqref{eq:problem-intro} is NP-hard in general (see, for example, \cite{sahni1974computationally,murty1987some}), additional assumptions are required to make it tractable. In the Euclidean setting $M = \mathbb{R}^n$, the classical assumption is that the objective function \( f \) is convex on \( \mathbb{R}^n \). For general Hadamard manifolds, the analogous assumption is \emph{geodesic convexity} (g-convexity), which requires that \( f \) is convex along every geodesic segment (see, for example, \cite{udriste1994convex,zhang2016first}).
	A common application of g-convex optimization is computing intrinsic means or medians, e.g., for phylogenetics and computational anatomy~\cite{karcher1977riemannian,YuanHAG20,fletcher2011horoball,bacak2014convex}. 
	Scaling problems (including matrix, operator and tensor scaling) comprise a broad class of problems that can also be formulated as g-convex optimization~\cite{burgissernoncommutativeoptimization}.  Applications of scaling problems include robust covariance estimation and subspace recovery~\cite{auderset2005angular,wiesel2012gconvexity,zhang2012gconvexity,wiesel2015gconvexity,sra2015conicgeometricoptimspd,ciobotaru2018geometrical,franksmoitra2020}, estimation for matrix normal models~\cite{tang2021integrated,amendola2021,franks2021neartyler}, computing Brascamp--Lieb constants~\cite{sravishnoibracsampliebconstant}, Horn's problem~\cite{franksprescribedmarginals2018,burgissernoncommutativeoptimization,kapovich2009convex}, and the quantum marginals problem~\cite{Walter2013,tensorscaling2018}.
	
	Notably, algorithms for g-convex optimization led to the first polynomial-time algorithms for operator scaling~\cite{Garg2020Operator,allenzhuoperatorsplitting}; yet, these algorithms have so far fallen short of providing fully polynomial-time algorithms for many other scaling problems, e.g., tensor scaling and Horn's problem (the latter fits into our framework, see Section~\ref{sec:prob_example}).\footnote{Significant progress has nonetheless been made on these problems~\cite{franksprescribedmarginals2018,tensorscaling2018}.}
	A key obstruction is that the query complexity of existing algorithms\footnote{Including cutting-plane methods~\cite{rusciano2019riemannian,ellipsoidcriscitiello23b}, interior-point methods based on barriers for the ball~\cite{hnw2023}, and first-order methods~\cite{zhang2016first,zhang2018estimate,ahn2020nesterov,alimisis2021momentum,kim2022accelerated,martinez2022global,martinez2023accelerated}.} for g-convex optimization depends linearly on the distance to an approximate minimizer, which can be exponentially large for tensor scaling and related problems~\cite{franksreichenback2021}. 
	In contrast, this extra linear dependence on the distance to an approximate minimizer is absent in Euclidean convex optimization.\footnote{As a concrete example, consider \emph{array scaling} --- the commutative analogue of tensor scaling --- which corresponds to a convex optimization problem in Euclidean space. 
		Although the distance to an approximate minimizer in array scaling can be exponentially large~\cite{franksreichenback2021}, polynomial-time algorithms are still available~\cite[Thm~3.1, Rem.~3.1]{NEMIROVSKI1999435}. 
		This is because the query complexity of standard convex optimization methods (e.g., the ellipsoid method~\cite{NEMIROVSKI1999435} or interior-point methods~\cite{bürgisser2020interiorpointmethodsunconstrainedgeometric}) depends only logarithmically on the distance to the solution.}
	Recent work~\cite{hamilton2021nogo,criscitiello2022negative,criscitiello2023curvature} has shown that this dependence on the distance to a minimizer is unavoidable in the setting of \emph{general} g-convex optimization, and is due to the \emph{negative curvature} of the underlying Hadamard manifold $M$.
	
	This raises a natural question:  
	\begin{center}
		\emph{Is there a meaningful subclass of g-convex functions whose query complexity does not depend on the curvature of the manifold?}
	\end{center}
	In this paper, we identify such a subclass: the set of \emph{horospherically convex} (h-convex) functions. 
	This class is meaningful in the sense that it covers nontrivial applications (see Section~\ref{sec:prob_example}) and generalizes Euclidean convexity. 
	To be clear, we do \emph{not} resolve the complexity of tensor scaling or similar problems.  However, we believe our work provides one reasonable starting point to this end.
	
	Recent progress in bridging the gap between convexity and g-convexity was made by Lewis \emph{et al.} \cite{lewis2024horoballs}. Instead of the family of convex functions, they consider the family of \emph{quasi-convex} functions. Their crucial insight is to generalize quasi-convexity in a \emph{horospherical} sense rather than a geodesic sense. Specifically, they show that if every sublevel set of $f$ is \emph{horospherically convex} in $M$, then the oracle complexity of the subgradient descent method matches exactly its Euclidean counterpart, without dependence on the curvature of the manifold.
	
	{The problem setting in \cite{lewis2024horoballs} is somewhat less standard, as they assume quasi-convexity of $f$ instead of convexity. In this paper, using a similar idea, we bridge the gap between convex functions and g-convex functions by introducing the concept of \emph{horospherical convexity} (h-convexity) for functions on $M$. To the best of our knowledge, h-convexity has previously been defined only for sets in $M$. The notion of h-convex functions is new to the literature.\footnote{The term \emph{horo-convex function} has appeared in the pure mathematics literature, for example, \cite{mejia2005horocyclically,fillastre2016hyperbolization,veronelli2019boundary}. However, these concepts are not closely related to our notion of h-convex functions.} }
	
	During the preparation of our manuscript, Goodwin \emph{et al.}~\cite{goodwin2024subgradient} independently introduced the concept of h-convexity, which they refer to as \emph{Busemann subdifferentiability}. While some of our observations overlap, our approaches and results differ in key ways, making the two works complementary. A detailed comparison with \cite{goodwin2024subgradient} is provided in Section~\ref{sec:prior} and summarized in Table~\ref{table:concurrent}.

	In Section~\ref{sec:contributions}, we detail the contributions of this paper (Part I), and also the topics considered by Part II. We conclude Section~\ref{sec:intro} by detailing prior and concurrent works (Section~\ref{sec:prior}).
	
	\begin{SCfigure}[50][t]
		\label{fig:horoball}
		\includegraphics[width=0.6\textwidth]{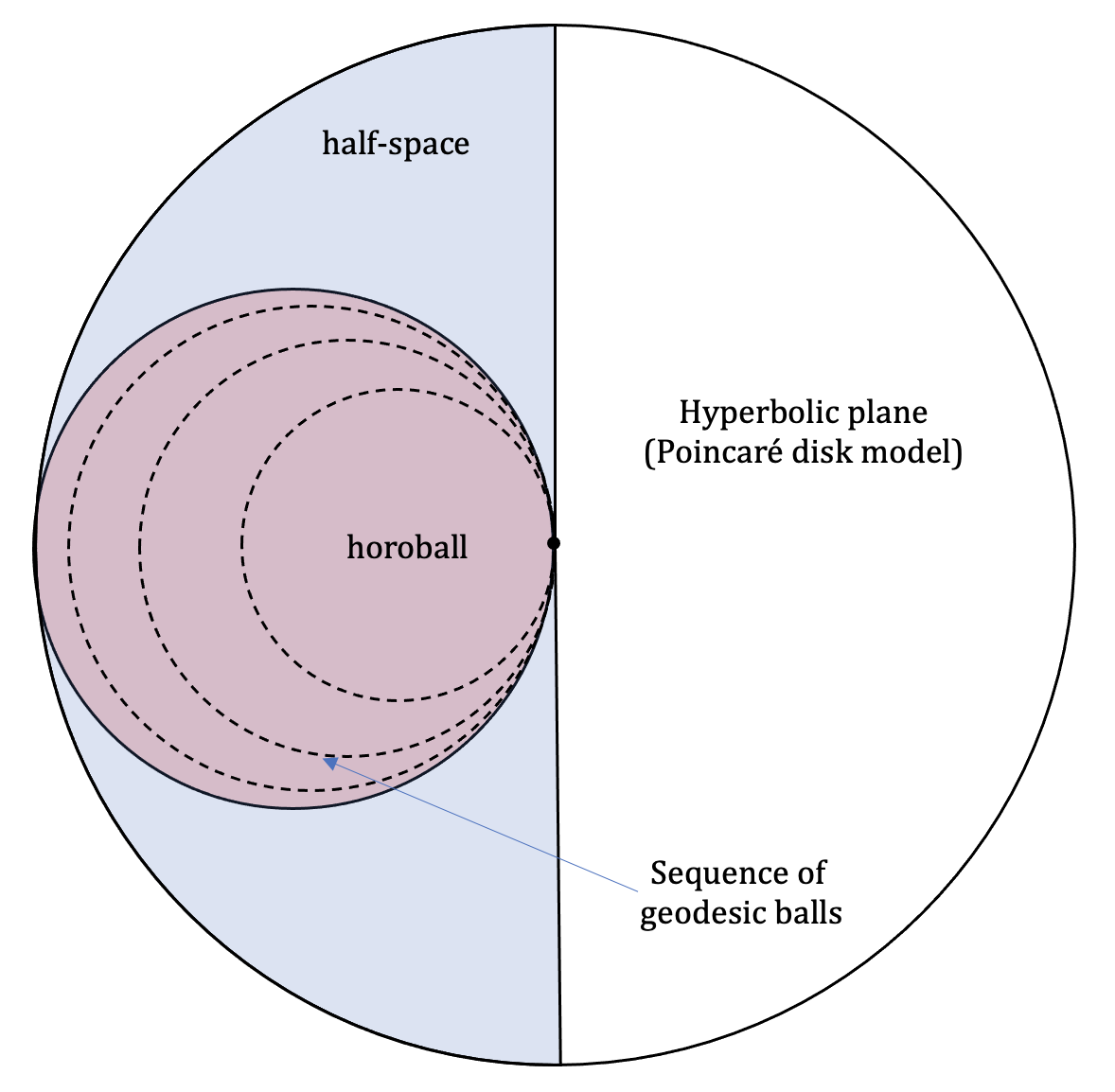}
		\caption{A horoball (purple) and half-space (blue) in the hyperbolic plane (Poincar\'e disk model). A horoball can be defined as the limit of a sequence of geodesic balls (portrayed with dashed boundaries).  In Euclidean spaces, horoballs are simply half-spaces.  On general Hadamard, horoballs differ from half-spaces.}
	\end{SCfigure}
	
	\subsection{Contributions}\label{sec:contributions}
	In Part~I of our paper, we develop the basic theory for h-convex functions and analyze the convergence rates of new (accelerated) first-order methods for minimizing the sum of h-convex functions. An overview of our contributions is provided below.
	
	\begin{itemize}
		\item {\bf Definition of horospherically convex functions.} We define the notion of horospherical convexity for both the non-strongly convex case ($\mu=0$) and the strongly convex case ($\mu>0$). 
		Notably, $\mu$-strongly h-convex functions are naturally defined only using squared distances, without reference to Busemann functions; see Section~\ref{sec:h-convex_functions}. Taking the limit $\mu \to 0$ then recovers h-convexity defined through Busemann functions (Proposition~\ref{prop:sctoc}).
		
		\item {\bf Definition of \emph{smooth} horospherically convex functions.} Since we are working on smooth Hadamard manifolds, it makes sense to consider smooth optimization problems.
		To that end, we introduce the notion of $L$-h-smoothness, a complementary notion to $\mu$-strong h-convexity. $L$-h-smoothness is weaker than the standard notion of $L$-smoothness.  Importantly, strongly g-convex function cannot be globally $L$-smooth~\cite[Appendix~A]{hamilton2021nogo}, but they can be globally $L$-h-smooth. Conveniently, this allows us to work without constraints, unlike in the g-convex setting~\cite{martinez2023accelerated}.

		\item {\bf Horospherically convex analysis.} In Section~\ref{sec:h-convex_analysis}, we study the basic properties of h-convex functions. We show that most (but not all) properties of convex functions can be extended to h-convex functions. For example, h-convexity is preserved under taking the supremum, but not under addition. 
		We also connect the h-convexity of a function on $M$ with the h-convexity of its epigraph on $M\times\mathbb{R}$.
		Finally, we show that h-convexity is preserved under the Moreau envelope.
		When comparing h-convexity and g-convexity, h-convexity generalizes the \emph{outer} and \emph{global} aspects of convexity, while g-convexity generalizes the \emph{inner} and \emph{local} aspects of convexity. 
		
		\item {\bf Horospherically convex optimization.} Since h-convexity is not preserved under addition, the problem of minimizing a single h-convex function has limited applications. Therefore, we consider the more general problem of minimizing the sum of h-convex functions (i.e., $\min \sum f_i$ where $f_i$ are h-convex), assuming access to a Fr\'echet-mean oracle.
		In Section~\ref{sec:prob_example}, we provide examples of optimization problems whose objective functions are sums of h-convex functions.
		These examples motivated the development of our general framework.
		
		\item {\bf Gradient descent and subgradient descent.} As a fundamental algorithm for h-convex optimization, we propose iteratively minimizing a sum of squared distances: given $x_{k}$, define $x_{k+1}$ as the minimizer of a (uniform) Fr\'echet mean, see \eqref{eq:gd} for the precise update rule.
		This algorithm recovers standard gradient descent when $M=\mathbb{R}^n$. 
		However, it differs from standard Riemannian gradient descent $x_{k+1}=\exp_{x_{k}}(-s\nabla f(x_{k}))$, unless the objective function itself is h-convex (not just a sum of h-convex functions).   
		The convergence rates we provide exactly match those for convex optimization.  We also show that in hyperbolic spaces, h-convex algorithms can converge significantly \emph{faster} than their Euclidean counterparts, see Section~\ref{sec:faster_rates}. 
		
		\item {\bf Accelerated first-order methods.} We generalize Nesterov's accelerated gradient methods \cite{nesterov1983method,nesterov2018lectures} and show that their convergence rates are exactly the same as in the Euclidean setting. Unlike generalizations of Nesterov's method in the g-convex setting (e.g., \cite{ahn2020nesterov,kim2022accelerated,martinez2023accelerated}), our algorithms do not require bounds on the sectional curvature of $M$.
	\end{itemize}
	
	In Part II of our paper, we focus on applications of the h-convex framework.
	{First, we consider a few problems arising in practice, most notably the \emph{minimum enclosing ball problem} (also considered by Lewis et al.~\cite{lewis2024horoballs}).
		Then we turn our attention to (smooth) \emph{interpolation} by h-convex functions, in the spirit of~\cite{taylor2017smooth,azagra2017extension}. 
		Even though we do not have a well-developed theory of duality for h-convex functions, we provide \emph{necessary and sufficient} conditions for interpolation by both non-smooth and smooth h-convex functions (by appealing to the Moreau envelope --- see {Section~\ref{sec:moreau}).
			In contrast, interpolation by g-convex functions is much trickier, and necessary and sufficient conditions are unknown~\cite[Sec.~8, App.~C]{criscitiello2023curvature}.
			We use these interpolation conditions to generalize the optimized gradient method (OGM) {\cite{kim2016optimized}} to the h-convex setting.

			\subsection{Prior and concurrent work}
			\label{sec:prior}
			
			The concept of horospherically convex sets (h-convex sets) has a long history in pure mathematics. To the best of our knowledge, this idea dates back to the works \cite{santalo1967horocycles, santalo1968horospheres} in hyperbolic geometry (i.e., $M=\mathbb{H}^n$). It was then naturally extended to the general Hadamard space setting (see, for example, \cite{borisenko2002convex}). However, h-convexity rarely appeared in the applied mathematics literature. {To the best of our knowledge, \cite{lewis2024horoballs} is the only prior work using the concept of h-convex sets to develop optimization theory.} Fletcher \emph{et al.} \cite{fletcher2011horoball} studied algorithms for approximating the \emph{horoball hull} (h-convex hull) in the space of positive definite matrices $\mathrm{PD}(n)$. 
			
			Although not explicitly dealing with h-convexity, the works \cite{hirai2023convex,bento2023fenchel} are closely related works to our paper. 
			They adopt the perspective of viewing the set $CM(\infty)$ of scaled Busemann functions as the dual space of $M$, which is also the perspective we take in our work (see Appendix~\ref{app:dual_space}). 
			More generally, the concepts of horospheres and Busemann functions arise in various areas of mathematics, notably in Fourier analysis on hyperbolic spaces~\cite{helgason1984groups}, and have recently found applications in machine learning~\cite{chami2021horopca,ghadimi2021hyperbolic,sonoda2022fully,bonet2023hyperbolic,fan2023horospherical,berg2024horospherical}.

			\begin{table}
				\centering
				\begin{tabular}{|c|c|} 
					\hline 
					Ours  &  Goodwin \emph{et al.}'s \cite{goodwin2024subgradient} \\
					\hline
					Hadamard manifolds & General Hadamard spaces \\
					Smooth and non-smooth functions & Non-smooth functions \\
					Convex and strongly convex functions  & Non-strongly convex functions\\
					Simultaneously updating iterates & Cyclically updating iterates\\
					Non-accelerated and accelerated methods & Non-accelerated methods\\
					\hline
					Prop.~\ref{prop:basic}~(iv) & Prop.~3.1  \\
					Prop.~\ref{prop:basic}~(vi) & Rem.~3.9 \\
					Thm.~\ref{thm:supporting_horosphere} & Def.~3.2 \\
					Consequence of Thm.~\ref{thm:envelope_rep_h-convex} and Prop.~\ref{prop:prod-B-functions} & Prop.~3.2 \\
					Prop.~\ref{prop:basic}~(vii) & Prop.~3.4 \\
					Prop.~\ref{prop:basic}~(iii) & Thm.~3.14 \\
					Appendix~\ref{app:rmk:basic} & Example~4.1 \\
					Thm.~\ref{thm:subg-descent-c} & Thm.~4.4 \\
					\hline
				\end{tabular}
				\caption{Comparison between our work and \cite{goodwin2024subgradient}. 
				}
				\label{table:concurrent}
			\end{table}
			
			{During the preparation of our manuscript, the closely related work of Goodwin \emph{et al.}~\cite{goodwin2024subgradient} appeared, addressing the same problem we study: minimizing the sum of h-convex functions.  While some of our theoretical findings overlap (see Table~\ref{table:concurrent}), our perspectives and settings differ, making the two works complementary rather than identical. Below, we summarize the key differences.}
			
			To emphasize the strength of their work, a significant difference is that they assume $M$ is a \emph{Hadamard space}, which is a more general concept than a Hadamard manifold. For example, the BHV tree space \cite{billera2001geometry}, used in phylogenetics, is a Hadamard space but not a Hadamard manifold. From an algorithmic perspective, their algorithm updates the iterate \emph{cyclically} (i.e., handling a single $f_i$ at a time), while our gradient update \eqref{eq:gd} must handle all $f_i$ simultaneously. This difference makes their algorithm more practical, although it does not affect the oracle complexity of h-convex optimization.
			(In fact, the oracle complexity of our subgradient method improves on that of~\cite{goodwin2024subgradient} by a factor $m$, where $m$ is the number of h-convex functions in the sum.)  
			
			{The primary distinguishing features of our work (Part I) are: }
			\begin{itemize}
				\item {Introduction of \emph{strongly} h-convex functions.  In particular, $\mu$-strongly h-convex functions are naturally defined only using squared distance, without reference to Busemann functions.  We then observe taking the limit $\mu \to 0$ recovers h-convexity defined through Busemann functions.}
				
				\item Introduction of \emph{smooth} h-convex functions.
				Since we are working on smooth Hadamard manifolds, it makes sense to consider smooth h-convex optimization.
				We introduce $L$-h-smoothness --- a complementary notion to $\mu$-strong h-convexity --- and show that it is weaker than the standard notion of $L$-smoothness.  Then we analyze algorithms in the presence of $L$-h-smoothness.  We also show that the \emph{Moreau envelope} makes h-convex functions $L$-h-smooth.

				\item For smooth h-convex optimization, we propose variants of gradient descent and Nesterov's accelerated method, and prove their curvature-independent\footnote{By curvature-independent, we mean the algorithm performs at least as well as the corresponding Euclidean algorithm, and no lower bound on the sectional curvature of $M$ is needed.} convergence rates (see Section~\ref{sec:momentum}), while only subgradient methods are considered in \cite{goodwin2024subgradient}.
				We also show that when $M$ is significantly curved, rates in h-convex optimization are, surprisingly, much faster than those in convex optimization (Section~\ref{sec:faster_rates}).
			\end{itemize}

			\section{Preliminaries}
			\label{sec:pre}

			In this section, we review some definitions and notations in Riemannian geometry. For a comprehensive review, the readers can consult the following textbooks: \cite{petersen2006riemannian,lee2018introduction} for basic Riemannian geometry, \cite{ballmann1985,ballmann2012lectures,bridson2013metric} for the theory of horoballs and Busemann functions, \cite{rockafellar1997convex,hiriart1996convex,boyd2004convex,rockafellar2009variational} for convex analysis and optimization on $\mathbb{R}^n$, and \cite{udriste1994convex,bacak2014convex,boumal2020intromanifolds} for geodesically convex analysis and optimization on $M$.
			
			\subsection{Basics of Riemannian geometry and Hadamard manifolds}
			\label{sec:pre1}
			
			A \emph{Riemannian manifold} $M$ is a smooth manifold equipped with a positive definite inner product $\langle \cdot, \cdot \rangle_p$ on each tangent space $T_p M$. This structure defines the norm $\| \cdot \|_p$ on $T_p M$, and the distance $d(\cdot,\cdot)$ on $M$. A \emph{Hadamard manifold} is a complete, simply connected Riemannian manifold with non-positive sectional curvature. According to the \emph{Cartan-Hadamard theorem}, every $n$-dimensional Hadamard manifold is diffeomorphic to $\mathbb{R}^n$. 
			Let $M$ be a Hadamard manifold.
			A \emph{geodesic} is a smooth curve $\gamma: I \to M$ with zero acceleration. For any $p \in M$ and $v \in T_p M$, there exists a unique geodesic $\gamma_{p, v}: (-\infty, \infty) \to M$ with $\gamma_{p, v}(0) = p$ and $\gamma_{p, v}'(0) = v$. The \emph{exponential map} $\exp_p: T_p M \to M$ is defined as $\exp(v) = \gamma_{p, v}(1)$, and is a diffeomorphism with the inverse $\exp^{-1}_p$. The parallel transport along the geodesic segment from $p$ to $q$ is denoted by $\Gamma_{p}^{q} \colon T_p M \to T_p M$.
			
			Throughout this paper, $M$ denotes a Hadamard manifold and $n$ denotes its dimension. We denote the open geodesic ball with center $p$ and radius $r$ as $B(p, r)$. Similarly, closed geodesic balls are denoted as $\bar{B}(p, r)$.
			We will often use the following fundamental triangle comparison theorem for Hadamard manifolds.
			\begin{lemma}[Triangle comparison]
				\label{lem:rev_toponogov}
				For any $p,x,y\in M$, the following inequality holds:
				\[
				d(x,y)^{2}\geq\left\Vert \exp_{p}^{-1}(x)-\exp_{p}^{-1}(y)\right\Vert ^{2}=d(x,p)^{2}+d(y,p)^{2}-2\left\langle \exp_{p}^{-1}(x),\exp_{p}^{-1}(y)\right\rangle .
				\]
			\end{lemma}
			\begin{proof}
				The following equivalent statement can be found in \cite[\S6.2.2]{petersen2006riemannian}: if there is a geodesic triangle in $M$ with sides lengths $a$, $b$, $c$ and where the angle opposite $a$ is $\alpha$, then we have $a^{2}\geq b^{2}+c^{2}-2bc\cos\alpha.$
			\end{proof}
			
			\subsection{Geodesic convexity}
			\label{sec:pre2}
			
			For $x, y \in M$, denote the geodesic segment from $x$ to $y$ as $\gamma_{x}^{y} : [0,1] \to M$ (i.e., $\gamma_{x}^{y}(0)=x$ and $\gamma_{x}^{y}(1)=y$). A subset $C \subseteq M$ is \emph{geodesically convex} (or {g-convex}) if, for any $x, y \in C$ and $\lambda \in [0,1]$, we have $\gamma_{x}^{y}(\lambda) \in C$. A function $f : M \to \mathbb{R}$ is \emph{geodesically $\mu$-strongly convex} (or $\mu$-strongly g-convex) if, for any $x, y \in M$ and $\lambda \in [0,1]$, the inequality 
			$$f(\gamma_{x}^{y}(1-\lambda)) \leq \lambda f(x) + (1-\lambda) f(y) - \frac{\mu}{2} \lambda(1-\lambda) d(x, y)^{2}$$
			holds. It is well-known that $f$ is $\mu$-strongly g-convex if and only if, for any $y\in M$, there is a tangent vector $v\in T_y M$ such that
			\begin{align}\label{eq:stronglygrad}
				f(x)\geq f(y)+\langle v,\exp^{-1}_{y}(x)\rangle+\frac{\mu}{2}d(x,y)^{2}\quad\textup{for all }x\in M.
			\end{align}
			If $f$ is differentiable, then it is $\mu$-strongly g-convex if and only if the inequality \eqref{eq:stronglygrad} holds everywhere, with $v$ set to the gradient $\nabla f(y)$ of $f$ at $y$. If $f$ is twice-differentiable, then the $\mu$-strong g-convexity condition is equivalent to $\nabla^2 f(x) \succeq \mu I$ for all $x\in M$. When $M = \mathbb{R}^n$ with the standard metric, the notion of g-convexity reduces to the standard notion of convexity, which we will simply refer to as \emph{Euclidean convexity} throughout the paper.
			
			On Hadamard manifolds, distance functions $x \mapsto d(x,p)$ are g-convex, and squared distance functions $x \mapsto d(x,p)^2$ are $2$-strongly g-convex~\cite[Lem.~2]{alimisis2020continuous}, \cite[Cor.~16]{roux2025implicit}.
			The gradient of a (squared) distance points directly away from its minimizer: provided $x \neq p$, we have
			\begin{align}\label{eq:graddist}
				\nabla_x d(x, p)^2 = -2 \exp^{-1}_x(p), \quad \quad \nabla_x d(x, p) = - \frac{\exp^{-1}_x(p)}{d(x,p)}.
			\end{align}

			\subsection{Horoballs and Busemann functions}
			\label{sec:pre3}
			
			Two geodesic rays $\gamma_1, \gamma_2 : [0, \infty) \to M$ are \emph{asymptotic} if $d(\gamma_1(t), \gamma_2(t))$ is bounded. The \emph{boundary at infinity} $M(\infty)$, is the set of equivalence classes of asymptotic unit-speed geodesic rays. The equivalence class containing $\gamma$ is denoted by $\gamma(\infty)$. The set $M(\infty)$ is endowed with the \emph{sphere topology} (also known as the \emph{cone topology} or the \emph{standard topology}).\footnote{This topology is induced by the topology of uniform convergence on compact sets \cite[\S3]{ballmann1985}, \cite[\S II.1]{ballmann2012lectures}. However, it can also be characterized by the following property: fix a reference point $p \in M$, and consider the unit sphere $S_{p}=\{v\in T_{p}M:\|v\|=1\}$. Define a map $\phi_p: S_p \to M(\infty)$ by $\phi_p(v) = \gamma_v(\infty)$, where $\gamma_v$ is the geodesic ray with $\gamma_v(0) = p$ and $\gamma_v'(0) = v$. Then, the map $\phi_p$ is a homeomorphism \cite[p. 22]{ballmann1985}.} 
			
			For a unit-speed geodesic ray $\gamma:[0,\infty)\to M$, the \emph{Busemann function} $B_{\gamma}:M\to \mathbb{R}$ is defined by
			\[
			B_{\gamma}(x)=\lim_{t\rightarrow\infty}\left(d\left(\gamma(t),x\right)-t\right).
			\]
			We also denote $B_{\gamma}$ as $B_{p,\xi}$ with $p=\gamma(0)$ and $\xi=\gamma(\infty)$, or as $B_{p,v}$ with $p=\gamma(0)$ and $v=-\gamma'(0)$.\footnote{Our notation $B_{p,v}$ is not standard in the literature ($B_{p,-v}$ is more common.). We have chosen our notation to align with the affine function $x\mapsto \langle v, x-p \rangle$.} Sublevel sets $B_{\gamma}^{-1}((-\infty,c])$ are called \emph{closed horoballs}, and level sets $B_{\gamma}^{-1}(c)$ are called \emph{horospheres} centered at $\gamma(\infty)$.  Busemann functions have the following properties:\footnote{We are assuming $\|v\|=1$ here. It is straightforward to extend these properties to scaled Busemann functions.}
			\begin{itemize}
				\item When $M=\mathbb{R}^n$, we have $B_{p,v}(x)=\langle v,x-p\rangle$.
				\item $B_{p,v}$ is geodesically convex and $1$-Lipschitz continuous.
				\item $B_{p,v}$ is continuously differentiable with $\|\grad B_{p,v}(x)\|=1$ everywhere.
				\item $B_{p, v}(p) = 0$ and $\nabla B_{p,v}(p) = v$.  For $x = \exp_p(-t v)$ with $t > 0$, we have
				\begin{align}\label{eq:gradBusemann}
					\nabla B_{p,v}(x)=\frac{\exp_{x}^{-1}(p)}{d(x,p)}=\frac{1}{t}\exp_{x}^{-1}(p).
				\end{align}
				\item Gradient flows of $B_{p,v}$ are geodesics. In particular, the gradient flow of $B_{p,v}$ starting at $p\in M$, that is, the solution to $\dot{X}(t)=-\nabla f(X(t))$ with the initial condition $X(0)=p$, can be expressed as $t\mapsto\exp_{p}(-tv)$.
				\item $f:M\to\mathbb{R}$ is a Busemann function if and only if $f+r$ is a Busemann function for any $r\in\mathbb{R}$. 
				If $\gamma_1, \gamma_2$ are asymptotic unit-speed geodesic rays, then $x \mapsto B_{\gamma_1}(x) - B_{\gamma_2}(x)$ is a constant. 
				Two Busemann functions differ by a constant if they have the same gradient at some point.
			\end{itemize}
			
			Since Busemann functions only generalize affine functions with norm $1$, they do not encompass the notion of all affine functions when $M = \reals^n$. Therefore, we extend it by considering non-negative multiples of Busemann functions. For $r\geq 0$, a scaled Busemann function $rB_{p,v}$ is denoted by $B_{p,rv}$. To clarify the terminology, we often refer to (unscaled) Busemann functions as \emph{unit-scale} Busemann functions. The space of scaled Busemann functions can be viewed as the cone $CM(\infty)$ of $M(\infty)$. We assume that the topology on $CM(\infty)$ is induced by the sphere topology on $M(\infty)$. See Appendix~\ref{app:dual_space} for details about the space $CM(\infty)$. See Table~\ref{table:1} for a comparison of horoballs and Busemann functions in Euclidean spaces and general Hadamard manifolds.
			\begin{table}
				\centering
				\begin{tabular}{|c|c|} 
					\hline 
					Hadamard manifolds $M$ &  Euclidean spaces $\mathbb{R}^n$    \\
					\hline
					Scaled Busemann function $B_{p,v}(x)$ & Affine function $\langle v,x-p\rangle$ \\
					Closed horoball $\{x:B_{p,v}(x)\leq r\}$ & Closed half-space $\{x:\langle v,x-p\rangle\leq r\}$  \\
					Cone at infinity $CM(\infty)$ & Dual space $(\mathbb{R}^n)^*$ \\
					\hline
				\end{tabular}
				\caption{Correspondence between the notions related to scaled Busemann functions on $M$ and the notions related to affine functions on $\mathbb{R}^n$.}
				\label{table:1}
			\end{table}
			
			Perhaps, the two most important examples of Hadamard manifolds in applications (besides $\mathbb{R}^n$) are hyperbolic space $\mathbb{H}^n$ and the symmetric space of positive definite matrices $\mathrm{PD}(n)$.
			
			\begin{example}[Hyperbolic space]\label{exampleHyper}
				\textup{
					Hyperbolic space $\mathbb{H}^n$ is the $n$-dimensional Hadamard manifold with constant sectional curvature $-1$; see, e.g.,~\cite{ratcliffe2019hyperbolic}.
					For simplicity, focus on the hyperbolic plane $\mathbb{H}^2$, modeled by the open unit disk $M = B(0,1) \subseteq \mathbb{R}^2$ equipped with the Riemannian metric $\langle u, v \rangle_x = \frac{4 u^\top v}{(1 - \|x\|^2)^2},$
					known as the Poincar\'e disk model.
					Isometries in this model correspond to M\"obius transformations that preserve the unit disk.
					In this setting, geodesic balls are Euclidean balls contained in the unit disk, while \emph{horoballs} are Euclidean balls \emph{tangent} to the boundary circle: see Figure~\ref{fig:horoball}.
					Explicit formulas for unit-speed geodesics and Busemann functions in this model are given in Appendix~\ref{sec:h2}.
					These concepts extend naturally to the Poincaré \emph{ball} model of $\mathbb{H}^n$ for $n > 2$.}
				
			\end{example}
			
			\begin{example}[Positive definite matrices.]\label{examplePD}
				\textup{
					Equip the space of real symmetric positive definite matrices $\mathrm{PD}(n)$ with the Riemannian metric $\langle X, Y \rangle_P = \mathrm{Tr}(P^{-1} X P^{-1} Y),$
					where $P \in \mathrm{PD}(n)$ and $X, Y$ are symmetric matrices. 
					This is known as the \emph{affine-invariant metric}; it also arises as the Fisher-Rao metric for covariance estimation for multivariate Gaussians. Under this metric, $\mathrm{PD}(n)$ becomes a $\frac{n(n+1)}{2}$-dimensional Hadamard manifold and a noncompact symmetric space. Its isometry group acts via $P \mapsto g P g^\top$ for $g \in \mathrm{GL}(n)$.
					Unit-speed geodesics emanating from the identity matrix are given by $t \mapsto \exp(t X),$
					where $X$ is a symmetric matrix with Frobenius norm $\|X\| = 1$.
					Euclidean spaces, hyperbolic spaces, $\mathrm{SPD}(n) = \{P \in \mathrm{PD}(n) : \det(P) = 1\}$, and the set of \emph{complex} positive definite matrices (of size $n/2 \times n/2$, assuming $n$ is even) can all be realized as totally geodesic submanifolds of $\mathrm{PD}(n)$~\cite[Thm 10.58]{bridson2013metric}.}
				
				\textup{
					Let $\lambda$ be a diagonal matrix with strictly decreasing entries: $\lambda_{11} > \lambda_{22} > \cdots > \lambda_{nn}$.
					The Busemann function associated with the geodesic ray $t \mapsto \exp(t \lambda)$ is given by}
				\begin{align}\label{busemannPDn}
					b_\lambda(g \exp(\mu) g^\top) = \langle \lambda, \mu \rangle = \mathrm{Tr}(\lambda \mu),
				\end{align}
				\textup{for any diagonal matrix $\mu$, and any upper triangular matrix $g$ with ones on the diagonal~\cite[Thm 10.69]{bridson2013metric}.
					Formulas for most other Busemann functions and horoballs in $\mathrm{PD}(n)$ can be derived via the isometries $P \mapsto g P g^\top$.
					Alternatively, see~\cite[Lem 2.32]{hirai2023convex} for an expression involving determinants of minors.
					For further background on the geometry of $\mathrm{PD}(n)$, see~\cite[Sec II.10]{bridson2013metric}.}

			\end{example}

			\section{Horospherically convex analysis}
			\label{sec:h-convex_analysis}
			
			In this section, we study the basic properties of horospherically convex functions.

			\subsection{Basic facts and definitions}
			\label{sec:h-convex_functions}

			We start with the definition of horospherically convex functions.
			\begin{definition}[Horospherically convex functions]
				\label{def:h-convex_funtion}\ 
				\begin{enumerate}[label=\textup{(\roman*)}]
					\item A function $f \colon M \to \mathbb{R}$ is \emph{horospherically convex} (h-convex) if for all $y \in M$, there exists a tangent vector $v \in T_y M$ such that
					\begin{equation}
						\label{eq:defining_h-convex}
						f(x) - f(y) \geq B_{y,v}(x) \quad \forall x \in M.
					\end{equation}
					For $y\in M$, we refer to $v \in T_y M$ such that \eqref{eq:defining_h-convex} holds as a \emph{horospherical subgradient} (h-subgradient) of $f$ at $y$. We refer to the set of all h-subgradients of $f$ at $y$ as the \emph{horospherical subdifferential} (h-subdifferential) of $f$ at $y$, and denote it by $\ph f(y)$.
					\item A function $f \colon M \to \mathbb{R}$ is \emph{horospherically $\mu$-strongly convex} ($\mu$-strongly h-convex) if for all $y \in M$, there exists a tangent vector $v \in T_y M$ such that
					\begin{equation}
						\label{eq:defining_strong-h-convex}
						f(x)-f(y)\geq Q_{y,v}^{\mu}(x):=-\frac{1}{2\mu}\|v\|^{2}+\frac{\mu}{2}d\left(\exp_{y}\left(-\frac{1}{\mu}v\right),x\right)^{2} \quad  \forall x \in M.
					\end{equation}
					For $y\in M$, we denote the set of $v\in T_y M$ such that \eqref{eq:defining_strong-h-convex} holds by $\ph_{\mu} f(y)$.\footnote{One might instead try to define the notion of strong h-convexity by using the function $\tilde{Q}_{y,v}^{\mu}(x) := B_{y,v}(x) + \frac{\mu}{2} d(x, y)^2$ in the place of $Q_{y,v}^{\mu}(x)$. Note that $\tilde{Q}_{y,v}^{\mu}(x)\geq Q_{y,v}^{\mu}(x)$ by Lemma~\ref{lem:horotriangle}. Thus, this alternative definition is stricter than our definition of $\mu$-strong h-convexity. In fact it is too restrictive, since squared distance functions $x\mapsto \frac{\mu}{2}d(x,p)^2$ are not necessarily $\mu$-strongly convex in this sense.} 
				\end{enumerate}
			\end{definition}
			At first glance, the two definitions may seem unrelated. However, the following proposition shows that they are seamlessly connected (see Appendix~\ref{app:prop:sctoc} for the proof).
			\begin{proposition}
				\label{prop:sctoc}
				The right-hand side $Q_{y,v}^{\mu}(x)$ of \eqref{eq:defining_strong-h-convex} converges pointwise to the right-hand side $B_{y,v}(x)$ of \eqref{eq:defining_h-convex} as $\mu\to 0$.
			\end{proposition}
			In light of this result, we consider h-convexity as a special case of the $\mu$-strong h-convexity with $\mu = 0$. 
			As a counterpart to the list of properties of Busemann functions in Section~\ref{sec:pre3}, we present some straightforward properties of the functions $Q_{p,v}^{\mu}(x)$ here.
			\begin{itemize}
				\item When $M=\mathbb{R}^n$, we have $Q_{p,v}^{\mu}(x)=\langle v,x-p\rangle+\frac{\mu}{2}\|x-p\|^2$.
				\item $Q_{p,v}^{\mu}$ is $\mu$-strongly geodesically convex.
				\item $Q_{p,v}^{\mu}(p) = 0$ and $\nabla Q_{p,v}^{\mu}(p) = v$.
			\end{itemize}
			
			We summarize some basic properties of (strongly) h-convex functions in the following proposition, whose proof can be found in Appendix~\ref{app:prop:basic}.
			\begin{proposition}[Properties of horospherically convex functions]
				\label{prop:basic} \ 
				\begin{enumerate}[label=\textup{(\roman*)}]
					\item If $0\leq \mu_1\leq \mu_2$, then the set $\ph_{\mu_2}f(y)$ is a subset of the set $\ph_{\mu_1}f(y)$.
					\item If $f$ is $\mu_2$-strongly h-convex and $0\leq \mu_1\leq \mu_2$, then $f$ is $\mu_1$-strongly h-convex.
					\item For any $\mu\geq 0$, the set $\ph_{\mu} f(y)$ is a subset of the (geodesic) subdifferential $\partial f(y)$.
					\item If $f:M\to\mathbb{R}$ is $\mu$-strongly h-convex, then $f$ is $\mu$-strongly g-convex.
					\item If $f:M\to\mathbb{R}$ is differentiable and $\mu$-strongly h-convex, then $\ph_{\mu} f(y)=\{\grad f(y)\}$ for all $y\in M$.
					\item The supremum of $\mu$-strongly h-convex functions is also $\mu$-strongly h-convex.
					\item If $f:M\to\mathbb{R}$ is h-convex and $g:\mathbb{R}\to\mathbb{R}$ is increasing and convex, then $g\circ f$ is h-convex.
					\item If $f:M\to\mathbb{R}$ is a $\mu$-strongly h-convex function and $r$ is a non-negative constant, then the function $rf$ is $(r\mu)$-strongly h-convex.
				\end{enumerate}
			\end{proposition}
			Unlike geodesic convexity, we emphasize that horospherical convexity is \emph{not preserved under addition} in general. In particular, this shows that the inclusion $\ph f(y)\subseteq \partial f(y)$ may be proper. See Appendix~\ref{app:rmk:basic} for details.  
			Similarly, the restriction of an h-convex function to a totally geodesic submanifold is not h-convex in general.  This is because a scaled Busemann function on the product manifold $M \times M$ becomes a sum of two scaled Busemann functions (see Proposition~\ref{prop:prod-B-functions}) when restricted to the diagonal submanifold $\{(x, x) : x \in M\}$.
			
			Throughout the paper, we will also use the following ``horospherical law of cosines'':\footnote{Compare the inequality in Lemma~\ref{lem:rev_toponogov} with the inequality~\eqref{eq:eq1tri}. When $M = \reals^n$ both inequalities become equalities, and both reduce to the Euclidean law of cosines.}
			\begin{lemma}\label{lem:horotriangle}
				For $x, y \in M$, $v \in T_y M$ and $\mu > 0$, we have
				\begin{align}
					\label{uno}
					Q_{y, v}^L(x) &\leq B_{y,v}(x) + \frac{L}{2} d(x, y)^2, \\
					\label{duo}
					Q_{y, v}^L(x) &\leq Q_{y,v}^\mu(x) + \frac{L-\mu}{2} d(x, y)^2.
				\end{align}
				\noindent In particular, for $x, y, p \in M$ we have
				\begin{align}\label{eq:eq1tri}
					d(p, x)^2 \leq d(p, y)^2 + d(x, y)^2 + 2 B_{y,-\exp^{-1}_{y}(p)}(x).
				\end{align}
			\end{lemma}
			\begin{proof}
				Let us just prove~\eqref{uno}: the proof for~\eqref{duo} is identical.
				The function $x \mapsto B_{y,v}(x) + \frac{L}{2} d(x, y)^2$ is $L$-strongly g-convex with a minimizer at $p = \exp_{y}(-\frac{1}{L} v)$, because its gradient is zero at $p$ (use equations~\eqref{eq:gradBusemann} and~\eqref{eq:graddist}).
				Therefore~\eqref{eq:stronglygrad} yields
				$$B_{y,v}(x) + \frac{L}{2} d(x, y)^2 \geq B_{y,v}(p) + \frac{L}{2} d(y, p)^2 + \frac{L}{2}d(x, p)^2 = - \frac{1}{2L}\|v\|^2 + \frac{L}{2} d(x, p)^2 = Q_{y, v}^L(x).$$
				Lastly, observe that~\eqref{eq:eq1tri} comes from~\eqref{uno} setting $L=1$ and $p = \exp_y(-v)$.
			\end{proof}
			
			\begin{remark}
				
				From a geometric perspective, h-convexity is more natural than g-convexity, at least in some aspects.
				For example, if $f$ is $\mu$-strongly h-convex then for any $x\in M$, the minimizer $x^*$ lies in the geodesic ball 
				\begin{align}\label{sccccc1}
					\bar{B}\left(x^{++},\frac{1}{\mu}\left\Vert \nabla f(x)\right\Vert \right),
					\quad \textup{with } x^{++} = \exp_x\Big(-\frac{1}{\mu} \nabla f(x)\Big).
				\end{align}
				On the other hand, if $f$ is $\mu$-strongly g-convex then one can only guarantee that the minimizer $x^*$ is contained in the image of a non-centered ball under the exponential map:
				\begin{align}\label{sccccc2}
					x^* \in \exp_{x}\left(\bar{B}\left(\frac{1}{\mu} \nabla f(x),\frac{1}{\mu}\left\Vert \nabla f(x)\right\Vert \right)\right).
				\end{align}
				This set includes the geodesic ball \eqref{sccccc1} by Lemma~\ref{lem:rev_toponogov}, and one may view \eqref{sccccc2} as a ball which is stretched along the directions perpendicular to $\grad f(x)$. 
				This set aligns more with the geometry of $T_{x}M$ than with the geometry of $M$.
				Taking the limit \( \mu \to 0 \) further highlights this distinction: in the h-convex case~\eqref{sccccc1}, the containing set becomes a horoball (which is h-convex), whereas in the g-convex case~\eqref{sccccc2}, it becomes a ``half-space'' (which is not g-convex; see Figure~\ref{fig:horoball}).

				This geometric viewpoint --- characterizing sets that contain the minimizers --- is often instrumental in the design and analysis of optimization algorithms, such as the ellipsoid method or geometric descent~\cite{bubeck2015geometric, drusvyatskiy2018optimal, ma2017underestimate}, a variant of Nesterov's accelerated method. Accordingly, one might expect convergence analyses to be more transparent or more favorable in the h-convex setting.
				This is indeed the case; see Section~\ref{sec:h-opt}.
			\end{remark}
			
			\subsection{Outer characterization of horospherically convex functions}
			\label{sec:outer-h-conv}

			In classical convex analysis, there are two standard ways to characterize convex functions on $\mathbb{R}^n$. 
			The most familiar is the \emph{inner characterization}, which defines a function $f: \mathbb{R}^n \to \mathbb{R}$ as convex if, for all $x, y \in \mathbb{R}^n$ and $\lambda \in [0,1]$, it satisfies the inequality $f(\lambda x + (1 - \lambda) y) \leq \lambda f(x) + (1 - \lambda) f(y).$
			Alternatively, there exists an \emph{outer characterization}, often called the \emph{envelope representation}, which we state in the following proposition.
			\begin{proposition}[{\cite[Theorem~8.13]{rockafellar2009variational}}]
				\label{prop:E-outer}
				A function on $\mathbb{R}^n$ is convex if and only if it is the pointwise supremum of a non-empty family of affine functions.
			\end{proposition}
			\begin{proof}[Proof sketch]
				($\Longrightarrow$) Suppose that $f: \mathbb{R}^n \to \mathbb{R}$ is a convex function. Then, for each $x\in\mathbb{R}^n$, there is a vector (namely, a subgradient of $f$ at $x$) $g_x\in \mathbb{R}^n$ such that $f(y)\geq f(x)+\langle g_x,y-x\rangle$ for all $y\in\mathbb{R}^n$. Therefore, $f$ can be written as the pointwise supremum of its affine supports $L_{x, g_{x}, f(x)}: y \mapsto f(x) + \langle g_{x}, y-x \rangle$, where the supremum is taken over all $x \in \mathbb{R}^n$.
				
				($\Longleftarrow$) This follows from the facts that affine functions are convex and that the supremum of convex functions is convex.
			\end{proof}
			We now examine how the outer characterization can be extended to the manifold setting. 
			The affine supports $L_{x,g_x,f(x)}$ of $f$, which appeared in the proof of Proposition~\ref{prop:E-outer}, are generalized to the g-convex setting as $\tilde{L}_{x,g_{x},f(x)}: y \mapsto f(x) + \langle g_{x}, \exp_{x}^{-1}(y) \rangle$.\footnote{The function $\tilde{L}_{x,g_{x},f(x)}$ is the tightest obtainable lower bound, when one knows the value of $f$ at $x$ and that $g_x$ is a (geodesic) subgradient of $f$ at $x$ \cite[Prop.~35]{criscitiello2023curvature}.}  
			Using these functions in place of affine functions, one can establish the ($\Longrightarrow$) direction of Proposition~\ref{prop:E-outer} in the g-convex setting.
			
			However, a key obstruction arises: the functions $\tilde{L}_{x, g_x, f(x)}$ are generally not g-convex. 
			As a result, the converse implication ($\Longleftarrow$) does not hold, and the outer characterization fails to fully generalize to g-convexity.
			Fortunately, the situation is more favorable in the h-convex setting. 
			As we show in the next theorem (proof in Appendix~\ref{app:thm:envelope_rep_h-convex}), h-convex functions admit a natural and simple outer characterization.
			
			\begin{theorem}[Outer characterization for h-convex functions]
				\label{thm:envelope_rep_h-convex}\ 
				\begin{enumerate}[label=\textup{(\roman*)}]
					\item A function $f:M\to\mathbb{R}$ is h-convex if and only if it is the pointwise supremum of a non-empty family of scaled Busemann functions.
					\item A function $f:M\to\mathbb{R}$ is $\mu$-strongly h-convex if and only if it is the pointwise supremum of a non-empty family of functions in the form $x\mapsto \frac{\mu}{2}d(x,p)^{2}+r$.
				\end{enumerate}
			\end{theorem}
			\begin{remark}
				Theorem~\ref{thm:envelope_rep_h-convex} indicates that h-convexity is an instance of \emph{abstract convexity}, a framework that vastly generalizes the outer characterization of convexity by replacing the family of affine functions with an (almost) arbitrary family of functions~\cite{singer1997abstract}, \cite[\S7]{rubinov2013abstract}.
			\end{remark}

			\subsection{Horospherical $L$-smoothness and Moreau envelopes}
			\label{sec:moreau}
			
			In this subsection, we introduce the notion of \emph{horospherical $L$-smoothness} on Hadamard manifolds, which differs from the existing notion of geodesic $L$-smoothness.
			We show that the Moreau envelope naturally aligns with this smoothness notion and with h-convexity. Furthermore, we establish a connection between horospherical $L$-smoothness and the concept of $c$-concavity from optimal transport; see Remark~\ref{rmk:c-concavity}.
			\begin{definition}
				A differentiable function $f:M \to \mathbb{R}$ is \emph{horospherically $L$-smooth} ($L$-h-smooth) if for all $y \in M$, we have
				\begin{equation}
					\label{eq:defining-smooth}
					f(x)-f(y)\leq Q_{y,\nabla f(y)}^{L}(x):=-\frac{1}{2L}\|\nabla f(y)\|^{2}+\frac{L}{2}d\left(\exp_{y}\left(-\frac{1}{L}\nabla f(y)\right),x\right)^{2} \quad \forall x\in M.
				\end{equation}
			\end{definition}
			\noindent Note that the $L$-h-smoothness condition bounds the \emph{curvature} of $f$ from above, similarly to how the $\mu$-strong h-convexity condition~\eqref{eq:defining_strong-h-convex} bounds the \emph{curvature} of $f$ from below. 
			
			In g-convex optimization, one often assumes geodesic smoothness of the objective function. A differentiable function $f$ is said to be \emph{geodesically $L$-smooth} ($L$-g-smooth) if
			\[
			f(y) \leq f(x) + \langle \nabla f(x), \exp_x^{-1}(y)\rangle + \frac{L}{2} d(x,y)^2 \quad \textup{for all}\ x, y \in M.
			\]
			If $f$ is twice-differentiable, this is equivalent to $\nabla^2 f(x) \preceq L I$ for all $x\in M$.
			The following proposition (proof in Appendix~\ref{app:prop:smooth}) can be seen as a counterpart of Theorem~\ref{thm:envelope_rep_h-convex} and Proposition~\ref{prop:basic}~(iv).
			\begin{proposition}
				\label{prop:smooth}\ 
				\begin{enumerate}[label=\textup{(\roman*)}]
					\item A differentiable function $f:M\to\mathbb{R}$ is $L$-h-smooth if and only if it is the pointwise infimum of a non-empty family of functions in the form $x\mapsto \frac{L}{2}d(x,p)^{2}+r$.
					\item If $f$ is $L$-g-smooth, then it is $L$-h-smooth (see also \cite[Prop.~4.15~(iii)]{leger2023gradient} and Remark~\ref{rmk:c-concavity}).
					
					\item If $f$ is $L$-h-smooth, then it satisfies the \emph{descent condition} $f(y^{+})\leq f(y)-\frac{1}{2L}\|\nabla f(y)\|^{2}$, where $y^+=\exp_y(-\frac{1}{L}\nabla f(y))$.
				\end{enumerate}
			\end{proposition}
			\noindent In summary: 
			$$\text{Geodesic $L$-smoothness $\Longrightarrow$ Horospherical $L$-smoothness $\Longrightarrow$ Descent condition.}$$
			
			Busemann functions are $0$-h-smooth: this follows from Proposition~\ref{prop:sctoc} and the fact that $Q_{y,v}^{L}(x)$ is increasing in $L$, as shown in the proof of Proposition~\ref{prop:basic}~(i) in Appendix~\ref{app:prop:basic}.
			Squared distance functions are $2$-h-smooth.
			In contrast, if $M$ is a Hadamard manifold with sectional curvatures at least $K$, then the function $x \mapsto d(x, p)^2$ is geodesically $2\frac{r \sqrt{-K}}{\tanh(r \sqrt{-K})}$-smooth in $B(p, r)$ (see, for example, \cite[Lem.~2]{alimisis2020continuous}, \cite[Cor.~16]{roux2025implicit}), and this smoothness parameter is tight.
			Note that $2\frac{r \sqrt{-K}}{\tanh(r \sqrt{-K})}$ is always greater than $2$ and can be arbitrarily large as $r\sqrt{-K}$ increases. This in turn means that gradient descent based on geodesic smoothness takes much smaller step sizes than gradient descent based on horospherical smoothness: see Section~\ref{sec:GD}.
			\begin{remark}[Horospherical smoothness as $c$-concavity]
				\label{rmk:c-concavity}
				Let $X$ and $Y$ be sets, and suppose a cost function $c:X\times Y\to\mathbb{R}\cup\{\infty\}$ is given. A function $f:X\to\mathbb{R}$ is said to be $c$-concave if there is a function $h:Y\to\mathbb{R}\cup\{\infty\}$ such that $f(x)=\inf_{y\in Y}\{c(x,y)+h(y)\}$ for all $x\in X$.
				The notion of $c$-concavity arises in optimal transport~\cite[Ch.~5]{villani2008optimal}, where $X$ and $Y$ are typically probability spaces, and $c(x, y)$ denotes the cost of transporting one unit from $x \in X$ to $y \in Y$.
				
				Consider the particular setting where $X=Y=M$ and $c(x,y)=\frac{L}{2}d(x,y)^2$. If a differentiable function $f:M\to\mathbb{R}$ is $c$-concave, then $f$ is $L$-h-smooth by Proposition~\ref{prop:smooth}~(i). Conversely, suppose that $f:M\to\mathbb{R}$ is a differentiable and $L$-h-smooth function. By using Proposition~\ref{prop:smooth}~(i) again, we can write $f(x)=\inf_{(p,r)\in I}\{\frac{L}{2}d(x,p)^{2}+r\}$ where $I$ is some subset of $M\times\mathbb{R}$. For each $p\in M$, we define $r_{p}:=\inf\{r\in\mathbb{R}:(p,r)\in I\}$. Then, we have $f(x)=\inf_{p\in M}\{\frac{L}{2}d(x,p)^2+r_{p}\}$, and thus $f$ is $c$-concave. Therefore, $c$-concavity in this setting is equivalent to $L$-h-smoothness. 
				
				Recently, \cite{leger2023gradient} analyzed gradient descent under the assumption of $c$-concavity in various settings. In particular, Section~4.4 of their paper analyzes the setting where $c$-concavity is equivalent to $L$-h-smoothness. We note that our Proposition~\ref{prop:smooth}~(ii) is equivalent to their Proposition~4.15~(iii).

			\end{remark}

			\paragraph{Moreau envelope.}
			For a function $f:M\to\mathbb{R}$, its \emph{Moreau envelope} (or the \emph{Moreau-Yosida regularization}) $f_{\lambda}:M\to\mathbb{R}$ with parameter $\lambda>0$, and the corresponding \emph{proximal operator} $\prox_{\lambda f}:M\to M$ are defined as
			\begin{equation}
				\label{eq:moreau}
				\begin{aligned}
					f_{\lambda}(x) & =\inf_{y\in M}\left\{ f(y)+\frac{1}{2\lambda}d(x,y)^{2}\right\}, \\
					\prox_{\lambda f}(x) & =\underset{y\in M}{\arg\min}\left\{ f(y)+\frac{1}{2\lambda}d(x,y)^{2}\right\} .
				\end{aligned}
			\end{equation}
			The Moreau envelope was originally introduced for functions on Euclidean spaces in \cite{moreau1965proximite}. It regularizes a $C^0$-function $f$ into its $C^1$-approximation $f_{\lambda}$, where the degree of smoothness is determined by the parameter $\lambda$. The Moreau envelopes of functions on Hadamard manifolds $M$ are studied in \cite{azagra2006inf}. Further properties of $f_{\lambda}$, closely related to optimization theory, can be found in \cite[Appendix~G]{criscitiello2023curvature}, \cite[Appendices~C,E]{martinez2024convergence}. We review some of the basic properties of $f_{\lambda}$.
			\begin{proposition}[{\cite[Cor.~4.5]{azagra2006inf}, \cite[Lem.~10]{martinez2024convergence}}]
				\label{prop:moreau-basic}\ 
				\begin{enumerate}[label=\textup{(\roman*)}]
					\item If $f$ is g-convex, then $f_{\lambda}$ is also g-convex.\footnote{This is not true if $M$ is a manifold with positive curvature somewhere \cite[\S C.1]{martinez2024convergence}.}
					\item $f_{\lambda}$ has the same infimum and the same set of minimizers as $f$.
					\item The gradient of $f_{\lambda}$ is given by $\nabla f_{\lambda}(x)=-\frac{1}{\lambda}\exp_{x}^{-1}\left(\prox_{\lambda f}(x)\right)$.
				\end{enumerate}
			\end{proposition}
			One of the key properties of Moreau envelopes of functions on $\mathbb{R}^n$ is that $f_{\lambda}$ is always $(1/\lambda)$-smooth. However, for general Hadamard manifolds, the (geodesic) smoothness parameter is not given by $1/\lambda$, and worsens as the curvature of the manifold becomes more negative 
			\cite[Lem.~10]{criscitiello2023curvature}, \cite[Thm.~11]{martinez2024convergence}.
			This discrepancy arises from the fact that geodesic smoothness is too rigid a condition. 
			Instead, a more appropriate notion for measuring smoothness in this context is h-smoothness. 
			The following theorem shows that the operation $f \mapsto f_{\lambda}$ preserves h-convexity (as well as g-convexity), and that its regularizing effect is naturally measured by h-smoothness.
			\begin{theorem}
				\label{thm:moreau}
				Let $f \colon M \to \mathbb{R}$ be a h-convex function. Then, the Moreau envelope $f_{\lambda}$ defined in \eqref{eq:moreau} is h-convex and $(1/\lambda)$-h-smooth.
			\end{theorem}
			The proof of Theorem~\ref{thm:moreau} can be found in Appendix~\ref{app:thm:moreau}.
			That proof relies on the following geometric lemma, which may be of independent interest.
			\begin{lemma}
				\label{lem:moreau}
				Let $x,x',y,y'\in M$. Then, we have
				\[
				B_{x',\exp_{x'}^{-1}(x)}(y')+\frac{1}{2}d(y,y')^{2}+\frac{1}{2}d(x,x')^{2}\geq B_{x',\exp_{x'}^{-1}(x)}(y).
				\]
			\end{lemma}
			
			\subsection{Connection with horospherically convex geometry}
			\label{sec:h-geometry}
			
			In this section, we will see that the notion of h-convexity for functions, introduced in our paper, has a close connection with the existing notion of h-convexity for sets. It is well-known that a function $f: \mathbb{R}^n \to \mathbb{R}$ is convex if and only if its epigraph $\{(x, r) \in \mathbb{R}^n \times \mathbb{R} : f(x) \leq r \}$ is a closed convex set, that is, a set which can be written as the intersection of some closed half-spaces \cite[Thm.~11.5]{rockafellar1997convex}. The following theorem generalizes this fact to the h-convex setting.
			We give $M \times \reals$ the standard product metric, making it a Hadamard manifold as well.
			\begin{theorem}
				\label{thm:set_function}
				For a function $f \colon M\to\mathbb{R}$, the following statements are equivalent:
				\begin{enumerate}[label=\textup{(\roman*)}]
					\item The function $f$ is h-convex on $M$.
					\item The epigraph of $f$ is an intersection of closed horoballs in $M\times\mathbb{R}$.
				\end{enumerate}
			\end{theorem}
			The proof of Theorem~\ref{thm:set_function} can be found in Appendix~\ref{app:thm:set_function}. 
			This result reveals that the notion of h-convex functions 
			is rooted in the existing notion of \emph{horospherically convex} sets \cite{santalo1967horocycles, santalo1968horospheres}, whose (equivalent) definition is provided below.
			\begin{definition}
				\label{def:h-convex_set}
				A closed set $K\subseteq M$ is \emph{horospherically convex} (h-convex) if it is an intersection of closed horoballs. 
			\end{definition}
			We can now state Theorem~\ref{thm:set_function} as follows:
			\begin{center}
				$f$ is an h-convex function on $M$ $\iff$ The epigraph of $f$ is an h-convex set in $M \times \mathbb{R}$.
			\end{center}
			This connection provides a way to study h-convex functions in the language of h-convex sets. For example, Proposition~\ref{prop:basic}~(vi), which states that the supremum of h-convex functions is h-convex, translates into h-convex geometry as follows: the intersection of h-convex sets is h-convex.

			We note that Definition~\ref{def:h-convex_set} is not the only definition of h-convexity for sets. To our knowledge, there are three equivalent definitions of h-convexity for closed sets, as shown below.
			\begin{itemize}
				\item A closed set $K\subseteq M$ is h-convex if it is the intersection of closed horoballs. We follow this definition in our paper. This is a straightforward generalization of the well-known outer characterization of closed convex sets in $\mathbb{R}^n$ \cite[Thm.~11.5]{rockafellar1997convex}.
				
				\item A closed set $K\subseteq M$ is h-convex if, at each boundary point $x \in \bd(K)$, there exists a \emph{supporting horosphere} of $K$. It is common to assume that the interior of $K$ is non-empty when using this definition. With this assumption, this definition is equivalent to the first definition (see Appendix~\ref{app:first-order-as-supporting}). Without such an assumption, it is a weaker condition: a set of two distinct points is h-convex in the second sense but not in the first (see \cite[\S3]{goodwin2024subgradient}).

				\item A closed subset $K$ of the hyperbolic space $\mathbb{H}^n$ is h-convex if, for any $x, y \in K$, every \emph{horocycle segment} joining them is contained in $K$. For this reason, h-convex sets are also referred to as \emph{horocyclically convex} sets. It is known that this definition is equivalent to the previous two definitions (see, for example, \cite[\S4]{gallego2010linear}). 
			\end{itemize}
			It is remarkable that the third definition is stated in an \emph{inner} sense rather than an \emph{outer} sense, unlike the other two definitions. However, this definition is only applicable when $M=\mathbb{H}^n$, and we are not aware of a generalization of this definition to arbitrary Hadamard manifolds.
			We also note that the second definition is directly related to our definition of h-convex functions through Theorem~\ref{thm:set_function}.

			\subsection{Failure of local-to-global principles}
			\label{sec:analysis-ep}
			
			In this subsection, we see that h-convexity is global in nature, unlike g-convexity, which is local in nature. To explain this, we provide counterexamples for which a \emph{local-to-global principle} does not hold. We aim to convince the readers through numerical experiments, rather than rigorous proofs.
			\begin{itemize}
				\item {\bf Geodesic convexity is a local property but horospherical convexity is not.} It is known that convexity is a local property in the following sense: if for every $x \in \mathbb{R}^n$, there exists a convex neighborhood $U$ of $x$ such that $f$ restricted to $U$ is convex, then $f$ is convex on $\mathbb{R}^n$ (see, for example, \cite[\S2.1]{hormander2007notions}). A similar conclusion holds for g-convexity. Instead of providing a rigorous proof, we provide an easy explanation to illustrate this: a twice-differentiable function $f$ is g-convex if and only if its Hessian is positive semidefinite everywhere, which is clearly a local property. In contrast, h-convexity is \emph{not} a local property in this sense. Specifically, in Appendix~\ref{app:convexity_local_property} we exhibit a function which is locally h-convex but not globally h-convex.
				
				\item {\bf Geodesic subdifferential is a local concept but horospherical subdifferential is not.} Recall that when $M=\mathbb{R}^n$, the subdifferential $\partial f(y)$ is defined as the set of $v \in \mathbb{R}^n$ such that $\langle v, d \rangle \leq f'(y, d)$ for all $d$, where $f'(y,v)=\lim_{t \to 0^{+}} \frac{f(y+tv)-f(y)}{t}$ (see, for example, \cite[\S IV.1]{hiriart1996convex}). Note that this definition is given in a local sense since $f'(y,d)$ is entirely determined by the local behavior of $f$ near $y$. Replacing $y+tv$ with $\exp_y(tv)$ yields the definition of the geodesic subdifferential (see, for example, \cite[\S3.4]{udriste1994convex}). On the other hand, there is also a global way to define the set $\partial f(y)$: it is the set of $v\in\mathbb{R}^n$ such that $f(x) \geq f(y) + \langle v, x-y \rangle$ for all $x \in \mathbb{R}^n$. This definition is directly generalized by our definition of the h-subdifferential.\footnote{The geodesic subdifferential can be also defined as $\partial f(y)=\{v\in T_{y}M:f(x)\geq f(y)+\langle v,\log_{y}(x)\rangle\}$, but this characterization is not very favorable since the function $x\mapsto \langle v,\log_{y}(x)\rangle$ is not g-convex (see the discussion below Proposition~\ref{prop:E-outer}).} In Appendix~\ref{app:subdifferential_local_concept}, we show that the horospherical subdidfferential $\ph f(y)$ \emph{cannot} be defined in a local sense.  
				
			\end{itemize}
			
			\section{Horospherically convex optimization}
			\label{sec:h-opt}
			
			\begin{table}
				\centering
				\begin{tabular}{|c|c|} 
					\hline 
					Our guarantee for Hadamard manifolds $M$  & Known guarantee for Euclidean spaces $\mathbb{R}^n$ \\
					\hline
					Thm.~\ref{thm:gd-c} & \cite[Thm.~3.3]{bansal2019potential} \\
					Thm.~\ref{thm:gd-sc} & \cite[Thm.~3.8]{bansal2019potential} \\
					Thm.~\ref{thm:subg-descent-c} & \cite[Thm.~3.2]{bubeck2015convex} \\
					Thm.~\ref{thm:subg-descent-sc} & \cite[\S3.2]{lacoste2012simpler} \\
					Thm.~\ref{thm:agm-c} & Slight variant of \cite{nesterov1983method} \\
					Thm.~\ref{thm:agm-sc} & \cite[\S2.2.1]{nesterov2018lectures}\\
					\hline
				\end{tabular}
				\caption{Comparison between our results and known convergence guarantees.}
				\label{table:rates}
			\end{table}
			
			Since h-convexity is not preserved under addition (see Appendix~\ref{app:rmk:basic}), in this section we consider minimizing a \emph{sum} of h-convex functions. Formally, we study the problem
			\begin{equation}
				\label{eq:problem}
				\min_{x\in M} \; f(x) = \frac{1}{m}\sum_{i=1}^{m} f_i(x),
			\end{equation}
			where each $f_i$ is an h-convex function on $M$. We refer to~\eqref{eq:problem} as an \emph{h-convex optimization problem}.
			
			In Section~\ref{sec:prob_example}, we give examples of h-convex optimization.
			Then we develop first-order algorithms for h-convex optimization (Sections~\ref{sec:GD},~\ref{sec:momentum}). Our resulting convergence guarantees exactly match their Euclidean counterparts, as summarized in Table~\ref{table:rates}.
			To prove convergence rates, we need additional regularity conditions, such as a Lipschitz constant or a smoothness constant. We summarize below the conditions that we will use.
			\begin{itemize}
				\item {\bf Gradient descent (Section~\ref{sec:smooth-gd}).} We assume that each $f_i$ is $L$-h-smooth (see Defintion~\ref{eq:defining-smooth}).  Note that this does \emph{not} imply that $f$ is $L$-h-smooth, or that $f$ satisfies the descent condition (unless $m=1$, Proposition~\ref{prop:smooth}). 
				\item {\bf Subgradient descent (Section~\ref{sec:subg-des}).} We assume that $f$ is $L$-Lipschitz continuous, meaning that $|f(x) - f(y)| \leq Ld(x, y)$. 
				Note if \emph{each} $f_i$ is $L$-Lipschitz, then $f$ is $L$-Lipschitz.    
				
				\item {\bf Accelerated gradient method (Section~\ref{sec:momentum}).} We assume that $f$ satisfies the {descent condition}: $f(x^{+})\leq f(x)-\frac{1}{2L}\|\nabla f(x)\|^{2}$, where $x^+=\exp_x(-\frac{1}{L}\nabla f(x))$.  In particular, it is enough to assume $f$ is $L$-h-smooth if $m=1$ (Propositon~\ref{prop:smooth}).
				It is an open question whether one can extend our results to the case where \emph{each} $f_i$ is $L$-h-smooth.
			\end{itemize}
			For simplicity, we assume $f_1, \ldots, f_m$ have the same Lipschitz constants.  However, we expect our algorithms and analysis extend to non-uniform Lipschitz constants with little effort.
			
			\subsection{Examples}
			\label{sec:prob_example}
			
			In this subsection, we investigate examples of (strongly) h-convex optimization problems. From Theorem \ref{thm:envelope_rep_h-convex}, the following are straightforward: scaled Busemann functions $B_{p,v}(\cdot)$ are h-convex, and squared distance functions $\frac{1}{2}d(\cdot,p)^2$ are $1$-strongly h-convex. More elementary examples come from the following proposition.
			\begin{proposition}
				\label{prop:d-is-conv}
				Let $p \in \calM$. Then, the distance function $y \mapsto d(y, p)$ is h-convex.
			\end{proposition}
			The proof of Proposition~\ref{prop:d-is-conv} can be found in Appendix~\ref{squareddistancefunctionsareBconvex}. Using Proposition~\ref{prop:basic}~(vii) with $f(x)=d(x,p)$ and $g(t)=\frac{1}{p}\max(t,0)^p$, we also have that $\frac{1}{p}d(\cdot,p)^p$ is h-convex. 
			\paragraph{Finding a representative point.}
			The fundamental problem of finding a representative point for a cluster of points $p_1, \ldots, p_N \in M$ arises in various fields such as data analysis and operations research.
			\begin{itemize}
				\item The \emph{minimal enclosing ball} problem (or the \emph{1-center} problem) \cite{elzinga1972minimum,arnaudon2013approximating} can be formulated as \eqref{eq:problem-intro} with $f(x) = \sup_{i} d(x, p_{i})$, which is h-convex with $m=1$.\footnote{It is straightforward to generalize this problem to $\min_x f(x)=\sup_{p\in C} d(x,p)$, where $C$ is any closed set in $M$.} 
				\item The problem of computing the \emph{Fr\'echet mean} (or the \emph{Karcher mean}) \cite{grove1973conjugate} can be formulated as \eqref{eq:problem} with $f(x) = \sum_{i=1}^{m} d(x, p_{i})^{2}$, which is the sum of strongly h-convex functions. 
				\item The problem of computing the geometric median an be formulated as \eqref{eq:problem} with $f(x) = \sum_{i=1}^{m} d(x, p_{i})$, which is the sum of h-convex functions. In location theory, this problem is also called the \emph{Weber problem} \cite{weber1922ueber}.
			\end{itemize}
			Recognizing whether a point $p$ is a weighted mean of a set $x_1, x_2, \ldots x_k \in M$ can also be cast as an h-convex problem with $m=1$~\cite{goodwin2024recognizingweightedmeansgeodesic}.

			\paragraph{Tyler's M-estimator, and the angular distribution.} 
			Consider $m$ unobserved vectors $y_1, \ldots, y_m \in \mathbb{R}^d$, sampled independently from a centered Gaussian distribution with unknown covariance matrix $\Sigma^* \succ 0$.
			Each vector is then normalized to lie on the unit sphere via $x_i = \frac{y_i}{\|y_i\|}$.
			Given only the normalized samples $x_1, \ldots, x_m$, how can we estimate the underlying covariance $\Sigma^*$?
			
			The samples $x_i$ follow the so-called \emph{angular distribution}, parameterized by a positive definite matrix $\Sigma^*$ which is defined up to scale (so we may assume without loss of generality $\det(\Sigma^*) = 1$).
			The negative log-likelihood is given by
			\[
			\ell(\Sigma) = \frac{1}{m} \sum_{i=1}^m d \log(x_i^\top \Sigma^{-1} x_i), \qquad \Sigma \in \mathrm{SPD}(n),
			\]
			An MLE (i.e., a minimizer of $\ell$) is known as \emph{Tyler's M-estimator}, and serves as a \emph{robust} method for covariance estimation.
			Tyler~\cite{tyler1987distribution} showed that it is the most robust estimator for elliptical distributions.
			The negative log-likelihood $\ell$ is not convex in the Euclidean sense, but becomes geodesically convex when $\mathrm{PD}(n)$ is endowed with the affine-invariant metric (see Example~\ref{examplePD}).
			In fact, \emph{each summand $\log(x_i^\top \Sigma^{-1} x_i)$ is a scaled Busemann function}, so the optimization problem $\min_{\Sigma \in \mathrm{SPD}(n)} \ell(\Sigma)$ is an instance of h-convex optimization.
			
			Tyler's M-estimator is a special case of \emph{Grassmannian M-estimators of scatter}, whose negative log-likelihoods are also sums of Busemann functions~\cite{ciobotaru2018geometrical}.
			Auderset, Mazza and Ruh~\cite{auderset2005angular} were the first to use the geometry of $\mathrm{PD}(n)$ to study the properties of Tyler's M-estimator.
			See~\cite{wiesel2015gconvexity} for further applications of geodesic convexity in robust covariance estimation, and~\cite{franks2021neartyler} for connections to scaling problems.
			
			\paragraph{Horn's problem.}
			Given $m \geq 3$ vectors $\lambda_1, \ldots, \lambda_m \in \mathbb{R}^n$, when do there exist Hermitian matrices $A_1, \ldots, A_m \in \mathbb{C}^{n \times n}$ such that $\sum_{i=1}^m A_i = 0$ and the spectrum of $A_i$ equals $\lambda_i$ for each $i$? This question, known as \emph{Horn's problem}, has a long history in pure mathematics. It was first studied by Horn~\cite{Horn1962}. General results from symplectic geometry imply that such matrices $A_i$ exist if and only if the vectors $\lambda_i$ satisfy a finite system of linear inequalities~\cite{Kirwan1984,f5ed6840-1e83-3412-9379-8a84ba9e445d}. The set of inequalities was precisely characterized by Klyachko~\cite{klyachko1998stable} and by Knutson and Tao~\cite{knutson1999honeycomb}.
			Taking an algorithmic perspective, can one \emph{efficiently} decide whether such Hermitian matrices $A_i$ exist, and find them when they do?
			The number of inequalities is exponentially large, so directly checking them is inefficient.
			
			Remarkably, {this problem can be recast as minimizing a sum of scaled Busemann functions}. Assume that each $\lambda_i$ has rational entries, and sample random unitary matrices $U_i \in \mathrm{U}(n)$. Let $b_i$ be the Busemann function on $\mathrm{PD}(n)$ associated to the geodesic $t \mapsto \exp(t U_i \Lambda_i U_i^*)$, where $\Lambda_i$ is the diagonal matrix with diagonal $\lambda_i$. Define
			\[
			f(P) = \sum_{i=1}^m b_i(P), \quad \text{for } P \in \mathrm{PD}(n).
			\]
			Then matrices $A_i$ satisfying Horn's condition exist if and only if $f$ is bounded below, in which case a minimizer of $f$ encodes the desired matrices $A_i$.
			One way to see this connection to geodesic convexity is by expressing Horn's problem within the noncommutative optimization framework of~\cite{burgissernoncommutativeoptimization}.\footnote{Specifically, let $O_\lambda$ denote the set of Hermitian matrices with spectrum $\lambda$, and consider the diagonal action of $U \in \mathrm{U}(n)$ on $\prod_{i=1}^m O_{\lambda_i}$ given by $U \cdot (A_1, \ldots, A_m) = (U A_1 U^*, \ldots, U A_m U^*)$. This can be lifted to a $\mathrm{GL}(n)$-action on a projective space via the Borel–Weil–Bott–Kostant theorem; see Section 7 of~\cite{KNUTSON200061} where an explicit lift is given. Using this lift and~\cite[Lem.~2.32]{hirai2023convex}, one finds that the associated Kempf–Ness function is exactly the sum of scaled Busemann functions $f$.} To the best of our knowledge, this geometric method for studying Horn's problem was pioneered by Kapovich, Leeb and Millson~\cite{kapovich2009convex}. 
			
			B\"urgisser and Ikenmeyer~\cite{burgisserdecidinghorn} give a polynomial-time algorithm for the decision version of Horn's problem. However, it remains open whether the \emph{search} version (i.e., find $A_i$) admits a polynomial-time algorithm. 
			The main obstacle, discussed in Section~\ref{sec:intro}, is that known diameter bounds are exponentially large.
			Moreover, the iteration complexity of all existing high-precision algorithms for geodesically convex optimization depends polynomially on this diameter~\cite{burgissernoncommutativeoptimization}.
			
			\subsection{Gradient descent}\label{sec:GD}
			
			In this subsection, we study a generalization of (sub)gradient descent to solve the problem \eqref{eq:problem}. 
			If each $f_i$ is h-$\frac{1}{s}$-smooth, then we have the upper bound $f(y) \leq f(x) + \frac{1}{m} \sum_{i=1}^m Q_{x, \nabla f_i(x)}^{1/s}(y)$. 
			It makes sense to minimize this upper bound, like follows:
			\begin{equation}
				\label{eq:gd}
				x_{k+1} = \argmin_{x\in M}\left\{ \frac{1}{m}\sum_{i=1}^{m}Q_{x_{k}, g_{ik}}^{1/s_k}\left(x\right)\right\} , \quad  \text{where }g_{ik} \in \ph_{\mu} f_i(x_k).
			\end{equation}
			
			When the objective function is a single h-convex function ($m=1$), this algorithm simplifies to $x_{k+1} = \exp_{x_k}(-s_k g_k)$, i.e., standard Riemannian (sub)gradient step with exponential retraction.
			The subproblem in~\eqref{eq:gd} corresponds to computing a Fr\'echet mean (with uniform weights), and only requires access to a subgradient of each $f_i$ at $x_k$.
			
			\subsubsection{Gradient descent for smooth functions}
			\label{sec:smooth-gd}
			
			In this subsection, we analyze the convergence rates of the gradient descent method \eqref{eq:gd} for unconstrained smooth h-convex optimization problems.   
			The following theorems extend the known convergence guarantees of gradient descent in the convex setting.

			\begin{theorem}
				\label{thm:gd-c}
				If each $f_i$ is h-convex and $L$-h-smooth, the gradient descent method \eqref{eq:gd} with $s_k=1/L$ satisfies
				\[
				f(x_N)-f(x^*) \leq \frac{L}{2N}d(x_0,x^*)^2.
				\]
			\end{theorem}
			
			\begin{theorem}
				\label{thm:gd-sc}
				If each $f_i$ is $\mu$-strongly h-convex and $L$-h-smooth, the gradient descent method \eqref{eq:gd} with $s_k=1/L$ satisfies
				\begin{align} 
					\label{eq:first_result}
					f(x_{N})-f(x^{*}) + \frac{\mu}{2}d(x_{N},x^{*})^{2} &\leq\left(1-\frac{\mu}{L}\right)^{N}\left(f(x_{0})-f(x^{*})+\frac{\mu}{2}d(x_{0},x^{*})^{2}\right), \\
					\label{eq:second_result}
					f(x_{N})-f(x^{*}) &\leq\left(1-\frac{\mu}{L}\right)^{N}\left(f(x_{0})-f(x^{*})\right).
				\end{align}
			\end{theorem}
			The first result~\eqref{eq:first_result} relies on the \emph{horospherical law of cosines} (Lemma~\ref{lem:horotriangle}). 
			The proof of the second result~\eqref{eq:second_result} uses the following non-obvious lemma, which may be of independent interest:
			
			\begin{lemma}\label{prop:nonincreasing}
				Let $y \in M$ and $v_i \in T_y M$ for $i=1, \ldots, m$.
				The univariate function 
				$$h(L) := \min_{x \in M} \bigg\{ \frac{1}{m} \sum_{i=1}^m L^2 d\Big(x, \exp_y\big(-\frac{1}{L} v_{i}\big)\Big)^2 \bigg\}, \quad \quad L > 0$$
				is non-increasing in $L$.  (If $M = \reals^d$, then $h$ is constant.)
			\end{lemma}
			
			The proofs of Theorem~\ref{thm:gd-c}, Theorem \ref{thm:gd-sc} and Lemma~\ref{prop:nonincreasing} can be found in Appendices~\ref{app:thm:gd-c}, \ref{app:thm:gd-sc} and \ref{app:prop:nonincreasing}, respectively.

			\subsubsection{Subgradient descent for non-smooth functions}
			\label{sec:subg-des}
			
			In this subsection, we consider problems where the objective function $f$ is not necessarily differentiable. 
			We assume that a constraint $x \in C$ is given, where $C$ is a compact and g-convex set in $M$ whose diameter is $D$. 
			We consider the following \emph{projected subgradient  method}: 
			Given $x_k$, first update $x'_{k+1}$ using \eqref{eq:gd}. Then project $x'_{k+1}$ back to $C$ as $x_{k+1} =\mathcal{P}_{C}(x_{k+1}')$, where $\mathcal{P}_C$ is the metric projection onto $C$. 
			\begin{theorem}
				\label{thm:subg-descent-c}
				If each $f_i$ is h-convex and $f$ is $L$-Lipschitz on $C$, then the projected subgradient method with $s_{k}=\frac{D}{L\sqrt{N+1}}$ satisfies
				\[
				f\left(\bar{x}_{N}\right)-f\left(x^{*}\right)\leq\frac{DL}{\sqrt{N+1}},
				\]
				where $\bar{x}_{0}=x_{0}$ and $\bar{x}_{k+1}=\exp_{\bar{x_{k}}}(\frac{1}{k+2}\exp^{-1}_{\bar{x}_{k}}(x_{k+1}))$.
			\end{theorem}
			\begin{theorem}
				\label{thm:subg-descent-sc}
				When each $f_i$ is $\mu$-strongly convex and $f$ is $L$-Lipschitz on $C$, the projected subgradient method with $s_{k}=\frac{2}{\mu(k+2)}$ satisfies
				\[
				f\left(\bar{x}_{N}\right)-f\left(x^{*}\right)\leq\frac{2L^{2}}{\mu(N+2)}.
				\]
				where $\bar{x}_{0}=x_{0}$ and $\bar{x}_{k+1}=\exp_{\bar{x_{k}}}(\frac{2}{k+3}\exp^{-1}_{\bar{x}_{k}}(x_{k+1}))$.
			\end{theorem}
			\noindent The proofs of Theorems~\ref{thm:subg-descent-c} and \ref{thm:subg-descent-sc} can be found in Appendices~\ref{app:thm:subg-descent-c} and \ref{app:thm:subg-descent-sc}, respectively. 
			They rely on the \emph{horospherical law of cosines} (Lemma~\ref{lem:horotriangle}).
			
			To find an $\epsilon$-approximate solution $\bar{x}_N$, we used geodesic averaging (i.e., averaging two points iteratively). This technique was used in the prior work \cite{zhang2016first}. See also \cite[Lem.~27]{martinez2024convergence} for a general result on this technique.
			
			\begin{remark}[Larger step sizes in h-convex optimization]
				Algorithms for h-convex optimization admit larger step sizes than their g-convex counterparts, which leads to faster convergence rates. 
				For instance, for Lipschitz g-convex optimization in a ball of radius $r$, step sizes of the form $s_k = \frac{1}{\zeta} \cdot \frac{D}{L \sqrt{N+1}}$
				are used~\cite{zhang2016first}, where 
				$\zeta = \frac{r \sqrt{-K}}{\tanh(r \sqrt{-K})} \sim r \sqrt{-K}$,
				and $K$ is a lower bound on the sectional curvature of $M$.
				In contrast, the analogous $h$-convex setting allows for the larger step size 
				$s_k = \frac{D}{L \sqrt{N+1}}$,
				as shown in Theorem~\ref{thm:subg-descent-c}.
			\end{remark}
			
			\begin{remark}[Approximately solving the subproblems]
				\label{rem:bitcomplexity}
				It is straightforward to extend Theorems~\ref{thm:gd-c},~\ref{thm:gd-sc},~\ref{thm:subg-descent-c}, and~\ref{thm:subg-descent-sc} to the setting where we only approximately solve the subproblems~\eqref{eq:gd} up to some additive accuracy (while maintaining the curvature-independent rates). That is, we allow $x_{k+1}$ to satisfy
				\[
				\frac{1}{m} \sum_{i=1}^{m} Q_{x_k, g_{ik}}^{1/s_k}(x_{k+1}) 
				\leq \delta + \min_{x \in M} \left\{ \frac{1}{m} \sum_{i=1}^{m} Q_{x_k, g_{ik}}^{1/s_k}(x) \right\}.
				\]
				
				However, we emphasize that the optimization of $h$-convex functions is particularly sensitive to noise in the gradients. In particular, the lower bound of Hamilton and Moitra~\cite{hamilton2021nogo} shows that minimizing an $h$-convex function over a ball of radius $r$ in hyperbolic space requires computing each gradient $\nabla f_i$ to precision $O(e^{-r})$, i.e., with $O(r)$ bits of accuracy.
				
			\end{remark}

			\subsection{Accelerated first-order methods using Nesterov's momentum}
			\label{sec:momentum}
			
			In this subsection, we generalize Nesterov's accelerated gradient method \cite{nesterov1983method, nesterov2018lectures}, a first-order method that achieves the optimal oracle complexity up to a constant factor. The convergence guarantees we obtain here are independent of the sectional curvature bound and are exactly the same as the known convergence rate in the Euclidean case ($M=\mathbb{R}^n$). 
			\paragraph{Non-strongly convex case.}
			When $\mu=0$, the accelerated gradient method~\cite{nesterov2018lectures} is given by
			\begin{equation}
				\label{eq:E-agm-c}
				\begin{aligned}
					y_{k} & =x_{k}+\frac{2}{k+1}(z_{k}-x_{k}),\\
					x_{k+1} & =y_{k}-\frac{1}{L}\grad f(y_{k}),\\
					z_{k+1} & =z_{k}-\frac{k+1}{2L}\grad f(y_{k}),
				\end{aligned}
			\end{equation}
			starting at an initial point $x_0=z_0$. To solve the h-convex optimization problem \eqref{eq:problem}, we propose the following generalization of this algorithm:
			\begin{equation}
				\label{eq:agm-c}
				\begin{aligned}
					y_{k} & =\exp_{x_{k}}\left(\frac{2}{k+1}\exp^{-1}_{x_{k}}(z_{k})\right),\\
					x_{k+1} & =\exp_{y_{k}}\left(-\frac{1}{L}\grad f(y_{k})\right),\\
					z_{k+1} & =\argmin_{z\in M}\left\{ \frac{1}{2}d(z,z_{k})^{2}+\frac{k+1}{2L}\frac{1}{m}\sum_{i=1}^{m}B_{y_k,\grad f_{i}(y_{k})}\left(z\right)\right\} .
				\end{aligned}
			\end{equation}
			In particular, when $m=1$, the updating rule for $z_k$ simplifies to 
			\begin{align}\label{eq:agm-c2}
				z_{k+1}=\exp_{z_{k}}\left(-\frac{k+1}{2L}\grad B_{y_{k},\grad f(y_{k})}(z_{k})\right).
			\end{align}
			\begin{theorem}
				\label{thm:agm-c}
				Assume $f$ satisfies the {descent condition}: $f(x^{+})\leq f(x)-\frac{1}{2L}\|\nabla f(x)\|^{2}$, where $x^+=\exp_x(-\frac{1}{L}\nabla f(x))$.
				Then algorithm \eqref{eq:agm-c} satisfies
				\[
				f(x_{N})-f(x^{*})\leq\frac{2L}{N^{2}}d(x_{0},x^{*})^{2}.
				\]
			\end{theorem}
			The proof, which can be found in Appendix~\ref{app:thm:agm-c}, consists of showing that the energy function
			$\mathcal{E}_{k}=\frac{1}{2}d(x^{*},z_{k})^{2}+\frac{k^{2}}{4L}(f(x_{k})-f(x^{*}))$
			is non-increasing.
			
			\paragraph{Strongly convex case.}
			When $\mu>0$, the accelerated gradient method~\cite{nesterov2018lectures} is given by
			\begin{equation}
				\label{eq:E-agm-sc}
				\begin{aligned}
					y_{k} & =x_{k}+\frac{\sqrt{\mu/L}}{1+\sqrt{\mu/L}}(z_{k}-x_{k}),\\
					x_{k+1} & =y_{k}-\frac{1}{L}\grad f(y_{k}),\\
					z_{k+1} & =\left(1-\sqrt{\frac{\mu}{L}}\right)z_{k}+\sqrt{\frac{\mu}{L}}\left(y_{k}-\frac{1}{\mu}\grad f(y_{k})\right),
				\end{aligned}
			\end{equation}
			starting at an initial point $x_0=z_0$. To solve the h-convex optimization problem \eqref{eq:problem}, we propose the following generalization of this algorithm:
			\begin{equation}
				\label{eq:agm-sc}
				\begin{aligned}
					y_{k} & =\exp_{x_{k}}\left(\frac{\sqrt{\mu/L}}{1+\sqrt{\mu/L}}\exp^{-1}_{x_{k}}(z_{k})\right)\\
					x_{k+1} & =\exp_{y_{k}}\left(-\frac{1}{L}\grad f(y_{k})\right)\\
					z_{k+1} & =\argmin_{z\in M}\left\{ \left(1-\sqrt{\frac{\mu}{L}}\right)\frac{\mu}{2}d(z,z_{k})^{2}+\sqrt{\frac{\mu}{L}}\frac{1}{m}\sum_{i=1}^{m} Q^{\mu}_{y_k,\nabla f_i(y_k)}(z) \right\} .
				\end{aligned}
			\end{equation}
			Hence, \( z_{k+1} \) is the solution to a \emph{weighted} Fr\'echet mean problem.
			In particular, when $m=1$, the updating rule for $z_k$ simplifies to 
			\begin{align}\label{eq:agm-sc2}
				z_{k+1}=\exp_{z_{k}}\left(\sqrt{\frac{\mu}{L}}\exp^{-1}_{z_{k}}\left(y_{k}^{++}\right)\right), \quad  \textup{where $y_{k}^{++}=\exp_{y_{k}}\left(-\frac{1}{\mu}\grad f(y_{k})\right)$.}
			\end{align}
			\begin{theorem}
				\label{thm:agm-sc}
				Assume $f$ satisfies the {descent condition}: $f(x^{+})\leq f(x)-\frac{1}{2L}\|\nabla f(x)\|^{2}$ for all $x\in M$, where $x^+=\exp_x(-\frac{1}{L}\nabla f(x))$.
				Also assume each $f_i$ is $\mu$-strongly h-convex.
				Then algorithm \eqref{eq:agm-sc} satisfies
				\[
				f(x_{N})-f(x^{*})\leq\left(1-\sqrt{\frac{\mu}{L}}\right)^{N}\left(f(x_{0})-f(x^{*})+\frac{\mu}{2}d(x_{0},x^{*})^{2}\right).
				\]
			\end{theorem}
			The proof, which can be found in Appendix~\ref{app:thm:agm-sc}, consists of showing that the energy function
			$\mathcal{E}_{k}=\left(1-\sqrt{\frac{\mu}{L}}\right)^{-k}\left(f(x_{k})-f(x^{*})+\frac{\mu}{2}d(z_{k},x^{*})^{2}\right)$
			is non-increasing.
			\begin{remark}
				We observe that Nesterov's acceleration scheme for $m=1$ (equations~\eqref{eq:agm-sc} and~\eqref{eq:agm-sc2}) is best interpreted as a ``four-point'' algorithm, involving the points $y_k$, $x_{k+1} = y_k^+$, $y_k^{++}$, and $z_{k+1}$. 
				The non-strongly h-convex variant (equations~\eqref{eq:agm-c} and~\eqref{eq:agm-c2}) can also be viewed as a four-point method, with the distinction that one of the points lies on the boundary at infinity.
			\end{remark}

			\section{Much faster rates in hyperbolic space}\label{sec:faster_rates}
			
			In this section we show that if $M$ has sufficiently negative curvature, then subgradient descent for minimizing an h-convex function (i.e., the problem~\eqref{eq:problem} with $m=1$) converges significantly \emph{faster} than in Euclidean space.
			We also give a curvature-independent ``ellipsoid'' method on hyperbolic space.
			For simplicity we focus on hyperbolic space $M = \mathbb{H}^n$ with constant sectional curvature $-1$, although we expect that most results extend to Hadamard manifolds with upper curvature bound.
			
			As a warm up, we start with the following proposition, which follows from a simple geometric analysis of the subgradient method (and in particular is a consequence of bounding a sublevel set of $f$ by a well-chosen supporting half-space\footnote{By half-space, we mean $\{x\in M:\langle v,\exp^{-1}_y(x)\rangle \leq 0\} \subseteq M$ for $(y, v) \in T M$.}).
			For simplicity, we consider optimization over closed balls.
			The proof of the following result can be found in Appendix~\ref{app:faster-hyper2}.
			\begin{proposition}\label{prop:thisone}
				Let $f$ be h-convex and $L$-Lipschitz on $\mathbb{H}^n$, and let $f^* = \min_{x \in \bar{B}(p, r)} f(x)$.
				Let $\delta \in (0,r)$ ($\delta$ has the same units as $r$). 
				Then the subgradient method
				$$x_{k+1} = \exp_{x_k}(- s_k g_k), \quad{g_k\in \ph f(x_k)}, \quad s_k = \delta \frac{g_k}{\|g_k\|}, \quad x_0 = p,$$
				satisfies
				$\min_{k = 0, 1, \ldots, N} \{f(x_k) - f^*\} \leq L \delta$
				provided $N \geq \frac{\log \cosh(r)}{\log \cosh(\delta)}$.\footnote{Note that $\frac{\log \cosh(r)}{\log \cosh(\delta)} \leq \frac{r^2}{\delta^2}$, since $r\mapsto \frac{\log \cosh(r)}{r^2}$ is decreasing.}
				In particular, when $\delta \leq 1$, it is enough to assume $N \geq \frac{4}{r} \cdot \frac{r^2}{\delta^2}$.
			\end{proposition}
			Compare this with the same setting in Euclidean space, where $\Theta(\frac{r^2}{\delta^2})$ iterations of subgradient descent are necessary (and this is tight)~\cite[\S7]{nemirovskibook}.
			We see that the h-convex rate is faster than the convex rate by at least a factor $r$.
			For g-convex optimization, $\Theta(\frac{r^3}{\delta^2})$ iterations are necessary when $r \geq \Omega(1)$, and this is tight~\cite[Thm.~7, Thm.~8, Cor.~14]{criscitiello2023curvature}. Therefore h-convex optimization is faster by at least a factor $r^2$.
			
			When $r \geq \Omega(1)$, the previous proposition is far from tight, and we can do much better by bounding the sublevel sets of $f$ with a supporting horoball instead of a supporting half-space.
			The following lemma states that only $\log(r)$ iterations of subgradient descent are necessary to localize the set of minimizers $\argmin_{x \in \bar{B}(p, r)} f(x)$ to a ball of radius $O(1)$.
			Contrast this with the Euclidean case, where one needs strong convexity to localize the set of minimizers in a similar manner.
			\begin{proposition}\label{prop:anotherone}
				Let $f$ be h-convex and $L$-Lipschitz on $\mathbb{H}^n$, and let $f^* = \min_{x \in \bar{B}(p, r)} f(x)$, with $r \geq 4$.
				Define the subgradient method
				$$x_{k+1} = \exp_{x_k}(- s_k g_k), \quad g_k\in \ph f(x_k), \quad s_k = \frac{r e^{-k/4}}{2} \frac{g_k}{\|g_k\|}, \quad x_0 = p.$$
				Setting $N =\lceil4\log(r/4)\rceil$, we obtain 
				$\arg\min_{x\in \bar{B}(p, r)}f(x) \subseteq \bar{B}(x_N, 4)$.
			\end{proposition}
			The idea behind the proof of Proposition~\ref{prop:anotherone} (see Appendix~\ref{app:faster-hyper3}) is as follows. 
			We begin with the assumption that the minimizers of $f$ lie within the radius-$r$ ball $B(x_0, r)$. 
			A subgradient $g_0 \in T_{x_0} M$ implies that the minimizers are contained in the intersection of $B(x_0, r)$ and a horoball whose boundary contains $x_0$.
			Using the specific geometry of hyperbolic space, we can bound that intersection by a smaller ball centered at $x_1 = \exp_{x_0}(\frac{r}{2} \frac{g_0}{\|g_0\|})$ with radius $r' = 1 + \frac{r}{2}$, which is less than $r e^{-1/4}$ if $r \geq 4$.
			We then repeat the process iteratively.
			In contrast, for the corresponding setup in Euclidean space, the best we can do is take $r' = r$. 
			\paragraph{An ``ellipsoid method'':}
			In~\cite{criscitiello2023open}, the authors show that in the Beltrami-Klein model of hyperbolic space, g-convex functions on $\mathbb{H}^n$ become quasi-convex in the Euclidean sense.  
			Using this, they are able to reduce Lipschitz g-convex optimization in a ball of radius $O(1)$ to Lipschitz quasi-convex Euclidean optimization.  The latter such problems can be solved by any cutting plane method, like the ellipsoid method.
			Pairing this observation with Proposition~\ref{prop:anotherone}, we obtain a curvature-independent algorithm for minimizing $f \colon \mathbb{H}^n \to \reals$ h-convex and Lipschitz.
			
			Specifically, the algorithm begins by performing $N = \lceil4\log(r/4)\rceil$ subgradient steps as in Proposition~\ref{prop:anotherone}, thereby localizing the set of minimizers to a constant-sized ball \( \bar{B}(x_N, 4) \).
			Let \( \Phi \colon B(0,1) \to \mathbb{H}^n \) denote the diffeomorphism given by the Poincaré ball model\footnote{For g-convex functions, the Beltrami–Klein model is required. For h-convex functions, either the Beltrami–Klein or Poincaré ball models suffice; we adopt the latter for consistency.} which maps the origin to $x_N$, (i.e., $\Phi(0) = x_N $). 
			We then apply the \emph{Euclidean} ellipsoid method to the function \( y \mapsto (f \circ \Phi)(y) \), subject to the constraint \( y \in \Phi^{-1}(B(x_N, 4))\), which is just a Euclidean ball.
			The Euclidean ellipsoid method makes $O(d^2 \log(\frac{1}{\delta}))$ subgradient queries to find a point $y$ such that $(f \circ \Phi)(y) - f^* \leq L \delta$~\cite[Ch.~3]{nemirovskibook}.
			
			To ensure the correctness of the ellipsoid method, we require that \( f \circ \Phi\) is quasi-convex and $O(L)$-Lipschitz, both in the Euclidean sense. Indeed:
			\begin{itemize}
				\item The function \( f \circ \Phi \) is Euclidean quasi-convex. To show this, let $C$ be a sublevel set of \( f \circ \Phi \). Then, $\Phi(C)$ is a sublevel set of $f$. This set is h-convex, meaning it is an intersection of horoballs. Since the inverse images of these horoballs under $\Phi$ are Euclidean balls in the Poincaré ball model, the set $C$ is an intersection of Euclidean balls, making it convex in the Euclidean sense.
				
				\item While \( f \circ \Phi \) may not be globally Lipschitz over the entire Euclidean unit ball due to distortion from \( \Phi \), it is \( O(L) \)-Lipschitz over the constant-sized region \( \Phi^{-1}(B(x_N, 4)) \). This follows from {the fact that the norm of the differential of $\Phi$ is bounded by a constant on that region.} See also~\cite{martinez2022global} for controlling distortion in the Beltrami–Klein model. 
				
			\end{itemize}
			The complexity of the algorithm is as follows:
			\begin{proposition}\label{prop:ellipsoidHn}
				Let $f$ be h-convex and $L$-Lipschitz on $\mathbb{H}^n$, and let $f^* = \min_{x \in \bar{B}(p, r)} f(x)$.
				For $\delta \in (0, 1)$, there is an algorithm which finds $x \in M$ satisfying $f(x) - f^* \leq L \delta$ using at most
				$$O\Big(\log(r) + d^2 \log\Big(\frac{1}{\delta}\Big)\Big)$$
				subgradient queries.
			\end{proposition}

			In contrast, the ellipsoid method for g-convex optimization given by~\cite{criscitiello2023open} requires~$O(r d^2 \log^2(\frac{r}{\delta}))$ subgradient queries.
			We are able to avoid the extra factor of $r$ in Proposition~\ref{prop:ellipsoidHn} because Proposition~\ref{prop:anotherone} allows us to quickly localize the minimizing set.

			\section{Perspectives}
			We conclude with a list of directions for future research.
			\begin{itemize}
				\item \textbf{H-convexity for scaling problems}:
				The examples from Section~\ref{sec:prob_example} indicate that \emph{some} scaling problems (e.g., Tyler's M-estimator, Horn's problem) can be formulated as the minimization of a sum (or an integral) of h-convex functions.  
				Does tensor scaling and related problems fall into this category?  
				Even if so, it is not immediately clear how to leverage the algorithms we have given to obtain polynomial-time algorithms, for three reasons.
				First, we assume access to an oracle (e.g., solving the subproblem~\eqref{eq:gd}) which may be too strong --- the splitting approach of Goodwin et al.~\cite{goodwin2024subgradient} provides an encouraging direction.
				Second, we are not concerned with bit complexity, e.g., we assume exact evaluation of the exponential map: see Remark~\ref{rem:bitcomplexity}.
				Third, obtaining polynomial-time algorithms for these problems requires solving the associated optimization problem to very high accuracy~\cite{burgissernoncommutativeoptimization}. 
				This necessitates algorithms that (i) achieve curvature-independent convergence rates, and (ii) exhibit linear convergence in the absence of strong convexity. 
				This motivates the next item.
				
				\item \textbf{Methods with linear convergence}: 
				We focus on dimension-independent first-order methods. 
				However, in the absence of strong convexity, such methods typically achieve only sublinear convergence rates. 
				In contrast, in the Euclidean setting, there exist dimension-dependent algorithms --- such as the ellipsoid method or interior-point methods --- that attain linear convergence even without strong convexity. 
				Can one design curvature-independent algorithms for (non-strongly) h-convex optimization that achieve linear convergence rates?
				In Section~\ref{sec:faster_rates}, we sketched how to do this for hyperbolic space and when $m=1$. 
				However, it remains unclear how to extend this approach to more general Hadamard manifolds or $m > 1$. 
				
				\item \textbf{Translation between convex and h-convex}: 
				Extending algorithms for convex optimization to g-convex optimization can be tricky, sometimes requiring a fair amount of effort and ingenuity.\footnote{The line of research on g-convex accelerated methods is a prime example of this. See \cite{zhang2018estimate,ahn2020nesterov,lin2020accelerated,alimisis2021momentum,jin2022understanding,kim2022accelerated,martinez2022global,han2023riemannian,martinez2023accelerated}.}
				In contrast, all the convex algorithms we considered (even Nesterov acceleration), generalizes relatively easily to the h-convex setting.  Why?
				Is there an underlying correspondence/algorithm which, given a convex algorithm, provides the corresponding h-convex algorithm, with the identical complexity guarantees?
				Perhaps there is some connection to the performance estimation problem (PEP) framework~\cite{drori2014performance} in convex optimization?
				
				\item \textbf{Inner characterization}: When $M=\mathbb{H}^n$, there is an inner characterization of h-convex sets (see the third item of the list in Setion~\ref{sec:h-geometry}).  
				Is there an inner characterization of h-convex sets or functions on general Hadamard manifolds?
			\end{itemize}
			
			\section*{Acknowledgements}
			We are grateful to Michael Walter for insightful discussions on Horn's problem.  
			We also thank Nicolas Boumal for many valuable conversations and his generous guidance throughout this work.  
			CC was supported by the Swiss State Secretariat for Education, Research and Innovation (SERI) under contract number MB22.00027.
			
			\bibliographystyle{plain}
			\bibliography{sample}
			
			\appendix
			
			\allowdisplaybreaks[1]
			
			\section{Busemann functions: Details for Section~\ref{sec:pre}}
			\label{app:prelim}
			
			\subsection{Cone at infinity $CM(\infty)$ as dual space of $M$}
			\label{app:dual_space}
			
			The idea of considering $CM(\infty)$ as the dual space of $M$ can be found in \cite{hirai2023convex}. Our approach differs from their work only in the choice of topology on the space $CM(\infty)$.
			
			For a fixed point $p \in M$, the space $M(\infty)$ can be considered as the set of unit-scale Busemann functions $B_{p,\xi}=B_{p,v}$ by identifying each $B_{p,\xi}=B_{p,v}$ with $\xi=\exp_{p}(-\infty v) \in M(\infty)$. In particular, when $M=\mathbb{R}^n$, the topological space $M(\infty)$ can be identified with the unit sphere in the dual space $(\mathbb{R}^n)^*$. 
			
			\paragraph{Constructing $CM(\infty)$.}
			The cone at infinity $CM(\infty)$ is defined as the quotient space
			\[
			CM(\infty)=\left(M(\infty)\times[0,\infty)\right)/\left(M(\infty)\times\{0\}\right),
			\]
			where $M(\infty)$ is endowed with the sphere topology. Note that the resulting topological structure on $CM(\infty)$ differs from those induced by the angular metric or the Tits metric, which can be found in \cite[\S II.4]{ballmann2012lectures} and \cite[\S II.9]{bridson2013metric}. Since $M(\infty)$ is homeomorphic to $\mathbb{S}^{n-1}$, $CM(\infty)$ is homeomorphic to $\mathbb{R}^{n}$. Precisely, the map $T_{p}M\to M(\infty)\times[0,\infty)\to CM(\infty)$, obtained by composing the map $v\mapsto\left(\exp_{p}(-\infty v),\|v\|\right)$ and the quotient map, is a homeomorphism.

			\paragraph{Scaled Busemann functions as dual elements.}
			We consider non-negative scalar multiples of Busemann functions. For $c\geq 0$, $p\in M$, and $v\in T_p M$, we denote the non-negatively scaled Busemann function $cB_{p,v}$ by $B_{p,cv}$. For a fixed point $p \in M$, the space $CM(\infty)$ can be seen as the set of $B_{p,cv}$. Here, each $B_{p,cv}$ with $c\geq 0$ and $\|v\|=1$ corresponds to the equivalence class of $(\xi,c) \in M(\infty) \times [0,\infty)$, where $\xi=\exp_{p}(-\infty v)$. We also use the notion $B_{p,\omega}$ to denote $cB_{p,\xi}$, where $\xi\in M(\infty)$ and $\omega=[(\xi,c)]\in CM(\infty)$.
			
			\paragraph{Angles in $M(\infty)$.}
			Given $x \in M$ and $\xi_1, \xi_2 \in M(\infty)$, the angle at $x$ between the geodesic rays from $x$ to $\xi_1$ and $\xi_2$ is denoted by $\angle_x(\xi_1, \xi_2)$. The angle $\angle (\xi_1,\xi_2)$ between $\xi_1$ and $\xi_2$ in $M(\infty)$ is defined as $\sup_{x\in M}\angle_{x}(\xi_{1},\xi_{2})$.

			\begin{remark}
				To justify why we use the sphere topology on $M(\infty)$ instead of the topologies induced by the angular metric or the Tits metric, consider the following: In hyperbolic spaces $\mathbb{H}^n$, the angle between any two non-asymptotic geodesic rays is $\pi$ \cite[Exam.~II.9.6]{bridson2013metric}. Consequently, the topology on $M(\infty)$ induced by the angular metric or the Tits metric is discrete. This obstructs the generalization of many theories in convex geometry and convex analysis. For instance, the proof of Theorem~\ref{thm:supporting_horosphere} in Appendix~\ref{app:first-order-as-supporting} does not hold in this setting since $M(\infty)$ is not compact. 
			\end{remark}
			We recall that the map $(x, \xi, y) \mapsto B_{x, \xi}(y)$ is continuous on $M \times M(\infty) \times M$ by the construction of the sphere topology (standard topology) on $M(\infty)$ (see \cite[\S II.1]{ballmann2012lectures}). The following lemma generalizes this result into the scaled case.
			\begin{lemma}
				\label{lem:B-is-conti}
				The map $\beta:(x,\omega,y)\mapsto B_{x,\omega}(y)$ is continuous on $M\times CM(\infty) \times M$.
			\end{lemma}
			\begin{proof}
				Fix a reference point $p\in M$. Since $CM(\infty)$ is homeomorphic to $T_p M$, we can consider $M\times CM(\infty)\times M$ as a metric space, where the metric on $CM(\infty)$ is given by the usual metric in $T_p M$. Thus, it suffices to show that $\beta$ is continuous at every point in $M \times CM(\infty) \times M$. Since $CM(\infty) \setminus \{0\}$ is simply the product of $M(\infty)$ and $(0, \infty)$, the continuity of $\beta$ at points in $(M \times CM(\infty) \times M) \setminus (M \times \{0\} \times M)$ follows from the fact that the product of two continuous functions is continuous. We now show that $\beta$ is continuous at points in $M \times \{0\} \times M$. To show this, we need to show $\left|\beta(x,\omega,y)-\beta(x',0,y')\right|\to 0$ as $(x,\omega,y)\to (x',0,y')$. By the $\|\omega\|$-Lipschitz continuity of the scaled Busemann function $B_{x,\omega}$, we have
				\[
				\left|\beta(x,\omega,y)-\beta(x',0,y')\right|=\left|\beta(x,\omega,y)-0\right|\leq\|\omega\|d(x,y),
				\]
				from which the desired result follows.
			\end{proof}
			
			\subsection{Busemann functions on $\mathbb{H}^2$ with Poincaré disk model}
			\label{sec:h2}
			
			The content in this subsection is from \cite[Chapter~II.8--9]{bridson2013metric}. Consider the hyperbolic plane $M=\mathbb{H}^2$, represented by the \Poincare disk model $D=\{z\in \mathbb{C}:|z|<1\}$ in the complex plane. This space is a Hadamard manifold. Here, geodesics are all diameters of the disk, plus all arcs of Euclidean circles contained within the disk that are orthogonal to the boundary of the disk. The points on the boundary of the disk $\bd(D) = \{z\in \mathbb{C}:|z|=1\}$ are called the ideal points, and they are identified as elements of $M(\infty)$. Distances in this model are given by Cayley-Klein metrics. The unit-speed geodesic ray $\gamma$ with $\gamma(0)=0\in D$ and $\gamma(\infty)=\zeta\in \bd(D)$ is given by $\gamma(t)=\zeta\tanh(t/2)\in D$. The unit-scale Busemann function associated to $\gamma$ can be written as follows:\footnote{This result can be straightforwardly extended to the Poincar\'e ball model $\{z\in\mathbb{R}^{n}:\|z\|<1\}$, which represents the $n$-dimensional hyperbolic space $\mathbb{H}^n$. In this setting, the unit-speed geodesic rays starting at $0$ and the unit-scale Busemann functions are given by the same formulas, where the absolute value of complex numbers are replaced by the norm of $n$-vectors (see, for example, \cite[Appendix~A]{ghadimi2021hyperbolic}).}
			\[
			B_{\zeta}(z)=-\log\left(\frac{1-|z|^{2}}{|z-\zeta|^{2}}\right).
			\]
			\begin{proposition}
				\label{prop:b-on-H2}
				Let $\zeta,\xi\in \bd(D)$. Let $\gamma$ be the geodesic ray with $\gamma(0)=0$ and $\gamma(\infty)=\xi$.
				\begin{enumerate}[label=\textup{(\roman*)}]
					\item If $\xi=\zeta$, then we have $\frac{d}{dt}B_{\zeta}(\gamma(t))\to - 1$ as $t\to\infty$. 
					\item If $\xi\neq\zeta$, then we have $\frac{d}{dt}B_{\zeta}(\gamma(t))\to1$ as $t\to\infty$. 
				\end{enumerate}
			\end{proposition}
			\begin{proof}
				Since the gradient of $B_{\xi}$ at $p$ is $v\in T_p M$ such that $\|v\|=1$ and $\exp_p(-\infty v)=\xi$, we have $\frac{d}{dt}B_{\zeta}(\gamma(t))=-\cos\angle_{\gamma(t)}(\zeta,\xi)$. Note that in $\mathbb{H}^2$, we have $\angle(\xi,\eta)=0$ when $\xi=\eta$ and $\angle(\xi,\eta)=\pi$ when $\xi \neq \eta$ \cite[Example~II.9.6]{bridson2013metric}. The desired result follows from $\lim_{t\to\infty}\angle_{\gamma(t)}(\xi,\eta)=\angle(\xi,\eta)$ \cite[Prop.~II.9.8]{bridson2013metric}.
			\end{proof}

			In particular, we will use the following immediate consequence when constructing counterexamples in Appendices~\ref{app:convexity_local_property} and \ref{app:subdifferential_local_concept}. It states that when comparing scaled Busemann functions on $\mathbb{H}^2$, the large-scale one (centered at $\zeta_2$) eventually surpasses the small-scale one in every direction except the direction toward $\zeta_2$.
			\begin{proposition}
				\label{prop:will-used-counter}
				Let $\zeta_1,\zeta_2,\xi\in \bd(D)$. Let $0<c_1<c_2$. Let $\gamma$ be the geodesic ray with $\gamma(0)=0$ and $\gamma(\infty)=\xi$.
				\begin{enumerate}[label=\textup{(\roman*)}]
					\item If $\xi=\zeta_2$, then there is $T$ such that $c_1B_{\zeta_1}(\gamma(t)) \geq c_2B_{\zeta_2}(\gamma(t))$ for all $t\geq T$. 
					\item If $\xi\neq\zeta_2$, then there is $T$ such that $c_1B_{\zeta_1}(\gamma(t)) \leq c_2B_{\zeta_2}(\gamma(t))$ for all $t\geq T$.
				\end{enumerate}
			\end{proposition}

			\section{Horospherically convex analysis: Details for Section~\ref{sec:h-convex_analysis}}
			
			\subsection{Proof of Proposition~\ref{prop:sctoc}: From squared distances to Busemann functions}
			\label{app:prop:sctoc}

			Note that the equality
			\begin{equation}
				\label{eq:scceq}
				\lim_{t\to\infty}(d(p,\gamma(t))-t)=\lim_{t\to\infty}\frac{d^{2}(p,\gamma(t))-t^{2}}{2t},
			\end{equation}
			holds for any $p\in M$ and any unit-speed geodesic $\gamma$, where both limits always exist (see \cite[Example~3.8]{ballmann1985}). Let $p=y$, $\gamma(0)=y$, and $\gamma'(0)=-v/\|v\|$. Then, we have
			\begin{align*}
				\lim_{\mu\to0}Q_{y,v}^{\mu}(x) & =\lim_{\mu\to0}\left[-\frac{1}{2\mu}\|v\|^{2}+\frac{\mu}{2}d\left(\exp_{y}\left(-\frac{1}{\mu}v\right),x\right)^{2}\right]\\
				& =\lim_{s\to\infty}\left[-\frac{s}{2}\|v\|^{2}+\frac{1}{2s}d\left(\exp_{y}\left(-sv\right),x\right)^{2}\right] && (s\leftarrow1/\mu)\\
				& =\lim_{t\to\infty}\left[-\frac{t}{2}\|v\|+\frac{\|v\|}{2t}d\left(\exp_{y}\left(-t\frac{v}{\|v\|}\right),x\right)^{2}\right] && (t\leftarrow\|v\|s)\\
				& =\|v\|\lim_{t\to\infty}\left[d\left(\exp_{y}\left(-t\frac{v}{\|v\|}\right),x\right)-t\right]\\
				& =\|v\|B_{y,v/\|v\|}(x)\\
				& =B_{y,v}(x),
			\end{align*}
			where we used \eqref{eq:scceq} at the fourth equality, and the definition of the Busemann function $B_{y,v/\|v\|}$ at the fifth equality.
			
			\subsection{Proof of Proposition~\ref{prop:basic}: Basic properties of h-convex functions}
			\label{app:prop:basic}
			
			\begin{enumerate}[label=\textup{(\roman*)}]
				\item Let $F(z)=d(x,z)^2$ and $\gamma(t)=\exp_{y}(-tv)$. Then, we have
				\[
				\frac{d}{dt}d\left(\gamma(t),x\right)^{2}=\left\langle \nabla F(\gamma(t)),\dot{\gamma}(t)\right\rangle =-2\left\langle \exp_{\gamma(t)}^{-1}(x),\dot{\gamma}(t)\right\rangle =2\left\langle \exp_{\gamma(t)}^{-1}(x),\Gamma_{y}^{\gamma(t)}v\right\rangle ,
				\]
				where $\Gamma_{p}^{q}$ denotes the parallel transport from $p$ to $q$. Thus, the partial derivative of the right-hand side $Q_{y,v}^{\mu}(x)$ of \eqref{eq:defining_strong-h-convex} by $\mu$ can be computed as
				\begin{align*}
					\frac{\partial}{\partial\mu}Q_{y,v}^{\mu}(x) & =\frac{\partial}{\partial\mu}\left\{ -\frac{1}{2\mu}\|v\|^{2}+\frac{\mu}{2}d\left(\exp_{y}\left(-\frac{1}{\mu}v\right),x\right)^{2}\right\} \\
					& =\frac{dt}{d\mu}\frac{\partial}{\partial t}\left\{ -\frac{t}{2}\|v\|^{2}+\frac{1}{2t}d\left(\exp_{y}\left(-tv\right),x\right)^{2}\right\} \\
					& =-\frac{1}{\mu^{2}}\left[-\frac{1}{2}\|v\|^{2}-\frac{1}{2t^{2}}d\left(y^{++},x\right)^{2}+\frac{1}{t}\left\langle \exp^{-1}_{y^{++}}(x),\Gamma_{y}^{y^{++}}v\right\rangle \right]\\
					& =\frac{1}{2\mu^{2}}\|v\|^{2}+\frac{1}{2}d\left(y^{++},x\right)^{2}-\mu\left\langle \exp^{-1}_{y^{++}}(x),\frac{1}{\mu^{2}}\Gamma_{y}^{y^{++}}v\right\rangle \\
					& =\frac{1}{2\mu^{2}}\|\Gamma_{y}^{y^{++}}v\|^{2}+\frac{1}{2}\|\exp^{-1}_{y^{++}}(x)\|^{2}-\left\langle \exp^{-1}_{y^{++}}(x),\frac{1}{\mu}\Gamma_{y}^{y^{++}}v\right\rangle \\
					& =\frac{1}{2}\left\Vert \frac{1}{\mu}\Gamma_{y}^{y^{++}}v-\exp^{-1}_{y^{++}}(x)\right\Vert ^{2}\\
					& \geq0,
				\end{align*}
				where we used the substitution $t=1/\mu$ and denoted $y^{++} = \exp_y(-\frac{1}{\mu}v)=\exp_y(-tv)$. Thus, the right-hand side $Q_{y,v}^{\mu}(x)$ of \eqref{eq:defining_strong-h-convex} increases as $\mu$ increases. It follows that
				\begin{align*}
					\ph_{\mu_{2}}f(y) & =\left\{ v\in T_{y}M:f(x)-f(y)\geq Q_{y,v}^{\mu_{2}}(x)\ \forall x\in M\right\} \\
					& \subseteq\left\{ v\in T_{y}M:f(x)-f(y)\geq Q_{y,v}^{\mu_{1}}(x)\ \forall x\in M\right\} \\
					& =\ph_{\mu_{1}}f(y).
				\end{align*}
				
				\item Note that $f$ is $\mu_i$-strongly h-convex if and only if $\ph_{\mu_{i}} f(y)$ is nonempty for all $y \in M$. Thus, (ii) follows from (i).
				
				\item Note that the right-hand side $Q_{y,v}^{\mu}(x)$ of \eqref{eq:defining_strong-h-convex} is $\mu$-strongly g-convex as a function of $x$. Since $Q_{y,v}^{\mu}(y) = f(y)$ and $\nabla Q_{y,v}^{\mu}(y) = v$, we have
				\[
				Q_{y,v}^{\mu}(x) \geq f(y)+\langle v,\exp^{-1}_{y}(x)\rangle+\frac{\mu}{2}d(x,y)^{2}.
				\]
				This shows that
				\begin{equation}
					\label{eq:in-proving-iii}
					\begin{aligned}
						\ph_{\mu}f(y) & \subseteq\{v\in T_{y}M:f(x)\geq f(y)+\langle v,\exp_{y}^{-1}(x)\rangle+\frac{\mu}{2}d(x,y)^{2}\textup{ for all }x\in M\}\\
						& \subseteq\{v\in T_{y}M:f(x)\geq f(y)+\langle v,\exp_{y}^{-1}(x)\rangle\textup{ for all }x\in M\}\\
						& =:\partial f(y).
					\end{aligned}
				\end{equation}

				\item Note that $f$ is $\mu$-strongly h-convex if and only if $\ph_{\mu} f(y)$ is nonempty for all $y \in M$, and that $f$ is $\mu$-strongly g-convex if and only if the set $\{v\in T_{y}M:f(x)\geq f(y)+\langle v,\exp_{y}^{-1}(x)\rangle+\frac{\mu}{2}d(x,y)^{2}\textup{ for all }x\in M\}$ is nonempty for all $y\in M$. Thus, (iv) follows from the first inclusion in \eqref{eq:in-proving-iii}.
				
				\item ($\subseteq$) This follows from $\ph_{\mu} f(y)\subseteq\partial f(y)$ (see (iii)) and $\partial f(y)=\{\nabla f(y)\}$.
				
				($\supseteq$) We have already shown $\ph_{\mu}f(y)\subseteq\{\nabla f(y)\}$. If the reverse inclusion does not hold, then $\ph_{\mu} f(y)$ is empty, which means that $f$ is not $\mu$-strongly h-convex, a contradiction. 
				
				\item Suppose that $\{f_i\}_{i\in I}$ is a family of $\mu$-strongly h-convex functions on $M$. We need to show that the supremum $f=\sup_{i\in I} f_i$ is $\mu$-strongly h-convex. For simplicity, we provide a proof for the case $\mu=0$. The case $\mu>0$ can be handled by using $Q_{y,v}^{\mu}$ instead of $B_{y,v}$. Fix $y\in M$. We need to show the existence of an h-subgradient of $f$ at $y$. For $k=1,2,\ldots$, since $f(y)=\sup_{i\in I}f_i (y)$, there is an index $i_k \in I$ such that $f_{i_k}(y)\geq f(y)-1/k$. For each $k$, choose $v_k\in \ph f_{i_k}(y)$. Then, for each $k$, by the h-convexity of $f_{i_k}$, we have
				\begin{equation}
					\label{eq:forthm1-copy}
					f(x)\geq f_{i_{k}}(x)\geq f_{i_{k}}(y)+B_{y,v_{k}}(x)\geq f(y)-\frac{1}{k}+B_{y,v_{k}}(x)\quad \forall x\in M.
				\end{equation}
				Note that the norms of $v_k$ are bounded.\footnote{\label{footnote:vkbounded}Since $f:M\to\mathbb{R}$ is a g-convex function, it is Lipschitz continuous on any compact subset of $M$ (see \cite[Cor.~3.10]{udriste1994convex}). Consider the closed unit ball $\bar{B}(y,1)$ centered at $y$, which is a compact subset of $M$. Then, there is a constant $L$ such that $|f(x)-f(y)|\leq Ld(x,y)$ for all $x\in\bar{B}(y,1)$. Let $x=\exp_{y}(v_k/\|v_k\|)$. Then, since $d(x,y)=1$, we have $f(x)\leq f(y)+L$. Combining this inequality with \eqref{eq:forthm1-copy}, we obtain $L+1/k\geq \|v_k\|$. Thus, the sequence $\|v_k\|$ is bounded.} Thus, we can extract a subsequence of $\{v_k\}$ which converges to some $v\in T_y M$. Denote by $\phi$ the homeomorphism from $T_y M$ to $CM(\infty)$ constructed in Appendix~\ref{app:dual_space}. We now take the limit $k\to\infty$ in \eqref{eq:forthm1}. Since the map $v\mapsto B_{y,v}(x)=B_{y,\phi(v)}(x)$ is continuous by Lemma~\ref{lem:B-is-conti}, we have
				\[
				f(x)\geq f(y)+B_{y,v}(x)\quad \forall x\in M.
				\]

				\item Let $x$ be an arbitrary point in $M$. We denote a subgradient of $f$ at $x$ as $g_x$, and a subgradient of $h$ at $f(x)$ as $m$. It suffices show that the inequality
				\[
				h(f(y))\geq h(f(x))+B_{x,mg_x}(y).
				\]
				holds for all $y\in M$. By the h-convexity of $f$, we have
				\[
				f(y)-f(x)\geq B_{x,g_x}(y).
				\]
				Multiplying both sides by $m$ yields
				\[
				m\left(f(y)-f(x)\right)\geq B_{x,mg_x}(y).
				\]
				Since $m$ is a subgradient of $h$ at $f(x)$, we have
				\[
				h(f(y))-h(f(x))\geq m\left(f(y)-f(x)\right).
				\]
				Combining these results completes the proof.
				\item A straightforward calculation shows $Q_{y,rv}^{r\mu}(x)=rQ_{y,v}^{\mu}(x)$. Thus, 
				\begin{align*}
					\textup{\ensuremath{f} is \ensuremath{\mu}-strongly h-convex} & \Longleftrightarrow[\exists v\in T_{y}M\textup{ such that }f(x)-f(y)\geq Q_{y,v}^{\mu}(x)]\ \forall x,y\in M\\
					& \Longleftrightarrow[\exists v\in T_{y}M\textup{ such that }rf(x)-rf(y)\geq Q_{y,rv}^{r\mu}(x)]\ \forall x,y\in M\\
					& \Longleftrightarrow\textup{\ensuremath{rf} is \ensuremath{(r\mu)}-strongly h-convex}.
				\end{align*}
			\end{enumerate}

			\subsection{Horospherical convexity is not preserved under addition}
			\label{app:rmk:basic}
			
			In this subsection, we show that the h-convexity of functions is not preserved under addition by providing a counterexample. Consider the hyperbolic plane $\mathbb{H}^2$ in the \Poincare disk model $D=\{z\in\mathbb{C}:|z|<1\}$, where $p\in\mathbb{H}^2$ is identified with $0\in D$ (see Section~\ref{sec:h2}). Consider three ideal points $\xi_{1}=1$, $\xi_{2}=i$, and $\bar{\xi}=\frac{\sqrt{2}}{2}+\frac{\sqrt{2}}{2}i$ in $\bd(D)$. If h-convexity is preserved under addition, then the function $B_{p,\xi_1}+B_{p,\xi_2}$ should be h-convex, and thus the inequality $B_{p,\xi_{1}}(x)+B_{p,\xi_{2}}(x)\geq\sqrt{2}B_{p,\bar{\xi}}(x)$ must hold for all $x\in D$ since $\sqrt{2}\nabla B_{\bar{\xi}}(p)$ is the gradient of $B_{p,\xi_1}+B_{p,\xi_2}$ at $p$. However, this inequality does not hold at $x=\frac{\sqrt{2}}{2}+0i$ because
			\begin{align*}
				B_{p,\xi_{1}}(x) & =-\log\left(\frac{1-(1/\sqrt{2})^{2}}{(1-1/\sqrt{2})^{2}}\right)=-\log\left(\frac{1/2}{(1-1/\sqrt{2})^{2}}\right)\approx-1.76275,\\
				B_{p,\xi_{2}}(x) & =-\log\left(\frac{1-(1/\sqrt{2})^{2}}{(1/\sqrt{2})^{2}+1^{2}}\right)=-\log\left(\frac{1/2}{3/2}\right)\approx1.09861,\\
				B_{p,\bar{\xi}}(x) & =-\log\left(\frac{1-(1/\sqrt{2})^{2}}{(1/\sqrt{2})^{2}}\right)=-\log\left(\frac{1/2}{1/2}\right)=0.
			\end{align*}
			\begin{remark}
				\label{rmk:from-sum-example}
				In particular, $f = B_{\xi_{1}} + B_{\xi_{2}}$ serves as an example of a g-convex function that is not h-convex. At $p\in\mathbb{H}^2$, we have $\partial f(p)=\{\sqrt{2}\nabla B_{\bar{\xi}}(p)\}$ and $\ph f(p)=\emptyset$. Thus, $\partial f(p)\supsetneq \ph f(p)$.
			\end{remark}
			
			\subsection{Proof of Theorem~\ref{thm:envelope_rep_h-convex}: Outer characterization for h-convex functions}
			\label{app:thm:envelope_rep_h-convex}
			
			The proof techniques are the same for both the non-strongly h-convex case and the strongly h-convex case, as one only needs to use $Q_{y,v}^{\mu}$ instead of $B_{y,v}$. Here, we provide a proof for the non-strongly h-convex case ($\mu=0$). We will prove that $f$ is h-convex if and only if it can be written as the supremum of scaled Busemann functions.
			
			($\Longrightarrow$) For each $y\in M$, choose $v_y\in \ph f(y)$. Then, we have
			\[
			f(x)=\sup\left\{ f(y)+B_{y,v_{y}}(x):y\in M\right\}.
			\]
			
			($\Longleftarrow$) Fix $y\in M$. We need to show the existence of an h-subgradient of $f$ at $y$. For $k=1,2,3,\ldots$, by the assumption, there is a scaled Busemann function $B_k$ minorizing $f$ such that $B_k(y)\geq f(y)-1/k$. Let $v_k=\nabla B_k(y)\in T_y M$. Then, we have
			\begin{equation}
				\label{eq:forthm1}
				f(x)\geq f(y)-\frac{1}{k}+B_{y,v_k}(x)\quad \forall x\in M.
			\end{equation}
			Since the norm of $v_k$ is bounded (using the argument in Footnote~\ref{footnote:vkbounded}), we can extract a subsequence of $\{v_k\}$ which converges to some $v$. Now, we take the limit $k\to\infty$ in \eqref{eq:forthm1}. Since the map $v\mapsto B_{y,v}(x)$ is continuous (Lemma~\ref{lem:B-is-conti}), we have
			\[
			f(x)\geq f(y)+B_{y,v}(x)\quad \forall x\in M.
			\]
			
			\subsection{Proof of Theorem~\ref{thm:moreau}: Moreau envelope}
			\label{app:thm:moreau}
			Let us first prove Lemma~\ref{lem:moreau}, which we repeat for the reader's convenience.
			\begin{lemma}
				Let $x,x',y,y'\in M$. Then, we have
				\[
				B_{x',\exp_{x'}^{-1}(x)}(y')+\frac{1}{2}d(y,y')^{2}+\frac{1}{2}d(x,x')^{2}\geq B_{x',\exp_{x'}^{-1}(x)}(y).
				\]
			\end{lemma}
			\begin{proof}
				Let $s$ be an arbitrary positive constant. Denote $a = d(x,x')$. Let $z_{s}=\exp_{x'}(-\frac{s}{a}\exp^{-1}_{x'}(x))$. Then, we have
				\begin{align*}
					& a(d(y',z_{s})-s)+\frac{1}{2}a^{2}+\frac{1}{2}d(y,y')^{2}-a(d(y,z_{s})-s)\\
					& \qquad=ad(y',z_{s})+\frac{1}{2}a^{2}+\frac{1}{2}d(y,y')^{2}-ad(y,z_{s})\\
					& \qquad\geq\frac{1}{2}(-d(y',z_{s})+d(y,y')+d(y,z_{s}))(d(y',z_{s})+d(y,y')-d(y,z_{s}))\\
					& \qquad\geq0,
				\end{align*}
				where the first inequality follows from minimizing over $a$, and the second inequality follows from the triangle inequalities. Taking the limit $s \to \infty$ yields the desired result.
			\end{proof}
			We now prove Theorem~\ref{thm:moreau}. Assume that $x,y\in M$ are given. We need to show $f_{\lambda}(x) + B_{x, \nabla f_{\lambda}(x)}(y) \leq f_{\lambda}(y)$. Let $x'=\prox_{\lambda f}(x)$ and $y'=\prox_{\lambda f}(y)$. Then, we have $\nabla f_{\lambda}(x) = -\frac{1}{\lambda} \exp^{-1}_{x}(x')$ by Proposition~\ref{prop:moreau-basic}~(iv). By the definition of the proximal operator, we also have
			\begin{align*}
				\nabla f(x') & =\frac{1}{\lambda}\exp_{x'}^{-1}(x),\;f_{\lambda}(x)=f(x')+\frac{1}{2\lambda}d(x,x')^{2},\\
				\nabla f(y') & =\frac{1}{\lambda}\exp^{-1}_{y'}(y),\;f_{\lambda}(y)=f(y')+\frac{1}{2\lambda}d(y,y')^{2}.
			\end{align*}
			We now have
			\begin{align*}
				f_{\lambda}(y) & =\frac{1}{2\lambda}d(y,y')^{2}+f(y')\\
				& \geq\frac{1}{2\lambda}d(y,y')^{2}+f(x')+B_{x',\frac{1}{\lambda}\exp^{-1}_{x'}(x)}(y')\\
				& =f_{\lambda}(x)+B_{x',\frac{1}{\lambda}\exp^{-1}_{x'}(x)}(y')+\frac{1}{2\lambda}d(x,x')^{2}-\frac{1}{\lambda}d(x,x')^{2}+\frac{1}{2\lambda}d(y,y')^{2}\\
				& \geq f_{\lambda}(x)+B_{x',\frac{1}{\lambda}\exp^{-1}_{x'}(x)}(y)-\frac{1}{\lambda}d(x,x')^{2}\\
				& =f_{\lambda}(x)+B_{x,-\frac{1}{\lambda}\exp^{-1}_{x}(x')}(y),
			\end{align*}
			where the first inequality follows from the h-convexity of $f$, and the second inequality follows from Lemma~\ref{lem:moreau}. This shows the h-convexity of $f_{\lambda}$. The $(1/\lambda)$-h-smoothness of $f_{\lambda}$ follows from Proposition~\ref{prop:smooth}~(i).

			\subsection{Proof of Proposition~\ref{prop:smooth}: Horospherical $L$-smoothness}
			\label{app:prop:smooth}
			
			\begin{enumerate}[label=\textup{(\roman*)}]
				\item The proof is similar to that of Theorem~\ref{thm:envelope_rep_h-convex} in Appendix~\ref{app:thm:envelope_rep_h-convex}. 
				
				($\Longrightarrow$) We have
				\[
				f(y)=\inf\left\{ f(x)+Q_{x,\nabla f(x)}^{L}(y):x\in M\right\} .
				\]
				
				($\Longleftarrow$) We need to show that the inequality \eqref{eq:defining-smooth} holds for all $x,y\in M$. Fix $y\in M$. For $k=1,2,3,\ldots$, by the assumption that $f$ is the infimum of functions of the form $q^{L}_{p,r}(x)= \frac{L}{2}d(x,p)^{2}+r$, there is a function $q_k$ (in such form) majorizing $f$ such that $q_k(y)\leq f(y)+1/k$. Let $v_k=\nabla q_k(y)\in T_y M$. Then, we have
				\begin{equation}
					\label{eq:forthm1-smooth}
					f(x)\leq f(y)+\frac{1}{k}+Q_{y,v_{k}}^{L}(x)\quad \forall x\in M.
				\end{equation}
				Using the differentiability of $f$, one can check $\lim_{k\to\infty}v_k=\nabla f(y)$. Now, we take the limit $k\to\infty$ in \eqref{eq:forthm1-smooth}. The map $v\mapsto Q_{y,v}^{L}(x)$ is continuous since distance functions are continuous. Thus, we have
				\[
				f(x)\leq f(y)+Q_{y,\nabla f(y)}^{L}(x)\quad \forall x\in M.
				\]
				
				\item Note that $Q_{y,\nabla f(y)}^{L}(x)$ is $L$-strongly g-convex, its function value at $y$ is $f(y)$, and that its gradient values at $y$ is $\nabla f(y)$. From the (geodesic) $L$-smoothness of $f$, it follows that
				\[
				f(x)\leq f(y)+\langle \nabla f(y),\exp_{y}^{-1}(x)\rangle+\frac{L}{2}d(x,y)^{2}\leq f(y)+Q_{y,\nabla f(y)}^{L},
				\]
				which shows the horospherical $L$-smoothness of $f$.
				\item Put $x\leftarrow y^{+}$ in \eqref{eq:defining-smooth}.
			\end{enumerate}
			
			\subsection{Proof of Theorem~\ref{thm:set_function}: Equivalence between h-convex sets and h-convex functions}
			\label{app:thm:set_function}

			We first review the standard proof of Theorem~\ref{thm:set_function} in the case where $M=\mathbb{R}^n$.
			\begin{proof}[Proof of Theorem~\ref{thm:set_function} with $M=\mathbb{R}^n$]
				($\Longrightarrow$) If $f$ is convex, then $f$ is the supremum of its affine supports. Thus, $\epi f$ is the intersection of their epigraphs, which are closed half-spaces. Therefore, $\epi f$ is closed and convex.
				
				($\Longleftarrow$) If $\epi f$ is closed and convex, then it is the intersection of closed half-spaces in $\mathbb{R}^n \times \mathbb{R}$ by the outer characterization of closed convex sets in $\mathbb{R}^n \times \mathbb{R}$ \cite[Theorem~11.5]{rockafellar1997convex}. Note that closed half-spaces in $\mathbb{R}^n \times \mathbb{R}$ can be categorized into the following three cases: (i) epigraphs of affine functions, (ii) $H \times \mathbb{R}$ for some closed half-space $H$ in $\mathbb{R}^n$, and (iii) hypographs of affine functions. Since only the first type can contain $\epi f$, it follows that $\epi f$ is the intersection of epigraphs of affine functions. Hence, $f$ is the supremum of affine functions and thus convex.
			\end{proof}
			Note that the proof relies crucially on the fact that the epigraphs of affine functions are precisely the upper half-spaces in $\mathbb{R}^n \times \mathbb{R}$. We develop a similar result that fits well for h-convexity. We first review a basic property of Busemann functions on the product manifold $M=M_1\times M_2$.
			\begin{proposition}
				\label{prop:prod-B-functions}
				Let $x_1\in M_1$, $x_2\in M_2$, $(p_1,v_1)\in TM_1$, and $(p_2,v_2)\in TM_2$. Denote $p=(p_1,p_2)\in M$, $x=(x_1,x_2)\in M$, and $v=(v_1,v_2)\in T_p M$. Then, we have
				\[
				B_{p,v}(x)=B_{p_1,v_{1}}(x_1)+B_{p_2,v_{2}}(x_2).
				\]
			\end{proposition}
			\begin{proof}
				A proof for the case where $\Vert v\Vert=1$ can be found in \cite[Example~8.24]{bridson2013metric}. The general result can be proved by combining this special result with the identity $B_{p,cv}=cB_{p,v}$.
			\end{proof}
			We can now prove the following lemma.
			\begin{lemma}
				\label{lem:Buse-horo}
				Closed horoballs in $M \times \mathbb{R}$ are classified into the following three types:
				\begin{enumerate}[label=\textup{(\roman*)}]
					\item Epigraph of a non-negatively scaled Busemann function.
					\item $H \times \mathbb{R}$, where $H$ is a closed horoball in $M$.
					\item Hypograph of a non-positively scaled Busemann function.
				\end{enumerate}
				Conversely, any subset of $M \times \mathbb{R}$ that falls into one of these classifications is a closed horoball in $M \times \mathbb{R}$.
			\end{lemma}
			\begin{proof}
				We observe that, by Proposition~\ref{prop:prod-B-functions}, scaled Busemann functions $\bar{B}$ on $M\times\mathbb{R}$ can be expressed in one of the following three forms:
				\[
				\textup{(i) }\bar{B}((x,r))=B(x)-cr\textup{ with }c\geq0,\quad\textup{(ii) }\bar{B}((x,r))=B(x),\quad\textup{(iii) }\bar{B}((x,r))=B(x)+cr\textup{ with }c\geq0,
				\]
				where $B$ is some non-negatively scaled Busemann functions on $M$. For each of the three cases, the set $\{(x,r) \in M \times \mathbb{R} : \bar{B}((x,r)) \leq 0\}$ corresponds to type (i), (ii), or (iii), respectively. This proves the first statement. The second statement (converse) is straightforward.
			\end{proof}
			We now prove Theorem~\ref{thm:set_function}. We use the outer characterization of h-convex functions in Theorem~\ref{thm:envelope_rep_h-convex}. If $f$ is h-convex on $M$, then the epigraph of $f$ is h-convex in $M\times\mathbb{R}$ since it is the intersection of epigraphs of non-negatively scaled Busemann functions. Conversely, if the epigraph of $f$ is h-convex in $M\times\mathbb{R}$, then it is the intersection of the sets of the three types in the statement of Lemma~\ref{lem:Buse-horo}. The cases (ii) and (iii) do not happen, so the epigraph of $f$ is the intersection of epigraphs of non-negatively scaled Busemann functions. Thus, $f$ is the supremum of such non-negatively scaled Busemann functions. This completes the proof.
			
			\subsection{Equivalence of two definitions of horospherically convex sets}
			\label{app:first-order-as-supporting}

			In this section, we show that the first two definitions of h-convex sets, listed in Section~\ref{sec:h-geometry}, are equivalent. This result is well-known, as both definitions are widely used. For completeness, we provide a proof here.
			
			\subsubsection{First definition implies second definition (Supporting horosphere theorem)}
			
			The following lemma is immediate.
			\begin{lemma}
				\label{lem:seperating_horosphere}
				Let $K$ be a closed h-convex set in $M$ (in the sense of Definition~\ref{def:h-convex_set}), and $x \notin K$. Then, there is a closed horoball $H$ containing $K$ but not $x$.
			\end{lemma}
			Using this lemma, we can now generalize the \emph{supporting hyperplane theorem} \cite[Thm.~18.8]{rockafellar1997convex}, \cite[Lem.~III.4.2.1]{hiriart1996convex}. The proof generalizes the standard proof technique used in the Euclidean case (see, for example, \cite[\S III.4]{hiriart1996convex}).
			\begin{theorem}[Supporting Horosphere Theorem]
				\label{thm:supporting_horosphere}
				Let $K$ be a closed h-convex set in $M$ (in the sense of Definition~\ref{def:h-convex_set}), and $x \in \bd(K)$. Then, there exists a closed horoball $H$ supporting $K$ at $x$ (i.e., $K \subseteq H$ and $x \in \bd(H)$).
			\end{theorem}
			\begin{proof}
				Since $\bd(K)$ is the boundary of $M\backslash K$, there exists a sequence $\{x_k\}$ in $M$ such that $x_k \notin K$ and $x_k\to x$ as $k\to \infty$. Then, for each $k$, by Lemma~\ref{lem:seperating_horosphere}, there exists $\xi_k \in M(\infty)$ such that $B_{x_{k},\xi_{k}}(y)\leq 0$ for all $y\in K$. Since $M(\infty)$ is sequentially compact (as it is homeomorphic to the unit sphere in any tangent space of $M$ under the homeomorphism given in Appendix~\ref{app:dual_space}), we can extract a subsequence of $\{\xi_{k}\}$ that converges to some limit $\xi \in M(\infty)$. Since the map $(x,\xi,y)\mapsto B_{x,\xi}(y)$ is continuous \cite[\S II.1]{ballmann2012lectures}, taking the limit $k\to\infty$ in $B_{x_{k},\xi_{k}}(y)\leq 0$, we obtain $B_{x,\xi}(y)\leq 0$. Therefore, the closed horoball $\{y\in M:B_{x,\xi}(y)\leq 0\}$ supports $K$ at $x$.
			\end{proof}
			Therefore, we conclude that if a closed set $K$ is h-convex in the first sense, then it is also h-convex in the second sense.

			\subsubsection{Second definition implies first definition (Partial converse of supporting horosphere theorem)}

			In this subsection, we provide a proof of the following theorem,  often referred to as the \emph{converse supporting hyperplane theorem} when $M=\mathbb{R}^n$ \cite[p.63]{boyd2004convex}. The proof generalizes the well-known technique in the Euclidean setting (see, for example, Israel's post on Stack Exchange \cite{stack2011}).
			\begin{theorem}
				Let $K$ be a closed set with a nonempty interior in $M$. Suppose that $K$ has a supporting horosphere at every point in its boundary. Then, $K$ is h-convex (in the sense of Definition~\ref{def:h-convex_set}).
			\end{theorem}
			
			\begin{proof}
				Let $K'$ be the intersection of all supporting closed horoballs of $K$. It suffices to show that $K=K'$. It is clear that $K\subseteq K'$. To show $K'\subseteq K$, we need to show that if $y\notin K$, then $y\notin K'$. 
				
				Suppose $y\notin K$. Let $x$ be any point in $\inter K$. Consider the geodesic segment $\gamma_{x}^{y} : [0,1] \to M$ from $x$ to $y$. Let $t=\sup\{t\in[0,1]:\gamma_{x}^{y}(t)\in K\}$. Then, we have $t\in (0,1)$. Let $z=\gamma_{x}^{y}(t)$. Then, $z\in \bd(K)$. By the assumption of the theorem statement, there is a supporting horosphere of $K$ at $z$. This means that there is a Busemann function $B$ such that $B\leq 0$ on $K$ and $B(z)=0$. Since $x\in \inter K$, we have $B(x)<0$. Using the fact that the Busemann function $B$ is g-convex, we have $B(y)>0$, which shows that $y\notin K'$. This completes the proof.
			\end{proof}
			Therefore, we conclude that if a closed set $K$ is h-convex in the second sense, then it is also h-convex in the first sense.
			
			\subsection{Geodesic convexity is a local property but horospherical convexity is not}
			\label{app:convexity_local_property}
			
			In this subsection, we present a counterexample: a function $f$ that is locally h-convex but not (globally) h-convex on $M$. Let $M$ be the hyperbolic plane $\mathbb{H}^2$. Let $p\in M$. Let $u_1$ and $u_2$ be two mutually perpendicular unit tangent vectors in $T_p M$. We consider three scaled Busemann functions: $B_{1}=B_{p,2u_{1}}$, $B_{2}=B_{p,4u_{2}}$, and ${B}_3=B_{p,u_{1}+2u_{2}}$. We define $\hat{B}=\max\{B_{1},B_{2}\}$, which is h-convex since it is the supremum of two scaled Busemann functions.
			
			Using the \Poincare disk model $D=\{z\in\mathbb{C}:|z|<1\}$ (see Appendix~\ref{sec:h2}), with the identifications $p\leftrightarrow 0\in D$, $\exp_p(-\infty u_1)\leftrightarrow 1\in \bd(D)$, and $\exp_p(-\infty u_2)\leftrightarrow i\in \bd(D)$, these functions can be written as follows:
			\begin{equation}
				\label{eq:b1b2b3}
				\begin{aligned}
					B_{1}(z) & =-2\log\left(\frac{1-|z|^{2}}{|z-1|^{2}}\right),\\
					B_{2}(z) & =-4\log\left(\frac{1-|z|^{2}}{|z-i|^{2}}\right),\\
					B_3 (z) & =-\sqrt{5}\log\left(\frac{1-|z|^{2}}{\left|z-\left(\frac{1}{\sqrt{5}}+\frac{2}{\sqrt{5}}i\right)\right|^{2}}\right).
				\end{aligned}
			\end{equation}
			
			Let $\epsilon=0.1$. Consider the set 
			\[
			\{z:B_{3}(z)+\epsilon\geq\hat{B}(z)\}.
			\]
			This set is closed because it is the inverse image of a closed set under a continuous map. This set has at least two connected components, one containing $0\in D$ (denote this set as $C$), and one whose boundary contains the ideal point $i$ (denote this set as $C'$). We verified the existence of the first set $C$ numerically (see Figure~\ref{fig:local}). The existence of the second set follows from Proposition~\ref{prop:will-used-counter} applied to $\xi=i\in \bd(D)$, and using that $\sqrt{5} > 2$. 
			
			Define a function $f$ on $M$ as $f=B_3+\epsilon$ on $C$ and $f=\hat{B}$ on $M\backslash C$. Then, we have $\grad f(p) = \grad B_3 (p)=u_1+2u_2$. Thus, if $f$ is (globally) h-convex, then the inequality
			\[
			f(x)\geq f(p)+B_{p,\nabla f(p)}(x)=\epsilon+B_{3}(x)
			\]
			should hold for all $x\in \mathbb{H}^2$. However, if $\gamma$ is the geodesic starting at $0$ and proceeding toward the ideal point $i$, then we have $f(\gamma(t))=\hat{B}(\gamma(t))$ and ${B}_3 (\gamma(t)) > \hat{B}(\gamma(t))$ for sufficiently large $t$ (where $\gamma(t)\in C'$). Thus, $f$ is not h-convex. However, $f$ is locally h-convex since (i) it is the supremum of two scaled Busemann functions $\max\{B_1, B_2\}$ on $M \setminus C$, and (ii) it is the supremum of three scaled Busemann functions $\max\{B_1, B_2, B_3 + \epsilon\}$ on a neighborhood of $C$.
			This shows that h-convexity is not a local property. A Python code for generating Figure~\ref{fig:local} is available at
			\begin{center}
				\url{https://github.com/jungbinkim1/Horospherical-Convexity}.
			\end{center}
			
			\begin{figure}
				\centering
				\includegraphics[width=0.5\textwidth]{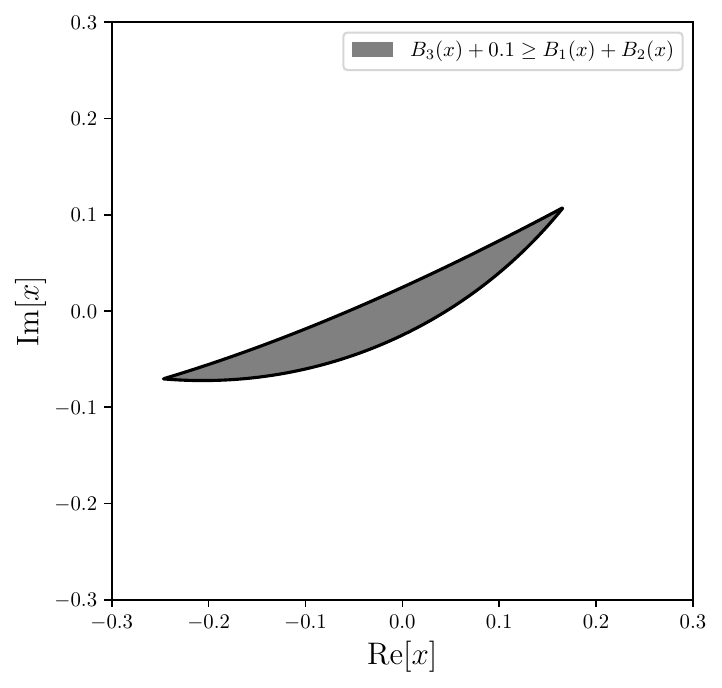}
				\caption{The set $C$ is shaded gray.}
				\label{fig:local}
			\end{figure}
			
			\subsection{Geodesic subdifferential is a local concept but horospherical subdifferential is not}
			\label{app:subdifferential_local_concept}
			
			In this subsection, we show that the following statement is false:
			\begin{center}
				If $f_1$ and $f_2$ are equal in some neighborhood $U \subseteq M$ of $x$, then $\ph f_1(x) = \ph f_2(x)$.
			\end{center}
			Consider the functions $B_1$, $B_2$, $B_3$, $\hat{B}$, and $f$ from the previous subsection. Define a function $\tilde{f} = \max\{B_3 + \epsilon, \hat{B}\}$. Then, $f$ and $\tilde{f}$ are equal on $C$, but their subdifferentials at $p$ differ, as $\ph f(p) = \emptyset$ and $\ph \tilde{f}(p) = \{u_1 + 2u_2\}$. This serves as a counterexample.
			
			\subsection{Some h-subgradients can be computed using only local information of $f$}
			\label{app:subg_oracle}

			When implementing the subgradient descent method in Section~\ref{sec:subg-des}, the following question arises:
			\begin{center}
				Q. {\it Is it reasonable to assume that the algorithm can obtain information about $f$ through queries to the subgradient oracle that outputs an element of $\partial f(x)$ for given $x \in M$?}
			\end{center}
			One might expect a negative answer since the set $\partial f(x)$ is not fully determined by the local behavior of $f$ (see Appendix~\ref{app:subdifferential_local_concept}). However, we provide a positive answer here.
			\begin{center}
				A. {\it When $f$ is h-convex, although it may not be possible to obtain the entire set $\ph f(x)$, one can obtain an element of $\ph f(x)$ using only the local information of $f$ near $x$.}
			\end{center}
			\begin{proof}
				Since $f$ is g-convex, it is locally Lipschitz continuous by \cite[Cor.~3.10]{udriste1994convex}. This implies that it is differentiable almost everywhere by a generalization of Rademacher's theorem in \cite[Thm.~5.7]{azagra2005nonsmooth}. Consider a sequence $x_k$ such that $f$ is differentiable at each $x_k$ and $x_k \to x$. By the local Lipschitz continuity of $f$, the norms of $\grad f(x_k)$ are bounded above by some constant $C$. Thus, the sequence $\{\grad f(x_k)\}$ is in a sequentially compact subset of $CM(\infty)$ (namely, the image of $\bar{B}(0,C)\subseteq T_{p}M$ under the homeomorphism from $T_p M$ to $CM(\infty)$ constructed in Appendix~\ref{app:prelim} for any $p\in M$), where we identified each $T_{x_k}M$ with $CM(\infty)$. Thus, we can extract a subsequence that converges to some $\omega\in CM(\infty)$. Taking the limit $k\to\infty$ in the inequality $f(y)\geq f(x_{k})+B_{x_{k},\nabla f(x_{k})}(y)$, we have $f(y)\geq f(x)+B_{x,\omega}(y)$ by Lemma~\ref{lem:B-is-conti}. Thus, we obtained an element $\omega$ of $\ph f(x)$, where we identified $T_x M$ and $CM(\infty)$.
			\end{proof}

			\section{Algorithms: Details for Section~\ref{sec:h-opt}}
			
			\subsection{Proof of Proposition~\ref{prop:d-is-conv}: Distance functions are h-convex} 
			\label{squareddistancefunctionsareBconvex}
			
			Recall that a sublevel set of a Busemann function is called a closed horoball. Its interior (an open horoball) can be expressed as a union of open geodesic balls \cite[p.23]{ballmann1985}. Precisely, if $z$ is on the geodesic $t \mapsto \exp_x(tv)$, where $t \in \reals$ and $\|v\|=1$, then we have
			\begin{align} \label{djdjkdnk}
				\{w \in \calM : B_{x, v}(w) < B_{x, v}(z)\} = \bigcup_{t > 0} B(\exp_z(-t \Gamma_x^z v), t),
			\end{align}
			where $\Gamma_x^z$ denotes the parallel transport from $x$ to $z$ along the geodesic $\gamma$.
			
			We now prove the proposition. For $x,p\in M$, let $\gamma_{x,p}$ be the unit-speed geodesic with $\gamma_{x,p}(0)=x$ and $\gamma_{x,p}(d(x,p))=p$. Define $\xi(x,p)\in M(\infty)$ as $\xi(x,p)=\gamma_{x,p}(\infty)$. We now prove the proposition. By Theorem~\ref{thm:envelope_rep_h-convex}, it suffices to show that the following equality holds for all $y \in \calM$:
			\begin{align} \label{dkbkjdb}
				d(y,p)=\sup_{x\in M}\bigg\{ B_{x,\xi(x,p)}(y)+d(x,p)\bigg\}.
			\end{align}
			
			\emph{Step} 1. We claim that it suffices to show that
			\begin{align} \label{ewwje}
				B_{x,\xi(x,p)}(y)+d(x,p)\leq d(y,p)\textup{ for all }x,y\in M.
			\end{align}
			to prove the equality \eqref{dkbkjdb}. Assume that \eqref{ewwje} is true. Then, we have
			\[
			\sup_{x\in M}\bigg\{ B_{x,\xi(x,p)}(y)+d(x,p)\bigg\}\leq d(y,p).
			\]
			Putting $x=y$ shows that this inequality actually becomes an equality. Thus, \eqref{dkbkjdb} follows from~\eqref{ewwje}.
			
			\emph{Step} 2. We now show~\eqref{ewwje}.  Fix $x$ and $y$. Let $z$ be the point where $\gamma_{x,p}$ intersects the horosphere $\{q \in \calM : B_{x, \xi(x,p)}(q) = B_{x, \xi(x,p)}(y)\}$. Then, we have
			\[
			B_{x,\xi(x,p)}(y)+d(x,p)=B_{x,\xi(x,p)}(z)+d(x,p)=d(z,p),
			\]
			where the last equality follows from the fact that $B_{x,\xi(x,p)}(z)$ is the signed distance between $x$ and $z$, and that the three points $x$, $p$, and $z$ are collinear. Lastly,~\eqref{djdjkdnk} implies that the open geodesic ball $B(p, d(z, p))$ is a subset of the open horoball 
			\begin{align} \label{kdjbkdj}
				\{w\in\calM:B_{x,\xi(x,p)}(w) < B_{x,\xi(x,p)}(y)=B_{x,\xi(x,p)}(z)\}.
			\end{align}
			This implies $d(z, p) \leq d(y, p)$, because $y$ is on the boundary of the horoball~\eqref{kdjbkdj} and so cannot be contained in of $B(p, d(z, p))$.  We conclude that
			$$B_{x, \xi(x,p)}(y) + d(x, p) = d(z, p) \leq d(y, p),$$
			showing the claim \eqref{ewwje}.
			
			\subsection{Proof of Theorem~\ref{thm:gd-c}: Gradient descent for smooth h-convex functions}
			\label{app:thm:gd-c}
			
			Using the h-convexity of each $f_{i}$, we have
			\begin{equation}
				\label{eq:gd-c-ineq1}
				\begin{aligned}
					f(x_{k})-f(x^{*}) & =\frac{1}{m}\sum_{i=1}^{m}\left(f_{i}(x_{k})-f_{i}(x^{*})\right) \leq-\frac{1}{m}\sum_{i=1}^{m}B_{x_{k},\grad f_{i}(x_{k})}(x^{*}).
				\end{aligned}
			\end{equation}
			Using $L$-h-smoothness of each $f_i$ yields
			\begin{align}\label{eq:gd-c-ineq2}
				f(x_{k+1}) \leq f(x_k) + \frac{1}{m}\sum_{i=1}^{m}Q_{x_{k},\grad f_{i}(x_{k})}^L\left(x_{k+1}\right)
				= f(x_k) + \min_{x \in M}\bigg\{\frac{1}{m}\sum_{i=1}^{m}Q_{x_{k},\grad f_{i}(x_{k})}^L\left(x\right)\bigg\},
			\end{align}
			where the last equality comes from the definition of $x_{k+1}$. 
			
			The function $\varphi(x)=\frac{1}{m}\sum_{i=1}^{m}Q_{x_{k},\grad f_{i}(x_{k})}^L\left(x\right)$
			is $L$-strongly g-convex and minimized at $x_{k+1}$, so 
			$$\varphi(x^*) 
			= \frac{1}{m}\sum_{i=1}^{m}Q_{x_{k},\grad f_{i}(x_{k})}^L\left(x^*\right) 
			\geq \frac{L}{2} d(x^*, x_{k+1})^2 + \min_{x \in M}\bigg\{\frac{1}{m}\sum_{i=1}^{m}Q_{x_{k},\grad f_{i}(x_{k})}^L\left(x\right)\bigg\}.$$
			On the other hand Lemma~\ref{lem:horotriangle} implies
			$$\varphi(x^*) 
			= \frac{1}{m}\sum_{i=1}^{m}Q_{x_{k},\grad f_{i}(x_{k})}^L\left(x^*\right) 
			\leq  \frac{1}{m}\sum_{i=1}^{m}B_{x_{k},\grad f_{i}(x_{k})}\left(x^*\right) + \frac{L}{2} d(x^*, x_{k})^2.$$
			Combining these last two inequalities yields
			\begin{align}\label{eq:gd-c-ineq3}
				\frac{1}{m}\sum_{i=1}^{m}B_{x_{k},\grad f_{i}(x_{k})}\left(x^*\right) - \min_{x \in M}\bigg\{\frac{1}{m}\sum_{i=1}^{m}Q_{x_{k},\grad f_{i}(x_{k})}^L\left(x\right)\bigg\}
				\geq
				\frac{L}{2}\Big(d(x^*, x_{k+1})^2 - d(x^*, x_k)^2\Big).
			\end{align}
			
			Defining the energy function
			$\mathcal{E}_{k}=k\left(f(x_{k})-f(x^{*})\right)+\frac{L}{2}d(x_{k},x^{*})^{2},$
			we have
			\begin{align*}
				\mathcal{E}_{k+1}-\mathcal{E}_{k} & =(k+1)\left(f(x_{k+1})-f(x_{k})\right) + (f(x_{k})-f(x^{*}))\\
				& \quad+\frac{L}{2}\left(d(x_{k+1},x^{*})^{2}-d(x_{k},x^{*})^{2}\right).
			\end{align*}
			Using the inequalities \eqref{eq:gd-c-ineq1}, \eqref{eq:gd-c-ineq2}, and \eqref{eq:gd-c-ineq3}, we have
			\begin{align*}
				\mathcal{E}_{k+1}-\mathcal{E}_{k} & \leq(k+1) \min_{x\in M}\bigg\{\sum_{i=1}^{m}Q_{x_{k},\grad f_{i}(x_{k})}^L\left(x\right)\bigg\} -\frac{1}{m}\sum_{i=1}^{m}B_{x_{k},\grad f_{i}(x_{k})}(x^{*})2\\
				& \quad+\frac{1}{m}\sum_{i=1}^{m}B_{x_{k},\grad f_{i}(x_{k})}(x^{*})-\min_{x\in M}\bigg\{\sum_{i=1}^{m}Q_{x_{k},\grad f_{i}(x_{k})}^L\left(x\right)\bigg\}\\
				& =k \min_{x\in M}\bigg\{\sum_{i=1}^{m}Q_{x_{k},\grad f_{i}(x_{k})}^L\left(x\right)\bigg\}\\
				& \leq0,
			\end{align*}
			where the last inequality is because $Q_{x_k, \nabla f_i(x_k)}^L(x_k) = 0$.
			Thus, $\mathcal{E}_{k}$ is non-increasing. Hence, we have
			\[
			f(x_{N})-f(x^{*})\leq\frac{1}{N}\mathcal{E}_{N}\leq\frac{1}{N}\mathcal{E}_{0}=\frac{L}{2N}d(x_{0},x^{*})^{2}.
			\]
			
			\subsection{Proof of Theorem~\ref{thm:gd-sc}: Gradient descent for smooth strongly h-convex functions}
			\label{app:thm:gd-sc}
			
			\subsubsection{Proof of~\eqref{eq:second_result}}
			First, let us prove~\eqref{eq:second_result}.
			For any $x \in M$, strong h-convexity of each $f_i$ implies
			\[
			f(x)=\frac{1}{m}\sum_{i=1}^{m}f_{i}(x)\geq\frac{1}{m}\sum_{i=1}^{m}\left\{ f_{i}(x_{k})+Q_{x_{k},\nabla f_{i}(x_{k})}^{\mu}(x)\right\} .
			\]
			Minimizing both sides of this inequality yields
			$$f^* \geq f(x_k) + \min_{x \in M} \bigg\{\frac{1}{m}\sum_{i=1}^m Q_{x_k, \nabla f_i(x_k)}^{\mu}(x)\bigg\}.$$
			Rearranging yields the following ``Polyak-\L ojasiewicz inequality''\footnote{Indeed, when $M = \mathbb{R}^d$, the right hand side is simply $\frac{1}{2\mu}\|\nabla f(x_k)\|^2$.}
			\begin{align}\label{eq:PL1}
				f(x_k) - f^* \leq - \min_{x \in M} \bigg\{\frac{1}{m}\sum_{i=1}^m Q_{x_k, \nabla f_i(x_k)}^{\mu}(x)\bigg\}.
			\end{align}
			On the other hand, $L$-h-smoothness of each $f_i$ implies
			$$f(x_{x_{k+1}}) \leq f(x_k) + \frac{1}{m}\sum_{i=1}^m Q_{x_k, \nabla f_i(x_k)}^{L}(x_{k+1})
			\stackrel{(1)}{=} f(x_k) + \min_{x \in M}\bigg\{\frac{1}{m}\sum_{i=1}^m Q_{x_k, \nabla f_i(x_k)}^{L}(x)\bigg\}.$$
			The equality $\stackrel{(1)}{=}$ follows from the definition~\eqref{eq:gd} of $x_{k+1}$.
			Rearranging yields
			\begin{align}\label{eq:PL2}
				- \min_{x \in M}\bigg\{\frac{1}{m}\sum_{i=1}^m Q_{x_k, \nabla f_i(x_k)}^{L}(x)\bigg\} \leq f(x_k) - f(x_{k+1}).
			\end{align}
			Lemma~\ref{prop:nonincreasing} implies
			\begin{align}\label{eq:PL3}
				- \frac{\mu}{L}\min_{x \in M} \bigg\{\frac{1}{m}\sum_{i=1}^m Q_{x_k, \nabla f_i(x_k)}^{\mu}(x)\bigg\} \leq - \min_{x \in M}\bigg\{\frac{1}{m}\sum_{i=1}^m Q_{x_k, \nabla f_i(x_k)}^{L}(x)\bigg\},
			\end{align}
			from which we conclude (combining inequalities~\eqref{eq:PL1},~\eqref{eq:PL2},~\eqref{eq:PL3})
			$$\frac{\mu}{L}(f(x_k) - f^*) \leq f(x_k) - f(x_{k+1}).$$
			Rearranging yields $f(x_{k+1}) - f^* \leq (1-\frac{\mu}{L})(f(x_{k}) - f^*)$, which implies the claimed result.
			
			\subsubsection{Proof of~\eqref{eq:first_result}}
			Second, let us prove~\eqref{eq:first_result}.
			Using the $\mu$-strong h-convexity of each $f_{i}$, we have
			\begin{equation}
				\label{eq:gd-sc-ineq1}
				\begin{aligned}
					f(x_{k})-f(x^{*}) & =\frac{1}{m}\sum_{i=1}^{m}\left(f_{i}(x_{k})-f_{i}(x^{*})\right) \leq-\frac{1}{m}\sum_{i=1}^{m}Q_{x_{k},\grad f_{i}(x_{k})}^{\mu}(x^{*}).
				\end{aligned}
			\end{equation}
			Using $L$-h-smoothness of each $f_i$ yields
			\begin{align}\label{eq:gd-sc-ineq2}
				f(x_{k+1}) \leq f(x_k) + \frac{1}{m}\sum_{i=1}^{m}Q_{x_{k},\grad f_{i}(x_{k})}^L\left(x_{k+1}\right)
				= f(x_k) + \min_{x \in M}\bigg\{\frac{1}{m}\sum_{i=1}^{m}Q_{x_{k},\grad f_{i}(x_{k})}^L\left(x\right)\bigg\},
			\end{align}
			where the last equality comes from the definition of $x_{k+1}$. 
			
			The function $\varphi(x)=\frac{1}{m}\sum_{i=1}^{m}Q_{x_{k},\grad f_{i}(x_{k})}^L\left(x\right)$
			is $L$-strongly g-convex and minimized at $x_{k+1}$, so 
			$$\varphi(x^*) 
			= \frac{1}{m}\sum_{i=1}^{m}Q_{x_{k},\grad f_{i}(x_{k})}^L\left(x^*\right) 
			\geq \frac{L}{2} d(x^*, x_{k+1})^2 + \min_{x \in M}\bigg\{\frac{1}{m}\sum_{i=1}^{m}Q_{x_{k},\grad f_{i}(x_{k})}^L\left(x\right)\bigg\}.$$
			On the other hand Lemma~\ref{lem:horotriangle} implies
			$$\varphi(x^*) 
			= \frac{1}{m}\sum_{i=1}^{m}Q_{x_{k},\grad f_{i}(x_{k})}^L\left(x^*\right) 
			\leq  \frac{1}{m}\sum_{i=1}^{m}Q_{x_{k},\grad f_{i}(x_{k})}^\mu\left(x^*\right) + \frac{L-\mu}{2} d(x^*, x_{k})^2.$$
			Combining these last two inequalities yields
			\begin{align}\label{eq:gd-sc-ineq3}
				\frac{1}{m}\sum_{i=1}^{m}Q_{x_{k},\grad f_{i}(x_{k})}^\mu\left(x^*\right) - \min_{x \in M}\bigg\{\frac{1}{m}\sum_{i=1}^{m}Q_{x_{k},\grad f_{i}(x_{k})}^L\left(x\right)\bigg\}
				\geq
				\frac{L}{2}d(x^*, x_{k+1})^2 - \frac{L-\mu}{2}d(x^*, x_k)^2.
			\end{align}
			
			Defining the energy function
			$\mathcal{E}_{k}=\left(1-\mu/L\right)^{-k}\left(f(x_{k})-f(x^{*})+\frac{\mu}{2}d(x_{k},x^{*})^{2}\right),$
			we have
			\begin{align*}
				\left(1-\mu/L\right)^{k+1}\left(\mathcal{E}_{k+1}-\mathcal{E}_{k}\right) & =f\left((x_{k+1})-f(x_{k})\right)+\frac{\mu}{L}\left(f(x_{k})-f(x^{*})\right)\\
				& \quad+\frac{\mu}{2}\left(d(x_{k+1},x^{*})^{2}-\left(1-\frac{\mu}{L}\right)d(x_{k},x^{*})^{2}\right).
			\end{align*}
			Using the inequalities \eqref{eq:gd-sc-ineq1}, \eqref{eq:gd-sc-ineq2}, and \eqref{eq:gd-sc-ineq3}, we have
			\begin{align*}
				\left(1-\mu/L\right)^{k+1}\left(\mathcal{E}_{k+1}-\mathcal{E}_{k}\right) & \leq \min_{x \in M}\bigg\{\frac{1}{m}\sum_{i=1}^{m}Q_{x_{k},\grad f_{i}(x_{k})}^L\left(x\right)\bigg\} -\frac{\mu}{L}\frac{1}{m}\sum_{i=1}^{m}Q_{x_{k},\grad f_{i}(x_{k})}^{\mu}(x^{*})\\
				& \quad+\frac{\mu}{L}\Bigg(\frac{1}{m}\sum_{i=1}^{m}Q_{x_{k},\grad f_{i}(x_{k})}^\mu\left(x^*\right) - \min_{x \in M}\bigg\{\frac{1}{m}\sum_{i=1}^{m}Q_{x_{k},\grad f_{i}(x_{k})}^L\left(x\right)\bigg\}\Bigg)\\
				& = \left(1-\frac{\mu}{L}\right) \min_{x \in M}\bigg\{\frac{1}{m}\sum_{i=1}^{m}Q_{x_{k},\grad f_{i}(x_{k})}^L\left(x\right)\bigg\}\\
				& \leq0,
			\end{align*}
			where the last inequality follows from $Q_{x_{k},\grad f_{i}(x_{k})}^L\left(x_k\right) = 0$. 
			Thus, $\mathcal{E}_{k}$ is non-increasing. Hence, we have
			\begin{align*}
				f(x_{N})-f(x^{*})+\frac{\mu}{2}d(x_N,x^*)^2 & =\left(1-\frac{\mu}{L}\right)^{N}\mathcal{E}_{N}\leq\left(1-\frac{\mu}{L}\right)^{N}\mathcal{E}_{0}\\
				& =\left(1-\frac{\mu}{L}\right)^{N}\left(f(x_{0})-f(x^{*})+\frac{\mu}{2}d(x_{0},x^{*})^{2}\right).
			\end{align*}
			
			\subsection{Proof of Lemma~\ref{prop:nonincreasing}}
			\label{app:prop:nonincreasing}
			
			Define a function $f \colon M \times (0, \infty) \to \reals$ by
			\[
			f(x, L) = \frac{1}{m} \sum_{i=1}^m L^2 d\Big(x, \exp_y\big(-\frac{1}{L} v_{i}\big)\Big)^2,
			\]
			so that $h(L)=\min_{x\in M} f(x,L)$. We also denote $f_L(x)=f(x,L)$. Let $x^*(L) = \argmin_{x \in M} f_L(x)$, which uniquely exists due to the strong g-convexity of $f_L$. Additionally, since $\nabla_x^2 f_L(x) \succ 0$, the implicit function theorem (applied to the equation $\nabla f_L(x)=0$ at pairs $(x^*(L), L)$, $L>0$) implies that $L \mapsto x^*(L)$ is a differentiable map. 
			We then have
			\[
			\frac{d}{dL}h(L)=\frac{d}{dL}\left\{ f(x^{*}(L),L)\right\} =\left\langle \nabla f_{L}(x^{*}(L)),(x^{*})'(L)\right\rangle +\frac{\partial f}{\partial L}(x^{*}(L),L)=\frac{\partial f}{\partial L}(x^{*}(L),L),
			\]
			where we used the chain rule for the second equality, and the fact that $x^*(L)$ minimizes $f_L$ for the last equality.
			
			Make the reparameterization $L = 1/t$ (so $\frac{dt}{dL} = -1/L^{2}$), and define $\gamma_i(t) = \exp_{y}(-t v_i)$.
			We then have (abbreviating $x^* = x^*(L)$)
			\begin{align*}
				\frac{\partial}{\partial L}f(x^{*}(L),L) & =\frac{dt}{dL}\frac{\partial}{\partial t}f(x^{*},1/t)=-\frac{1}{L^{2}}\frac{1}{m}\sum_{i=1}^{m}\frac{\partial}{\partial t}\left\{ \frac{1}{t^{2}}d(x^{*},\gamma_{i}(t))^{2}\right\} \\
				& =-\frac{1}{L^{2}}\frac{1}{m}\sum_{i=1}^{m}\left\{ -\frac{2}{t^{3}}d(x^{*},\gamma_{i}(t))^{2}-\frac{2}{t^{2}}\langle\gamma_{i}'(t),\exp_{\gamma_{i}(t)}^{-1}(x^{*})\rangle\right\} \\
				& =\frac{1}{m}\sum_{i=1}^{m}\left\{ 2Ld(x^{*},\gamma_{i}(t))^{2}+2\langle\gamma_{i}'(t),\exp_{\gamma_{i}(t)}^{-1}(x^{*})\rangle\right\} \\
				& =\frac{1}{m}\sum_{i=1}^{m}\left\{ Ld(x^{*},\gamma_{i}(t))^{2}+\left(Ld(x^{*},\gamma_{i}(t))^{2}-2L\langle\exp_{\gamma_{i}(t)}^{-1}(y),\exp_{\gamma_{i}(t)}^{-1}(x^{*})\rangle\right)\right\}  \\
				& \stackrel{(1)}{\leq}\frac{1}{m}\sum_{i=1}^{m}\left\{ Ld(x^{*},\gamma_{i}(t))^{2}+Ld(x^{*},y)^{2}-Ld(\gamma_{i}(t),y)^{2}\rangle\right\} \\
				& \stackrel{(2)}{\leq}\frac{1}{m}\sum_{i=1}^{m}2L\langle\exp_{x^{*}}^{-1}(\gamma_{i}(t)),\exp_{x^{*}}^{-1}(y)\rangle=\frac{2L}{m}\left\langle \sum_{i=1}^{m}\exp_{x^{*}}^{-1}(\gamma_{i}(t)),\exp_{x^{*}}^{-1}(y)\right\rangle \\
				& \stackrel{(3)}{=}0,
			\end{align*}
			where for $\stackrel{(1)}{\leq}$ we used Lemma~\ref{lem:rev_toponogov} with $p\leftarrow\gamma_{i}(t),x\leftarrow x^{*},y\leftarrow y$, and for $\stackrel{(2)}{\leq}$ we used Lemma~\ref{lem:rev_toponogov} with $p\leftarrow x^{*},x\leftarrow\gamma_{i}(t),y\leftarrow y$. The 
			equality $\stackrel{(3)}{=}$ follows from the fact that $x^*$ minimizes $f_L$:
			$$0 = \nabla f_L(x^*) = \frac{1}{m} \sum_{i=1}^m L^2 \exp^{-1}_{x^*}(\gamma_i(t)).$$
			Overall, we conclude that $\frac{d}{dL}h(L) \leq 0$, and so $h$ is non-increasing as claimed.
			
			\subsection{Proof of Theorem~\ref{thm:subg-descent-c}: Subgradient descent for Lipschitz and h-convex functions}
			\label{app:thm:subg-descent-c}
			
			We use the following result from g-convex geometry.
			\begin{lemma}[{\cite[Thm.~2.2.22]{bacak2014convex}}]
				\label{lem:proj}
				Let $C\subseteq M$ be a closed g-convex set. Then for any $x,y\in M$, we have $d\left(\mathcal{P}_{C}(x),\mathcal{P}_{C}(y)\right)\leq d(x,y)$, where $\mathcal{P}_C$ is the metric projection onto $C$.
			\end{lemma}
			We now prove Theorem~\ref{thm:subg-descent-c}. Using the h-convexity of each $f_i$, we have
			\begin{equation}
				\label{eq:subgd-c-ineq1}
				\begin{aligned}
					f(x_{k})-f(x^{*}) & =\frac{1}{m}\sum_{i=1}^{m}\left(f_{i}(x_{k})-f_{i}(x^{*})\right)
					\leq-\frac{1}{m}\sum_{i=1}^{m}B_{x_{k},g_{ik}}(x^{*}).
				\end{aligned}
			\end{equation}
			The function $\varphi(x)=\frac{1}{m}\sum_{i=1}^{m}Q_{x_{k}, g_{ik}}^{1/s_k}\left(x\right)$
			is $\frac{1}{s_k}$-strongly g-convex and minimized at $x_{k+1}'$, so 
			$$\varphi(x^*) 
			= \frac{1}{m}\sum_{i=1}^{m}Q_{x_{k},g_{ik}}^{1/s_k}\left(x^*\right) 
			\geq \frac{1}{2 s_k} d(x^*, x_{k+1}')^2 + \min_{x \in M}\bigg\{\frac{1}{m}\sum_{i=1}^{m}Q_{x_{k},g_{ik}}^{1/s_k}\left(x\right)\bigg\}.$$
			On the other hand Lemma~\ref{lem:horotriangle} implies
			$$\varphi(x^*) 
			= \frac{1}{m}\sum_{i=1}^{m}Q_{x_{k},g_{ik}}^{1/s_k}\left(x^*\right) 
			\leq  \frac{1}{m}\sum_{i=1}^{m}B_{x_{k},g_{ik}}\left(x^*\right) + \frac{1}{2 s_k} d(x^*, x_{k})^2.$$
			Combining these last two inequalities yields
			\begin{align}\label{eq:subgd-c-ineq2}
				-\frac{1}{m}\sum_{i=1}^{m}B_{x_{k},g_{ik}}\left(x^*\right) 
				\leq
				\frac{1}{2 s_k}\Big(d(x^*, x_k)^2 - d(x^*, x_{k+1}')^2\Big) - \min_{x \in M}\bigg\{\frac{1}{m}\sum_{i=1}^{m}Q_{x_{k},g_{ik}}^{1/s_k}\left(x\right)\bigg\}.
			\end{align}
			Using $\frac{1}{s_k}$-strong g-convexity of $Q_{x_k, g_{ik}}^{1/s_k}$, we have
			\begin{equation}\label{eq:subgd-c-ineq3}
				\begin{split}
					\min_{x \in M}\bigg\{\frac{1}{m}\sum_{i=1}^{m}Q_{x_{k},g_{ik}}^{1/s_k}\left(x\right)\bigg\}
					& \geq\inf_{x\in M}\left\{ \frac{1}{2s_{k}}d\left(x_{k},x\right)^{2}+\frac{1}{m}\sum_{i=1}^{m}\left\langle g_{ik},\exp^{-1}_{x_{k}}(x)\right\rangle \right\} \\
					& =\inf_{x\in M}\left\{ \frac{1}{2s_{k}}\left\Vert \exp^{-1}_{x_{k}}(x)\right\Vert ^{2}+\left\langle g_{k},\exp^{-1}_{x_{k}}(x)\right\rangle \right\} \\
					& =-\frac{s_{k}}{2}\left\Vert g_{k}\right\Vert ^{2} \geq -\frac{L^2 s_{k}}{2},
				\end{split}
			\end{equation}
			where $g_k=\frac{1}{m}\sum_{i=1}^{m} g_{ik}$. 
			The last inequality follows from the $L$-Lipschitz continuity of $f$.

			Combining \eqref{eq:subgd-c-ineq1}, \eqref{eq:subgd-c-ineq2} and \eqref{eq:subgd-c-ineq3}, we have
			\begin{align*}
				f\left(x_{k}\right)-f\left(x^{*}\right) & \leq-\frac{1}{m}\sum_{i=1}^{m}B_{x_{k},g_{ik}}\left(x^{*}\right)\\
				& \leq\frac{1}{2s_{k}}d\left(x_{k},x^{*}\right)^{2}-\frac{1}{2s_{k}}d\left(x_{k+1}',x^{*}\right)^{2}+\frac{L^{2}s_{k}}{2} \\
				& \leq\frac{1}{2s_{k}}d\left(x_{k},x^{*}\right)^{2}-\frac{1}{2s_{k}}d\left(x_{k+1},x^{*}\right)^{2}+\frac{L^{2}s_{k}}{2},
			\end{align*}
			where the final inequality follows from Lemma~\ref{lem:proj}.
			Using $s_{k}=\frac{D}{L\sqrt{N+1}}$ and summing the inequalities over $k$, we have
			\begin{align*}
				\frac{1}{N+1}\sum_{k=0}^{N}\left(f\left(x_{k}\right)-f\left(x^{*}\right)\right) & \leq\frac{L}{2D\sqrt{N+1}}\left(d\left(x_{0},x^{*}\right)^{2}-d\left(x_{N+1},x^{*}\right)^{2}\right)+\frac{DL}{2\sqrt{N+1}}\\
				& \leq\frac{DL}{2\sqrt{N+1}}+\frac{DL}{2\sqrt{N+1}}
				=\frac{DL}{\sqrt{N+1}}.
			\end{align*}
			Using the g-convexity of $f$, we have
			\begin{align*}
				f\left(\bar{x}_{N}\right)-f(x^{*}) & \leq\frac{\sum_{k=0}^{N}f\left(x_{k}\right)}{N+1}-f(x^{*})
				=\frac{1}{N+1}\sum_{k=0}^{N}\left(f(x_{k})-f(x^{*})\right)
				\leq\frac{DL}{\sqrt{N+1}}.
			\end{align*}
			This completes the proof.
			
			\subsection{Proof of Theorem~\ref{thm:subg-descent-sc}: Subgradient descent for Lipschitz and strongly h-convex functions}
			\label{app:thm:subg-descent-sc}
			
			Using the $\mu$-strong h-convexity of each $f_i$, we have
			\begin{equation}
				\label{eq:subgd-sc-ineq1}
				\begin{aligned}
					f(x_{k})-f(x^{*}) & =\frac{1}{m}\sum_{i=1}^{m}\left(f_{i}(x_{k})-f_{i}(x^{*})\right)
					\leq-\frac{1}{m}\sum_{i=1}^{m}Q_{x_{k},g_{ik}}^{\mu}(x^{*}).
				\end{aligned}
			\end{equation}
			The function $\varphi(x)=\frac{1}{m}\sum_{i=1}^{m}Q_{x_{k}, g_{ik}}^{1/s_k}\left(x\right)$
			is $\frac{1}{s_k}$-strongly g-convex and minimized at $x_{k+1}'$, so 
			$$\varphi(x^*) 
			= \frac{1}{m}\sum_{i=1}^{m}Q_{x_{k},g_{ik}}^{1/s_k}\left(x^*\right) 
			\geq \frac{1}{2 s_k} d(x^*, x_{k+1}')^2 + \min_{x \in M}\bigg\{\frac{1}{m}\sum_{i=1}^{m}Q_{x_{k},g_{ik}}^{1/s_k}\left(x\right)\bigg\}.$$
			On the other hand Lemma~\ref{lem:horotriangle} implies
			$$\varphi(x^*) 
			= \frac{1}{m}\sum_{i=1}^{m}Q_{x_{k},g_{ik}}^{1/s_k}\left(x^*\right) 
			\leq  \frac{1}{m}\sum_{i=1}^{m}Q_{x_{k},g_{ik}}^\mu\left(x^*\right) + \Big(\frac{1}{2 s_k} - \frac{\mu}{2}\Big) d(x^*, x_{k})^2.$$
			Combining these last two inequalities yields
			\begin{align*}
				-\frac{1}{m}\sum_{i=1}^{m}Q_{x_{k},g_{ik}}^\mu\left(x^*\right) 
				\leq
				\Big(\frac{1}{2 s_k} -\frac{\mu}{2} \Big)d(x^*, x_k)^2 - \frac{1}{2 s_k} d(x^*, x_{k+1}')^2 - \min_{x \in M}\bigg\{\frac{1}{m}\sum_{i=1}^{m}Q_{x_{k},g_{ik}}^{1/s_k}\left(x\right)\bigg\}.
			\end{align*}
			Using Lemma~\ref{lem:proj} and the bound~\eqref{eq:subgd-c-ineq3} (a consequence of Lipschitz continuity of $f$), we obtain
			\begin{align}\label{eq:subgd-ssc-ineq3}
				-\frac{1}{m}\sum_{i=1}^{m}Q_{x_{k},g_{ik}}^\mu\left(x^*\right) 
				\leq
				\Big(\frac{1}{2 s_k} -\frac{\mu}{2} \Big)d(x^*, x_k)^2 - \frac{1}{2 s_k} d(x^*, x_{k+1})^2 + \frac{L^2 s_k}{2}.
			\end{align}
			Combining \eqref{eq:subgd-sc-ineq1} and \eqref{eq:subgd-ssc-ineq3}, we have
			\[
			f\left(x_{k}\right)-f\left(x^{*}\right)  \leq\left(\frac{1}{2s_{k}}-\frac{\mu}{2}\right)d\left(x_{k},x^{*}\right)^{2}-\frac{1}{2s_{k}}d\left(x_{k+1},x^{*}\right)^{2}+\frac{L^{2}s_{k}}{2}.
			\]
			Using $s_{k}=\frac{2}{\mu(k+2)}$ and taking the weighted sum of the inequalities over $k$ with weight $k+1$,
			\begin{align*}
				\sum_{k=0}^{N}(k+1)\left(f\left(x_{k}\right)-f\left(x^{*}\right)\right) & \leq\sum_{k=0}^{N}(k+1)\left[\frac{\mu k}{4}d\left(x_{k},x^{*}\right)^{2}-\frac{\mu(k+2)}{4}d\left(x_{k+1},x^{*}\right)^{2}+\frac{L^{2}}{\mu(k+2)}\right]\\
				& =-\frac{\mu(N+1)(N+2)}{4}d\left(x_{N+1},x^{*}\right)^{2}+\sum_{k=0}^{N}\frac{L^{2}(k+1)}{\mu(k+2)}
				\leq\frac{(N+1)L^{2}}{\mu}.
			\end{align*}
			Using the g-convexity of $f$, we have
			\begin{align*}
				f\left(\bar{x}_{N}\right)-f(x^{*}) & \leq\frac{2\sum_{k=0}^{N}(k+1)f\left(x_{k}\right)}{(N+1)(N+2)}-f(x^{*})\\
				& =\frac{2}{(N+1)(N+2)}\sum_{k=0}^{N}(k+1)\left(f(x_{k})-f(x^{*})\right)\\
				& \leq\frac{2}{(N+1)(N+2)}\frac{(N+1)L^{2}}{\mu}
				=\frac{2L^{2}}{\mu(N+2)}.
			\end{align*}
			This completes the proof.
			
			\subsection{Proof of Theorem~\ref{thm:agm-c}: Nesterov acceleration for smooth h-convex functions}
			\label{app:thm:agm-c}
			
			Using the h-convexity of each $f_{i}$, we have
			\begin{equation}
				\label{eq:agm-c-ineq1}
				\begin{aligned}
					f(y_{k})-f(x^{*}) & =\frac{1}{m}\sum_{i=1}^{m}\left(f_{i}(y_{k})-f_{i}(x^{*})\right)
					\leq-\frac{1}{m}\sum_{i=1}^{m}B_{y_{k},\grad f_{i}(y_{k})}(x^{*}).
				\end{aligned}
			\end{equation}
			By the g-convexity of $f$, we have
			\begin{equation}
				\label{eq:agm-c-ineq-g}
				f(y_{k})-f(x_{k})\leq-\left\langle \grad f(y_{k}),\exp^{-1}_{y_{k}}(x_{k})\right\rangle .
			\end{equation}
			Define the function $\varphi$ on $M$ as
			\[
			\varphi(x)=\frac{1}{2}d\left(z_{k},x\right)^{2}+\frac{k+1}{2Lm}\sum_{i=1}^{m}B_{y_{k},\grad f_{i}(y_{k})}\left(x\right).
			\]
			Then, we have
			\begin{align*}
				\inf_{x\in M}\varphi_{k}(x) & =\inf_{x\in M}\left\{ \frac{1}{2}d\left(z_{k},x\right)^{2}+\frac{k+1}{2Lm}\sum_{i=1}^{m}B_{y_{k},\grad f_{i}(y_{k})}\left(x\right),\right\} \\
				& \geq\inf_{x\in M}\left\{ \frac{1}{2}d\left(z_{k},x\right)^{2}+\frac{k+1}{2L}\left\langle \grad f(y_{k}),\exp^{-1}_{y_{k}}(x)\right\rangle \right\} \\
				& \geq\inf_{x\in M}\left\{ \frac{1}{2}\left\Vert \exp^{-1}_{y_{k}}(z_{k})-\exp^{-1}_{y_{k}}(x)\right\Vert ^{2}+\frac{k+1}{2L}\left\langle \grad f(y_{k}),\exp^{-1}_{y_{k}}(x)\right\rangle \right\} \\
				& =\frac{k+1}{2L}\left\langle \grad f(y_{k}),\exp^{-1}_{y_{k}}(z_{k})\right\rangle -\frac{(k+1)^{2}}{8L^{2}}\left\Vert \grad f(y_{k})\right\Vert ^{2},
			\end{align*}
			where the first inequality follows from the g-convexity of each $B_{y_{k},\grad f_{i}(y_{k})}^{\mu}$, and the second inequality follows from Lemma~\ref{lem:rev_toponogov}. Since $\varphi$ is $1$-strongly g-convex on $M$ and minimized at $z_{k+1}$, we have
			\[
			\varphi(x)\geq\frac{1}{2}d\left(z_{k+1},x\right)^{2}+\frac{k+1}{2L}\left\langle \grad f(y_{k}),\exp^{-1}_{y_{k}}(z_{k})\right\rangle -\frac{(k+1)^{2}}{8L^{2}}\left\Vert \grad f(y_{k})\right\Vert ^{2}.
			\]
			In particular, putting $x=x^*$ yields
			\begin{equation}
				\label{eq:agm-c-ineq2}
				\begin{aligned}
					& \frac{1}{2}d\left(z_{k},x^{*}\right)^{2}+\frac{k+1}{2Lm}\sum_{i=1}^{m}B_{y_{k},\grad f_{i}(y_{k})}\left(x^{*}\right)\\
					& \geq\frac{1}{2}d\left(z_{k+1},x^{*}\right)^{2}+\frac{k+1}{2L}\left\langle \grad f(y_{k}),\exp^{-1}_{y_{k}}(z_{k})\right\rangle -\frac{(k+1)^{2}}{8L^{2}}\left\Vert \grad f(y_{k})\right\Vert ^{2}.
				\end{aligned}
			\end{equation}
			By assumption on $f$, we have
			\begin{equation}
				\label{eq:agm-c-ineq3}
				f(x_{k+1})\leq f(y_{k})-\frac{1}{2L}\left\Vert \grad f(y_{k})\right\Vert ^{2}.
			\end{equation}
			Consider the energy function
			\[
			\mathcal{E}_{k}=\frac{1}{2}d(x^{*},z_{k})^{2}+\frac{k^{2}}{4L}(f(x_{k})-f(x^{*})).
			\]
			Then, we have
			\begin{align*}
				\mathcal{E}_{k+1}-\mathcal{E}_{k} & \leq\frac{(k+1)^{2}}{4L}(f(x_{k+1})-f(x^{*}))-\left(1-\frac{2}{k+1}\right)\frac{(k+1)^{2}}{4L}(f(x_{k})-f(x^{*}))\\
				& \quad+\frac{1}{2}d(x^{*},z_{k+1})^{2}-\frac{1}{2}d(x^{*},z_{k})^{2}\\
				& =\frac{2}{k+1}\frac{(k+1)^{2}}{4L}(f(y_{k})-f(x^{*}))\\
				& \quad+\left(1-\frac{2}{k+1}\right)\frac{(k+1)^{2}}{4L}(f(y_{k})-f(x_{k}))\\
				& \quad+\frac{(k+1)^{2}}{4L}(f(x_{k+1})-f(y_{k}))\\
				& \quad+\frac{1}{2}d(x^{*},z_{k+1})^{2}-\frac{1}{2}d(x^{*},z_{k})^{2}.
			\end{align*}
			Using the inequalities \eqref{eq:agm-sc-ineq1}, \eqref{eq:agm-sc-ineq-g}, \eqref{eq:agm-sc-ineq2}, and \eqref{eq:agm-sc-ineq3}, we have
			\begin{align*}
				\mathcal{E}_{k+1}-\mathcal{E}_{k} & \leq-\frac{2}{k+1}\frac{(k+1)^{2}}{4L}\frac{1}{m}\sum_{i=1}^{m}B_{y_{k},\grad f_{i}(y_{k})}(x^{*})\\
				& \quad-\left(1-\frac{2}{k+1}\right)\frac{(k+1)^{2}}{4L}\left\langle \grad f(y_{k}),\exp^{-1}_{y_{k}}(x_{k})\right\rangle \\
				& \quad-\frac{(k+1)^{2}}{8L^{2}}\left\Vert \grad f(y_{k})\right\Vert ^{2}\\
				& \quad-\frac{k+1}{2Lm}\sum_{i=1}^{m}B_{y_{k},\grad f_{i}(y_{k})}\left(x^{*}\right)-\frac{k+1}{2L}\left\langle \grad f(y_{k}),\exp^{-1}_{y_{k}}(z_{k})\right\rangle \\
				& \quad +\frac{(k+1)^{2}}{8L^{2}}\left\Vert \grad f(y_{k})\right\Vert ^{2}\\
				& =-\frac{(k+1)^{2}}{4L}\left\langle \grad f(y_{k}),\left(1-\frac{2}{k+1}\right)\exp^{-1}_{y_{k}}(x_{k})+\frac{2}{k+1}\exp^{-1}_{y_{k}}(z_{k})\right\rangle \\
				& =0,
			\end{align*}
			where the last equality follows from $y_{k}=\exp_{x_{k}}\left(\frac{2}{k+1}\exp^{-1}_{x_{k}}(z_{k})\right)$. Thus, $\mathcal{E}_{k}$ is non-increasing. Hence, we have
			\[
			f(x_{N})-f(x^{*})\leq\frac{4L}{N^{2}}\mathcal{E}_{N}\leq\frac{4L}{N^{2}}\mathcal{E}_{0}=\frac{2L}{N^{2}}d(x_{0},x^{*})^{2}.
			\]
			
			\subsection{Proof of Theorem~\ref{thm:agm-sc}: Nesterov acceleration for smooth strongly h-convex functions}
			\label{app:thm:agm-sc}
			
			Using the $\mu$-strong h-convexity of each $f_{i}$, we have
			\begin{equation}
				\label{eq:agm-sc-ineq1}
				\begin{aligned}
					f(y_{k})-f(x^{*}) & =\frac{1}{m}\sum_{i=1}^{m}\left(f_{i}(y_{k})-f_{i}(x^{*})\right)
					\leq-\frac{1}{m}\sum_{i=1}^{m}Q_{y_{k},\grad f_{i}(y_{k})}^{\mu}(x^{*}).
				\end{aligned}
			\end{equation}
			By the g-convexity of $f$, we have
			\begin{equation}
				\label{eq:agm-sc-ineq-g}
				f(y_{k})-f(x_{k})\leq-\left\langle \grad f(y_{k}),\exp^{-1}_{y_{k}}(x_{k})\right\rangle .
			\end{equation}
			Define a function $\varphi$ on $M$ as
			\[
			\varphi(x)=\left(1-\sqrt{\frac{\mu}{L}}\right)\frac{\mu}{2}d\left(z_{k},x\right)^{2}+\sqrt{\frac{\mu}{L}}\frac{1}{m}\sum_{i=1}^{m}Q_{y_{k},\grad f_{i}(y_{k})}^{\mu}\left(x\right).
			\]
			Then, we have
			\begin{align*}
				\inf_{x\in M}\varphi_{k}(x) & =\inf_{x\in M}\left\{ \left(1-\sqrt{\frac{\mu}{L}}\right)\frac{\mu}{2}d\left(z_{k},x\right)^{2}+\sqrt{\frac{\mu}{L}}\frac{1}{m}\sum_{i=1}^{m}Q_{y_{k},\grad f_{i}(y_{k})}^{\mu}\left(x\right).\right\} \\
				& \geq\inf_{x\in M}\left\{ \left(1-\sqrt{\frac{\mu}{L}}\right)\frac{\mu}{2}d\left(z_{k},x\right)^{2}+\sqrt{\frac{\mu}{L}}\left(\left\langle \grad f(y_{k}),\exp^{-1}_{y_{k}}(x)\right\rangle +\frac{\mu}{2}d\left(y_{k},x\right)^{2}\right)\right\} \\
				& \geq\inf_{x\in M}\Bigg\{\left(1-\sqrt{\frac{\mu}{L}}\right)\frac{\mu}{2}\left\Vert \exp^{-1}_{y_{k}}(z_{k})-\exp^{-1}_{y_{k}}(x)\right\Vert ^{2}\\
				& \qquad+\sqrt{\frac{\mu}{L}}\left(\left\langle \grad f(y_{k}),\exp^{-1}_{y_{k}}(x)\right\rangle +\frac{\mu}{2}\left\Vert \exp^{-1}_{y_{k}}(x)\right\Vert ^{2}\right)\Bigg\}\\
				& =\inf_{x\in M}\Bigg\{\frac{\mu}{2}\left\Vert \exp^{-1}_{y_{k}}(x)\right\Vert ^{2}+\left\langle \sqrt{\frac{\mu}{L}}\grad f(y_{k})-\mu\left(1-\sqrt{\frac{\mu}{L}}\right)\exp^{-1}_{y_{k}}(z_{k}),\exp^{-1}_{y_{k}}(x)\right\rangle \\
				& \qquad+\left(1-\sqrt{\frac{\mu}{L}}\right)\frac{\mu}{2}\left\Vert \exp^{-1}_{y_{k}}(z_{k})\right\Vert ^{2}\Bigg\}\\
				& =-\frac{\mu}{2}\left\Vert \frac{1}{\mu}\sqrt{\frac{\mu}{L}}\grad f(y_{k})-\left(1-\sqrt{\frac{\mu}{L}}\right)\exp^{-1}_{y_{k}}(z_{k})\right\Vert ^{2}+\left(1-\sqrt{\frac{\mu}{L}}\right)\frac{\mu}{2}\left\Vert \exp^{-1}_{y_{k}}(z_{k})\right\Vert ^{2}\\
				& =\frac{\mu}{2}\sqrt{\frac{\mu}{L}}\left(1-\sqrt{\frac{\mu}{L}}\right)\left\Vert \exp^{-1}_{y_{k}}(z_{k})\right\Vert ^{2}+\sqrt{\frac{\mu}{L}}\left(1-\sqrt{\frac{\mu}{L}}\right)\left\langle \exp^{-1}_{y_{k}}(z_{k}),\grad f(y_{k})\right\rangle \\
				& \quad-\frac{1}{2L}\left\Vert \grad f(y_{k})\right\Vert ^{2},
			\end{align*}
			where the first inequality follows from the $\mu$-strong g-convexity of each $Q_{y_{k},\grad f_{i}(y_{k})}^{\mu}$, and the second inequality follows from Lemma~\ref{lem:rev_toponogov}. Since $\varphi$ is $\mu$-strongly g-convex on $M$ and minimized at $z_{k+1}$, we have
			\begin{align*}
				\varphi(x) & \geq\frac{\mu}{2}d\left(z_{k+1},x\right)^{2}+\frac{\mu}{2}\sqrt{\frac{\mu}{L}}\left(1-\sqrt{\frac{\mu}{L}}\right)\left\Vert \exp^{-1}_{y_{k}}(z_{k})\right\Vert ^{2}\\
				& \quad+\sqrt{\frac{\mu}{L}}\left(1-\sqrt{\frac{\mu}{L}}\right)\left\langle \exp^{-1}_{y_{k}}(z_{k}),\grad f(y_{k})\right\rangle -\frac{1}{2L}\left\Vert \grad f(y_{k})\right\Vert ^{2}.
			\end{align*}
			In particular, putting $x=x^*$ yields
			\begin{equation}
				\label{eq:agm-sc-ineq2}
				\begin{aligned}
					& \left(1-\sqrt{\frac{\mu}{L}}\right)\frac{\mu}{2}d\left(z_{k},x^{*}\right)^{2}+\sqrt{\frac{\mu}{L}}\frac{1}{m}\sum_{i=1}^{m}Q_{y_{k},\grad f_{i}(y_{k})}^{\mu}\left(x^{*}\right)\\
					& \geq\frac{\mu}{2}d\left(z_{k+1},x^{*}\right)^{2}+\frac{\mu}{2}\sqrt{\frac{\mu}{L}}\left(1-\sqrt{\frac{\mu}{L}}\right)\left\Vert \exp^{-1}_{y_{k}}(z_{k})\right\Vert ^{2}\\
					& \quad+\sqrt{\frac{\mu}{L}}\left(1-\sqrt{\frac{\mu}{L}}\right)\left\langle \exp^{-1}_{y_{k}}(z_{k}),\grad f(y_{k})\right\rangle -\frac{1}{2L}\left\Vert \grad f(y_{k})\right\Vert ^{2}.
				\end{aligned}
			\end{equation}
			By assumption on $f$, we have
			\begin{equation}
				\label{eq:agm-sc-ineq3}
				f(x_{k+1})\leq f(y_{k})-\frac{1}{2L}\left\Vert \grad f(y_{k})\right\Vert ^{2}.
			\end{equation}
			Consider the energy function
			\[
			\mathcal{E}_{k}=\left(1-\sqrt{\frac{\mu}{L}}\right)^{-k}\left(f(x_{k})-f(x^{*})+\frac{\mu}{2}d(z_{k},x^{*})^{2}\right).
			\]
			Then, we have
			\begin{align*}
				\left(1-\sqrt{\frac{\mu}{L}}\right)^{k+1}\left(\mathcal{E}_{k+1}-\mathcal{E}_{k}\right) & =(f(x_{k+1})-f(x^{*}))-\left(1-\sqrt{\frac{\mu}{L}}\right)(f(x_{k})-f(x^{*}))\\
				& \quad+\frac{\mu}{2}d\left(z_{k+1},x^{*}\right)^{2}-\left(1-\sqrt{\frac{\mu}{L}}\right)\frac{\mu}{2}d\left(z_{k},x^{*}\right)^{2}\\
				& =\sqrt{\frac{\mu}{L}}(f(y_{k})-f(x^{*}))\\
				& \quad+\left(1-\sqrt{\frac{\mu}{L}}\right)(f(y_{k})-f(x_{k}))\\
				& \quad+f(x_{k+1})-f(y_{k})\\
				& \quad+\frac{\mu}{2}d\left(z_{k+1},x^{*}\right)^{2}-\left(1-\sqrt{\frac{\mu}{L}}\right)\frac{\mu}{2}d\left(z_{k},x^{*}\right)^{2}.
			\end{align*}
			Using the inequalities \eqref{eq:agm-sc-ineq1}, \eqref{eq:agm-sc-ineq-g}, \eqref{eq:agm-sc-ineq2}, and \eqref{eq:agm-sc-ineq3}, we have
			\begin{align*}
				\left(1-\sqrt{\frac{\mu}{L}}\right)^{k+1}\left(\mathcal{E}_{k+1}-\mathcal{E}_{k}\right) & \leq-\sqrt{\frac{\mu}{L}}\frac{1}{m}\sum_{i=1}^{m}Q_{y_{k},\grad f_{i}(y_{k})}^{\mu}(x^{*})\\
				& \quad-\left(1-\sqrt{\frac{\mu}{L}}\right)\left\langle \grad f(y_{k}),\exp^{-1}_{y_{k}}(x_{k})\right\rangle \\
				& \quad-\frac{1}{2L}\left\Vert \grad f(y_{k})\right\Vert ^{2}\\
				& \quad+\sqrt{\frac{\mu}{L}}\frac{1}{m}\sum_{i=1}^{m}Q_{y_{k},\grad f_{i}(y_{k})}^{\mu}\left(x^{*}\right)-\frac{\mu}{2}\sqrt{\frac{\mu}{L}}\left(1-\sqrt{\frac{\mu}{L}}\right)\left\Vert \exp^{-1}_{y_{k}}(z_{k})\right\Vert ^{2}\\
				& \quad-\sqrt{\frac{\mu}{L}}\left(1-\sqrt{\frac{\mu}{L}}\right)\left\langle \exp^{-1}_{y_{k}}(z_{k}),\grad f(y_{k})\right\rangle +\frac{1}{2L}\left\Vert \grad f(y_{k})\right\Vert ^{2}\\
				& =-\left(1-\sqrt{\frac{\mu}{L}}\right)\left\langle \grad f(y_{k}),\exp^{-1}_{y_{k}}(x_{k})+\sqrt{\frac{\mu}{L}}\exp^{-1}_{y_{k}}(z_{k})\right\rangle \\
				& \quad-\frac{\mu}{2}\sqrt{\frac{\mu}{L}}\left(1-\sqrt{\frac{\mu}{L}}\right)\left\Vert \exp^{-1}_{y_{k}}(z_{k})\right\Vert ^{2}\\
				& \leq-\left(1-\sqrt{\frac{\mu}{L}}\right)\left\langle \grad f(y_{k}),\exp^{-1}_{y_{k}}(x_{k})+\sqrt{\frac{\mu}{L}}\exp^{-1}_{y_{k}}(z_{k})\right\rangle \\
				& =0,
			\end{align*}
			where the last equality follows from $y_{k}=\exp_{x_{k}}\left(\frac{\sqrt{\mu/L}}{1+\sqrt{\mu/L}}\exp^{-1}_{x_{k}}(z_{k})\right)$. Thus, $\mathcal{E}_{k}$ is non-increasing. Hence, we have
			\[
			f(x_{N})-f(x^{*})\leq\left(1-\sqrt{\frac{\mu}{L}}\right)^{N}\mathcal{E}_{N}\leq\left(1-\sqrt{\frac{\mu}{L}}\right)^{N}\mathcal{E}_{0}=\left(1-\sqrt{\frac{\mu}{L}}\right)^{N}\left(f(x_{0})-f(x^{*})+\frac{\mu}{2}d(x_{0},x^{*})^{2}\right).
			\]
			
			\section{Faster rates in hyperbolic space: Details for Section~\ref{sec:faster_rates}}
			
			\subsection{Proof of Proposition~\ref{prop:thisone}}\label{app:faster-hyper2}
			
			We prove the contrapositive, so
			assume that $f(x_k) - f^* \geq L \delta$ for all $k = 0, \ldots N-1$.
			Then the subgradients $g_k$ are non-zero, and h-convexity of $f$ implies 
			\begin{align*}
				-L \delta \geq f^* - f(x_k) \geq B_{x_k, g_k}(x^*)
				= \|g_k\| B_{x_k, g_k / \|g_k\|}(x^*)
				\stackrel{(1)}{\geq} L B_{x_k, g_k / \|g_k\|}(x^*)
			\end{align*}
			for all $k = 0, \ldots, N$.
			Inequality $\stackrel{(1)}{\geq}$ follows from $\|g_k\| \leq L$ and $B_{x,g_{k}/\|g_{k}\|}(x^{*})\leq0$.
			Therefore, $x^*$ is contained in the intersection of the ball $B_k = \bar{B}(x_k, d(x_k, x^*))$ and the horoball
			$$\{x \in M : - \delta \geq B_{x_k, g_k/\|g_k\|}(x)\},$$
			which is the horoball associated to the geodesic ray $t \mapsto \exp_{x_k}(-t g_k)$ whose boundary contains
			$x_{k+1} = \exp_{x_k}(-\delta g_k/\|g_k\|)$.
			That horoball is contained in the supporting half-space
			$${H}_k = \{x \in M : \langle \exp^{-1}_{x_{k+1}}(x), \Gamma_{x_k}^{x_{k+1}} g_k \rangle \leq 0\}.$$
			
			Let $q \in \bd(B_k) \cap \bd(H_k)$ and let $r_{k+1} = d(x_{k+1}, q)$ --- $r_{k+1}$ is, of course, independent of the choice of $q$ due to the symmetries of hyperbolic space.\footnote{$\bd(B_k) \cap \bd(H_k)$ is non-empty by the intermediate value theorem: choose some curve belonging to $\bd(B_k)$ from $\exp_{x_{k}}(\frac{d(x_{k},x^{*})}{d(x_{k},x_{k+1})}\exp_{x_{k}}^{-1}(x_{k+1}))$ to $\exp_{x_{k}}(-\frac{d(x_{k},x^{*})}{d(x_{k},x_{k+1})}\exp_{x_{k}}^{-1}(x_{k+1}))$.  Then the function $\langle\exp_{x_{k+1}}^{-1}(\cdot),\Gamma_{x_{k}}^{x_{k+1}}g_{k}\rangle$ takes value $0$ somewhere along that curve.}
			Since $q x_k x_{k+1}$ is a hyperbolic triangle with the right angle at $x_{k+1}$, the hyperbolic law of cosines gives 
			$$\cosh(r_{k+1})=\frac{\cosh(d(x_{k},q))}{\cosh(\delta)}=\frac{\cosh(d(x_{k},x^{*}))}{\cosh(\delta)},$$ 
			where the second equality follows from $q\in\bd (B_{k}) = \bd\big(\bar{B}(x_k, d(x_k, x^*))\big)$.
			
			Note that $\bar{B}(x_{k+1}, r_{k+1})$ contains $B_k \cap H_k \ni x^*$ by construction --- this is easy to see by visualizing in the Poincar\'e ball model centered at $x_{k+1}$. 
			Therefore,
			\[
			\cosh(d(x_{k+1},x^{*}))\leq\cosh(r_{k+1})=\frac{\cosh(d(x_{k},x^{*}))}{\cosh(\delta)}.
			\]
			Using $d(x_0, x^*) = d(p, x^*) \leq r$, we conclude
			$1 \leq \cosh(d(x_{N}, x^*)) \leq \cosh(r) / \cosh(\delta)^N$.
			Therefore, we must have $N \leq \frac{\log \cosh(r)}{\log \cosh(\delta)}$, as desired.
			
			For the final observation in the proposition statement, we have
			\[
			\frac{\log\cosh(r)}{\log\cosh(\delta)}\leq\frac{\log(e^{r})}{(\delta/4)\min\{1,\delta\}}=\frac{4r}{\delta}\max\{1,\frac{1}{\delta}\},
			\]
			using $e^t \geq \cosh(t)$ and $\log \cosh(t) \geq \frac{t}{4} \min\{1, t\}$ for $t \geq 0$.
			
			\subsection{Proof of Proposition~\ref{prop:anotherone}}\label{app:faster-hyper3}
			
			\begin{SCfigure}[50][t]
				\label{fig:twohoroballs}
				\includegraphics[width=0.6\textwidth]{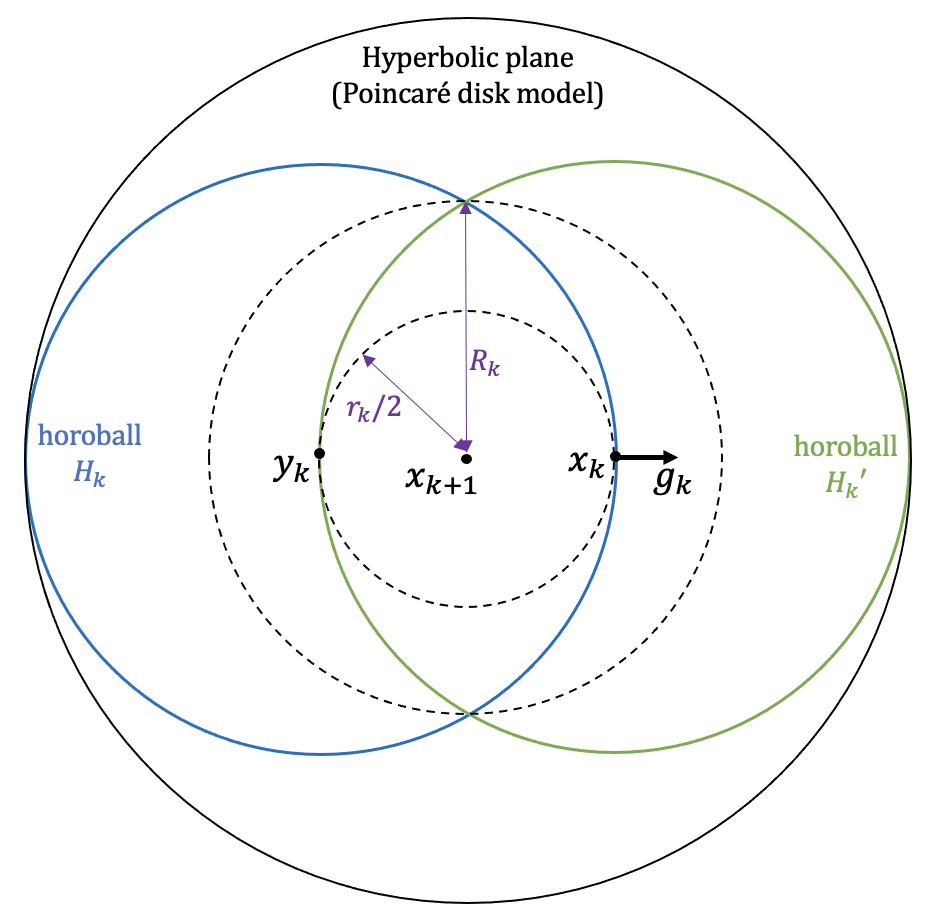}
				\caption{Bounding the intersection of two horoballs with a ball, in the Poincar\'e disk model.}
			\end{SCfigure}
			
			Let $r_k = r e^{-k/4}$.
			Let $x^* \in \arg\min_{x\in \bar{B}(p, r)}f(x)$.
			We shall show, by induction, that if $r_{k} \geq 4$, then $d(x_{k+1}, x^*) \leq r_{k+1} = r_k e^{-1/4}$. 
			Consequently, after $N=\lceil4\log(r/4)\rceil$ iterations, we must have $d(x_N, x^*) \leq 4$, which would prove the claim. 
			
			The base case $k=0$ follows immediately from $x_0 = p$ and $x^* \in \bar{B}(p, r)$.
			Now let us prove the inductive step.
			Assume $d(x_k, x^*) \leq r_k$.
			Consider the horoball $H_k'$ associated with the geodesic ray $t \mapsto \exp_{x_k}(t g_k)$ and whose boundary contains $y_k = \exp_{x_k}(- r_k \frac{g_k}{\|g_k\|})$.
			Since balls are h-convex, this horoball $H_k'$ contains the ball $\bar{B}(x_k, r_k)$, and thus $x^*$ too.
			
			On the other hand, $x^*$ is contained in the horoball $H_k = \{x \in M : B_{x_k, g_k}(x) \leq 0\}$, because $f$ is h-convex:
			$$0 \geq f^* - f(x_k) \geq B_{x_k, g_k}(x^*).$$
			Let $x_{k+1} = \exp_{x_k}(- \frac{r_k}{2} \frac{g_k}{\|g_k\|})$.
			Figure~\ref{fig:twohoroballs} shows the setup in the Poincar\'e disk model where the origin is placed at $x_{k+1}$.
			From that figure and a simple calculation,\footnote{Specifically: the \emph{Euclidean} radii of the inner and outer balls in the disk model are $\tanh((r_k/2)/2)$ and $\tanh(R_k/2)$, respectively.  Euclidean geometry then tells us that $\tanh(R_k/2) = \sqrt{\tanh((r_k/2)/2)}$.  Solving for $R_k$ yields $R_k = \arccosh(e^{r_k/2})$.}
			we see that $H_k \cap H_k'$ is contained in a ball of radius $R_k = \arccosh(e^{r_k/2}) \leq \frac{r_k}{2} + 1$ centered at $x_{k+1}$. 
			Since $x^* \in H_k \cap H_k'$, we find
			$$d(x_{k+1}, x^*) \leq \frac{r_k}{2} + 1 \leq e^{-1/4} r_k = r_{k+1},$$
			where the last inequality follows from assuming $r_k \geq 4$.
			
		\end{document}